\documentclass[11pt,a4paper,reqno]{amsart}
\usepackage{graphicx}
\usepackage[linktocpage=true,colorlinks,citecolor=blue,linkcolor=blue,urlcolor=blue]{hyperref}
\usepackage{amssymb}
\usepackage{cite}
\usepackage{amsmath}
\usepackage{latexsym}
\usepackage{amscd}
\usepackage{amsthm}
\usepackage{mathrsfs}
\usepackage{url}
\newtheorem{thm}{Theorem}[section]
\newtheorem{corr}[thm]{Corollary}
\newtheorem{lem}[thm]{Lemma}
\newtheorem{prop}[thm]{Proposition}

\newtheorem{claim}[thm]{Claim}
\theoremstyle{definition}

\newtheorem{rem}[thm]{Remark}

\numberwithin{equation}{section}
\def\R{\mathbb R}

\def\d{\mathrm d}

\def\ra{\rightarrow}
\def\pt{\partial}
\DeclareMathOperator\GG{G}
\DeclareMathOperator\GL{GL}
\DeclareMathOperator\SO{SO}

\DeclareMathOperator\Div{div}
\DeclareMathOperator\curl{curl}
\DeclareMathOperator\vol{vol}

\begin{document}
\title[Laplacian flow for closed \texorpdfstring{$\GG_2$}{G2} structures]{Laplacian flow for closed \texorpdfstring{{\boldmath $\GG_2$}}{G2} structures: Shi-type estimates, uniqueness and compactness}
\author{Jason D. Lotay}
\author{Yong Wei}
\address{Department of Mathematics, University College London, Gower Street, London, WC1E 6BT, United Kingdom}
\email{j.lotay@ucl.ac.uk; yong.wei@ucl.ac.uk}
\subjclass[2010]{{53C44}, {53C25}, {53C10}}
\keywords{Laplacian flow, $\GG_2$ structure, Shi-type estimates, uniqueness, compactness, soliton}
\thanks{This research was supported by EPSRC grant EP/K010980/1.}

%
\begin{abstract}
We develop foundational theory for the Laplacian flow for closed $\GG_2$ structures which will be essential for future study.
(1).~We prove Shi-type derivative estimates for the Riemann curvature tensor $Rm$ and torsion tensor $T$ along the flow, i.e.~that a bound on
 \begin{equation*}
  \Lambda(x,t)=\left(|\nabla T(x,t)|_{g(t)}^2+|Rm(x,t)|_{g(t)}^2\right)^{\frac 12}
\end{equation*}
will imply bounds on all covariant derivatives of $Rm$ and $T$.
(2).~We show that $\Lambda(x,t)$ will blow up at a finite-time singularity, so the flow will exist as long as $\Lambda(x,t)$ remains bounded.
(3).~We give a new proof of forward uniqueness and prove backward uniqueness of the flow, and give some applications. (4).~We prove a compactness theorem for the flow  and use it to strengthen our long time existence result from (2)
 to show that the flow will exist as long as the velocity of the flow remains bounded.
(5).~Finally, we study  soliton solutions of the Laplacian flow.
\end{abstract}
\maketitle
\tableofcontents
\section{Introduction}

In this article we analyse the Laplacian flow for closed $\GG_2$ structures, which provides a potential tool for studying the challenging problem of existence of
torsion-free $\GG_2$ structures, and thus Ricci-flat metrics with exceptional holonomy $\GG_2$, on a 7-dimensional manifold.  We develop foundational results for the flow,
both in terms of analytic and geometric aspects.

\subsection{Basic theory} Let $M$ be a $7$-manifold. A $\GG_2$ structure on $M$ is defined by a $3$-form $\varphi$ on $M$ satisfying a certain nondegeneracy
 condition. To any such $\varphi$, one associates a unique metric $g$ and orientation on $M$, and thus a Hodge star operator $*_{\varphi}$. If $\nabla$ is the
 Levi-Civita connection of $g$, we  interpret $\nabla\varphi$ as the torsion of the $\GG_2$ structure $\varphi$. Thus, if $\nabla\varphi=0$, which is equivalent
 to $d\varphi=d\!*_{\varphi}\!\varphi=0$, we say $\varphi$ is torsion-free and $(M,\varphi)$ is a $\GG_2$ manifold.

The key property of torsion-free $\GG_2$ structures is that the holonomy group of the associated metric satisfies $\textrm{Hol}(g)\subset \GG_2$, and
hence $(M,g)$ is Ricci-flat. If $(M,\varphi)$ is a compact $\GG_2$ manifold, then $\textrm{Hol}(g)= \GG_2$ if and only if $\pi_1(M)$ is finite, and thus finding torsion-free $\GG_2$ structures is essential for constructing compact manifolds with holonomy $\GG_2$.  Notice that the torsion-free condition is a nonlinear PDE on $\varphi$, since $*_{\varphi}$ depends on $\varphi$, and thus finding torsion-free $\GG_2$ structures is a challenging problem.

Bryant \cite{bryant1987} used the theory of exterior differential systems to first prove the local existence of holonomy $\GG_2$ metrics.  This was soon followed by the first explicit complete holonomy $\GG_2$ manifolds in work of Bryant--Salamon \cite{bryant-salamon}. In ground-breaking work, Joyce
\cite{joyce96-1} developed a fundamental existence theory for torsion-free $\GG_2$ structures by perturbing \emph{closed} $\GG_2$ structures with ``small''
 torsion which, together with a gluing method, led to the first examples of compact $7$-manifolds with holonomy $\GG_2$.  This theory has formed
 the cornerstone of the programme for constructing compact holonomy $\GG_2$ manifolds, of which there are now many examples (see \cite{CHNP, Kov}).

Although the existence theory of Joyce is powerful, it is a perturbative result and one has to work hard to find suitable initial
data for the theory.  In all known examples such data is always close to ``degenerate'', arising from a gluing procedure, and thus gives little
sense of the general problem of existence of torsion-free $\GG_2$ structures.  In fact, aside from some basic topological constraints, we
have a primitive understanding of when a given compact 7-manifold could admit a torsion-free $\GG_2$ structure, and this seems far out of reach of
current understanding.  However, inspired by Joyce's work, it is natural to study the problem of deforming a \emph{closed} $\GG_2$ structure, not necessarily with any smallness assumption on its torsion, to a torsion-free one, and to see if any obstructions arise to this procedure.  A proposal to tackle this
problem, due to Bryant (c.f.~\cite{bryant2005}), is to use a geometric flow.

Geometric flows are important and useful tools in geometry and topology.  For example, Ricci flow was instrumental in proving the Poincar\'e conjecture and the $\frac{1}{4}$-pinched differentiable sphere theorem, and K\"ahler--Ricci flow has proved to be a useful tool in
K\"ahler geometry, particularly in low dimensions.
In 1992, in order to study 7-manifolds admitting closed $\GG_2$ structures, Bryant
(see \cite{bryant2005}) introduced the Laplacian flow for closed $\GG_2$ structures:
\begin{equation}\label{Lap-flow-def}
  \left\{\begin{array}{rcl}
         \frac{\pt}{\pt t}\varphi &=&\Delta_{\varphi}\varphi,\\
           d\varphi &=&  0, \\
           \varphi(0) &=&\varphi_0,
         \end{array}\right.
\end{equation}
where $\Delta_{\varphi}\varphi=dd^*\varphi+d^*d\varphi$ is the Hodge Laplacian of $\varphi$ with respect to the metric $g$ determined by $\varphi$ and $\varphi_0$ is an initial closed $\GG_2$ structure.  The stationary
points of the flow are harmonic $\varphi$, which on a compact
manifold are the torsion-free $\GG_2$ structures.
The goal is to understand the long time behaviour of the flow; specifically, to find conditions
 under which the flow converges to a torsion-free $\GG_2$ structure. A reasonable conjecture (see \cite{bryant2005}), based on the work of Joyce described above, is that if the initial
$\GG_2$ structure $\varphi_0$ on a compact manifold is closed and has sufficiently small torsion, then the flow will exist for all time and converge to a torsion-free $\GG_2$ structure.

Another motivation for studying the Laplacian flow comes from work of Hitchin \cite{hitchin2000} (see also \cite{bryant-xu2011}), which demonstrates its relationship to
a natural volume functional.
Let $\bar{\varphi}$ be a closed $\GG_2$ structure on a compact 7-manifold $M$ and let
$[\bar{\varphi}]_+$ 
be the open subset of the cohomology class $[\bar{\varphi}]$ consisting of $\GG_2$ structures. The volume functional $\mathcal{H}:[\bar{\varphi}]_+\ra\R^+$ is defined by
\begin{equation*}
  \mathcal{H}(\varphi)=\frac 17\int_M\varphi\wedge *_{\varphi}\varphi=\int_M*_{\varphi}1.
\end{equation*}
Then $\varphi\in[\bar{\varphi}]_+$ is a critical point of $\mathcal{H}$ if and only if $d *_{\varphi}\!\varphi=0$, 
i.e.~$\varphi$ is torsion-free,   and the Laplacian flow can be viewed as the gradient flow for $\mathcal{H}$, with respect to a non-standard $L^2$-type metric on $[\bar{\varphi}]_+$
(see e.g.~\cite{bryant-xu2011}).

We note that there are other proposals for geometric flows of $\GG_2$ structures in various settings, which may also potentially
find torsion-free $\GG_2$ structures (e.g.~\cite{Grig,Kar,weiss-witt}).  The study of these flows is still in development.

An essential ingredient in studying the Laplacian flow \eqref{Lap-flow-def} is a short time existence result: this was claimed in \cite{bryant2005} and the proof given in \cite{bryant-xu2011}.
\begin{thm}
\label{thm-bryant-xu}
For a compact $7$-manifold $M$, the initial value problem \eqref{Lap-flow-def} has a unique solution for a short time $t\in[0,\epsilon)$ with $\epsilon$ depending on $\varphi_0$.
\end{thm}
\noindent To prove Theorem \ref{thm-bryant-xu}, Bryant--Xu showed that the flow \eqref{Lap-flow-def} is (weakly) parabolic in the direction of closed forms.
This is not a typical form of parabolicity, and so standard theory does not obviously apply.  It is also surprising since the flow is defined by the Hodge Laplacian (which is nonnegative) and thus appears at first sight to have the wrong sign for parabolicity.  Nonetheless,  the theorem follows by applying  DeTurck's trick and the
 Nash--Moser inverse function theorem.

This short time existence result  naturally motivates the study of the long time behavior of the flow.  Here little is known, apart from a compact example
computed by Bryant \cite{bryant2005} where the flow exists for all time but does not converge, and recently, Fern\'{a}ndez--Fino--Manero \cite{Fern-Fino-M}
 constructed some noncompact examples where the flow converges to a flat $\GG_2$ structure.

\subsection{Shi-type estimates} After some preliminary material on closed $\GG_2$ structures in \S\ref{sec:prelim} and deriving the essential evolution equations
along the flow in \S\ref{sec:evlution}, we prove our first main result in \S\ref{sec:shi}: Shi-type derivative estimates for the Riemann curvature and torsion
tensors along the Laplacian flow.

For a solution $\varphi(t)$ of the Laplacian flow \eqref{Lap-flow-def}, we define the quantity
\begin{equation}\label{Lambda-t-def}
  \Lambda(x,t)=\left(|\nabla T(x,t)|_{g(t)}^2+|Rm(x,t)|_{g(t)}^2\right)^{\frac 12},
\end{equation}
where $T$ is the torsion tensor of $\varphi(t)$ (see \S \ref{sec:prelim} for a definition) and $Rm$ denotes the Riemann curvature tensor of the  metric $g(t)$ determined by $\varphi(t)$.
Notice that $T$ is determined by the derivative of $\varphi$ and $Rm$ is second order in the metric which is determined algebraically by $\varphi$, so both $Rm$ and $\nabla T$
are second order in $\varphi$. We show that a bound on $\Lambda(x,t)$ will induce a priori bounds on all derivatives of $Rm$ and  $\nabla T$
for positive time.  More precisely, we have the following.
\begin{thm}\label{mainthm-shi}
Suppose that $K>0$ and $\varphi(t)$ is a solution of the Laplacian flow \eqref{Lap-flow-def} for closed $\GG_2$ structures on a compact manifold $M^7$ for $t\in [0,\frac 1K]$. For all $k\in\mathbb{N}$, there exists a constant $C_k$ such that if $\Lambda(x,t)\leq K$ on  $M^7\times [0,\frac 1K]$, then
\begin{equation}\label{shi-0}
  |\nabla^kRm(x,t)|_{g(t)}+|\nabla^{k+1}T(x,t)|_{g(t)}\leq {C_k}{t^{-\frac k2}}K,\quad t\in (0,\frac 1K].
\end{equation}
\end{thm}
\noindent We call the estimates \eqref{shi-0} Shi-type (perhaps, more accurately, Bernstein--Bando--Shi) estimates for the Laplacian flow, because they are analogues  of
 the well-known Shi derivative estimates in the Ricci flow.  In Ricci flow, a Riemann curvature bound will imply bounds on all the
derivatives of the Riemann curvature: this was proved by Bando \cite{Bando} and comprehensively by Shi \cite{shi} independently. The techniques used in
\cite{Bando,shi} were introduced by Bernstein (in the early twentieth century) for proving  gradient estimates via the maximum principle, and will also be used
 here in proving Theorem \ref{mainthm-shi}.

A key motivation for defining $\Lambda(x,t)$ as in \eqref{Lambda-t-def} is that the evolution equations of $|\nabla T(x,t)|^2$ and $|Rm(x,t)|^2$ both have some bad terms, but the chosen combination kills these terms and yields an effective evolution equation for $\Lambda(x,t)$.  We can then use the maximum principle
 to show that
\begin{equation}\label{Lambda-t-def-2}
  \Lambda(t)=\sup_M\Lambda(x,t)
\end{equation}
satisfies a doubling-time estimate (see Proposition \ref{prop-Rm-T^2}), i.e.~$\Lambda(t)\leq 2\Lambda(0)$ for all time $t\leq\frac 1{C\Lambda(0)}$ for which the flow exists, where $C$ is a uniform constant. This shows that $\Lambda$ has similar properties to Riemann curvature under Ricci flow. Moreover,
it implies that the assumption $\Lambda(x,t)\leq K$ in Theorem \ref{mainthm-shi} is reasonable as $\Lambda(x,t)$ cannot blow up quickly.
We conclude \S \ref{sec:shi} by giving a local version of Theorem \ref{mainthm-shi}.

In \S \ref{sec:longtime-I} we use our Shi-type estimates to study finite-time singularities of the Laplacian flow.  Given an initial closed $\GG_2$ structure $\varphi_0$ on a compact 7-manifold, Theorem \ref{thm-bryant-xu} tells us there exists a solution $\varphi(t)$ of the Laplacian flow
on a maximal time interval $[0,T_0)$.
If $T_0$ is finite, we call $T_0$ the singular time.  Using our global derivative estimates \eqref{shi-0} for $Rm$ and $\nabla T$, we can obtain the following long time existence result on the Laplacian flow.
\begin{thm}\label{mainthm-blowup}
If $\varphi(t)$ is a solution of the Laplacian flow \eqref{Lap-flow-def} on a compact manifold $M^7$ in a maximal time interval $[0,T_0)$ with $T_0<\infty$, then
\begin{equation*}
  \lim_{t\nearrow T_0}\Lambda(t)= \infty,
\end{equation*}
where $\Lambda(t)$ is given in \eqref{Lambda-t-def-2}.  Moreover, we have a lower bound on the blow-up rate:
\begin{equation*}
  \Lambda(t)\geq \frac{C}{T_0-t}
\end{equation*}
for some constant $C>0$.
\end{thm}
\noindent Theorem \ref{mainthm-blowup} shows that the solution $\varphi(t)$ of the Laplacian flow for closed $\GG_2$ structures will exist as long as the quantity $\Lambda(x,t)$ in \eqref{Lambda-t-def} remains bounded.
 We significantly strengthen this first long-time existence result in Theorem \ref{mainthm-longtime-II} below as a consequence of our compactness theory for the flow.

\subsection{Uniqueness} In \S \ref{sec:unique} we study uniqueness of the Laplacian flow, including both forward and backward uniqueness.

In Ricci flow, there are two standard arguments to prove forward uniqueness. One relies on the Nash--Moser inverse function theorem \cite{ha82} and another  relies on DeTurck's trick and the harmonic map flow (see \cite{ha95}).  Recently, Kotschwar \cite{kot2} provided a new approach to prove forward uniqueness. The idea in \cite{kot2} is to define an energy quantity $\mathcal{E}(t)$ in terms of the differences of the metrics, connections and Riemann curvatures of two Ricci flows, which vanishes if and only if the flows coincide. By deriving a differential inequality for $\mathcal{E}(t)$, it can be shown that $\mathcal{E}(t)=0$ if $\mathcal{E}(0)=0$, which gives the forward uniqueness.

In \cite{kot1}, Kotschwar proved backward uniqueness for complete solutions to the Ricci flow by deriving a general backward uniqueness theorem for time-dependent sections of vector bundles satisfying certain differential inequalities.
The method in \cite{kot1} is using Carleman-type estimates
inspired by \cite{Alexakis,Pin}. Recently, Kotschwar \cite{kot3} gave a simpler proof of the general backward uniqueness theorem in \cite{kot1}.

Here we will use the ideas in \cite{kot1,kot2} to give a new proof of forward uniqueness (given in \cite{bryant-xu2011}) and prove backward uniqueness of the Laplacian flow for closed $\GG_2$ structures, as stated below.

\begin{thm}\label{mainthm-uniq}
Suppose $\varphi(t)$, $\tilde{\varphi}(t)$ are two solutions to the Laplacian flow \eqref{Lap-flow-def}
on a compact manifold $M^7$ for $t\in [0,\epsilon]$, $\epsilon>0$. If $\varphi(s)=\tilde{\varphi}(s)$ for some $s\in [0,\epsilon]$, then $\varphi(t)=\tilde{\varphi}(t)$ for all $t\in[0,\epsilon]$.
\end{thm}

As an application of Theorem \ref{mainthm-uniq}, we show that on a compact manifold $M^7$,
the subgroup $I_{\varphi(t)}$ of diffeomorphisms of $M$ isotopic to the identity and
fixing $\varphi(t)$ is unchanged along the Laplacian flow.  Since $I_{\varphi}$ is strongly constrained for a torsion-free $\GG_2$ structure $\varphi$ on $M$,
this gives a test for when the Laplacian flow with a given initial condition could converge.

\subsection{Compactness} In the study of Ricci flow, Hamilton's compactness theorem \cite{ha95-compact} is an essential tool to study the behavior of the flow
 near a singularity. In \S \ref{sec:compact}, we prove an analogous compactness theorem for the Laplacian flow for closed $\GG_2$ structures.

Suppose we have a sequence $(M_i,\varphi_i(t))$ of compact solutions to the Laplacian flow and let $p_i\in M_i$.  For each $(M_i,\varphi_i(t))$, let
\begin{equation*}
  \Lambda_{\varphi_i}(x,t):=\left(|\nabla_{g_i(t)} T_i(x,t)|_{g_i(t)}^2+|Rm_{g_i(t)}(x,t)|_{g_i(t)}^2\right)^{\frac 12},
\end{equation*}
where $g_i(t)$ is the associated metric to $\varphi_i(t)$,
and let $inj(M_i,g_i(0),p_i)$ denote the injectivity radius of $(M_i,g_i(0))$ at the point $p_i$. Our compactness theorem
then states that under uniform bounds on $\Lambda_{\varphi_i}$ and $inj(M_i,g_i(0),p_i)$ we can extract a subsequence of $(M_i,\varphi_i(t))$ converging
to a limit flow $(M,\varphi(t))$.
\begin{thm}\label{mainthm-compact}
Let $M_i$ be a sequence of compact $7$-manifolds and let $p_i\in M_i$ for each $i$. Suppose that, for each $i$, $\varphi_i(t)$ is a solution
to the Laplacian flow \eqref{Lap-flow-def} on $M_i$ for $t\in (a,b)$, where $-\infty\leq a<0<b\leq \infty$. Suppose that
\begin{equation}\label{mainthm-compc-cond1}
  \sup_i\sup_{x\in M_i,t\in (a,b)}\Lambda_{\varphi_i}(x,t)<\infty
\end{equation}
and
\begin{equation}\label{mainthm-compc-cond2}
  \inf_i \textrm{inj}(M_i,g_i(0),p_i)>0.
\end{equation}
There exists a $7$-manifold $M$, a point $p\in M$ and a solution $\varphi(t)$ of the Laplacian flow
on $M$ for $t\in (a,b)$ such that, after passing to a subsequence,
\begin{equation*}
  (M_i,\varphi_i(t),p_i)\ra (M,\varphi(t),p)\quad\textrm{ as }i\ra\infty.
\end{equation*}
\end{thm}
\noindent We refer to \S\ref{sec:compact} for a definition of the notion of
convergence in Theorem \ref{mainthm-compact}.

To prove Theorem \ref{mainthm-compact}, we first prove a Cheeger--Gromov-type compactness theorem for the space of $\GG_2$ structures
(see Theorem \ref{compat-thm-G2}). Given this, Theorem \ref{mainthm-compact} follows from a similar argument for the analogous compactness theorem in
 Ricci flow as in \cite{ha95-compact}.

As we indicated, Theorem \ref{mainthm-compact} could be used to study the singularities of the Laplacian flow, especially if we can show some non-collapsing
 estimate as in Ricci flow (c.f.~\cite{perel}) to obtain the injectivity radius estimate \eqref{mainthm-compc-cond2}.
 Even without such an estimate, we can use Theorem \ref{mainthm-compact} to greatly strengthen Theorem \ref{mainthm-blowup} to the following desirable result,  which
states that the Laplacian flow will exist as long as the velocity of the flow remains bounded.
\begin{thm}\label{mainthm-longtime-II}
Let $M$ be a compact 7-manifold and $\varphi(t)$, $t\in [0,T_0)$, where $T_0<\infty$, be a solution to the Laplacian flow \eqref{Lap-flow-def} for closed $\GG_2$
 structures with associated metric $g(t)$ for each $t$. If the velocity of the flow satisfies
\begin{equation}\label{mainthm-9-1-cond}
  \sup_{M\times [0,T_0)}|\Delta_{\varphi}\varphi(x,t)|_{g(t)}<\infty,
\end{equation}
then the solution $\varphi(t)$ can be extended past time $T_0$.
\end{thm}

In Ricci flow, the analogue of Theorem \ref{mainthm-longtime-II} was proved in \cite{sesum2005}, namely that the flow exists as long as the Ricci tensor
remains bounded.
It is an open question whether just the scalar curvature (the trace of the Ricci tensor) can control the Ricci flow, although it is known for Type-I Ricci flow   \cite{Ender-M-T2010} and K\"{a}hler--Ricci flow \cite{zhang}.
In \S \ref{sec:hodge-lap}, we see that for a closed $\GG_2$ structure $\varphi$, we have $\Delta_{\varphi}\varphi=i_{\varphi}(h)$, where $i_{\varphi}: S^2T^*M\ra \Lambda^3T^*M$ is an injective map defined in \eqref{ivarphi-def} and $h$ is a symmetric $2$-tensor with trace equal to
$\frac 23|T|^2$.  Moreover, the scalar curvature of the metric induced by $\varphi$ is $-|T|^2$.  Thus, comparing with  Ricci flow,
one may ask whether the Laplacian flow for closed $\GG_2$ structures will exist as long as the torsion tensor remains bounded.  This is also the
natural question to ask from the point of view of $\GG_2$ geometry.  However, even though $-|T|^2$ is the scalar curvature,
it  is only \emph{first order} in $\varphi$, rather than second order like $\Delta_{\varphi}\varphi$, so it would be a major step forward to control the
Laplacian flow using just a bound on the torsion tensor.

\subsection{Solitons} In \S\ref{sec:solit}, we study soliton solutions of the Laplacian flow for closed $\GG_2$ structures, which are expected to play a role in understanding
 the behavior of the flow near singularities, particularly given our compactness theory for the flow.

Given a  7-manifold $M$, a Laplacian soliton of the Laplacian flow \eqref{Lap-flow-def} for closed $\GG_2$ structures on $M$
is a triple $(\varphi,X,\lambda)$ satisfying
\begin{equation}\label{solition-def}
  \Delta_{\varphi}\varphi=\lambda\varphi+\mathcal{L}_X\varphi,
\end{equation}
where $d\varphi=0$, $\lambda\in\R$, $X$ is a vector field on $M$ and $\mathcal{L}_X\varphi$ is the Lie derivative of $\varphi$ in the direction of $X$.  Laplacian solitons give self-similar solutions to the Laplacian flow.
 Specifically, suppose $(\varphi_0,X,\lambda)$ satisfies \eqref{solition-def}. Define
\begin{equation*}
  \rho(t)=(1+\frac {2}3\lambda t)^{\frac 32}, \quad X(t)=\rho(t)^{-\frac 23}X,
\end{equation*}
and let $\phi_t$ be the family of diffeomorphisms generated by the vector fields $X(t)$ such that $\phi_0$ is the identity. Then $\varphi(t)$ defined by
\begin{equation*}
  \varphi(t)=\rho(t)\phi_t^*\varphi_0
\end{equation*}
is a solution of the Laplacian flow \eqref{Lap-flow-def}, which only differs by a scaling factor $\rho(t)$ and pull-back by a diffeomorphism $\phi_t$ for different times $t$.
We say a Laplacian soliton $(\varphi,X, \lambda)$ is expanding if $\lambda>0$; steady if $\lambda=0$; and shrinking if $\lambda<0$.

Recently, there are several papers considering soliton solutions to flows of $\GG_2$ structures, e.g.~\cite{Kar-solit,Lin,weiss-witt2}.  In particular, Lin
\cite{Lin} studied Laplacian solitons as in \eqref{solition-def} and proved there are no compact shrinking  solitons, and that the only compact
 steady solitons are given by torsion-free $\GG_2$ structures.

A closed $\GG_2$ structure on a compact manifold which is stationary under the Laplacian flow must be torsion-free since here, unlike in the general non-compact
setting, harmonic forms are  always
closed and coclosed. We show that stationary points for the flow are torsion-free on \emph{any} 7-manifold
 and also give non-existence results for Laplacian solitons as follows.
\begin{prop}\label{prop-soliton}
\emph{(a)} Any Laplacian soliton of the form $(\varphi,0,\lambda)$ must be an expander or torsion-free.  Hence,
stationary points of the Laplacian flow are given by torsion-free $\GG_2$ structures.

\noindent \emph{(b)} There are no compact Laplacian solitons of the form $(\varphi,0,\lambda)$ unless $\varphi$ is torsion-free.
\end{prop}
Combining Lin's \cite{Lin} result and the above proposition,
 any Laplacian soliton on a compact manifold $M$ which is not torsion-free (if it exists) must satisfy \eqref{solition-def} for $\lambda>0$ and $X\neq 0$. This phenomenon is somewhat surprising, since it is very different from Ricci solitons
$Ric+\mathcal{L}_Xg=\lambda g$: when $X=0$, the Ricci soliton equation is just the Einstein equation $Ric=\lambda g$ and there are many examples of
compact Einstein metrics.

Since a $\GG_2$ structure $\varphi$ determines a unique metric $g$, it is natural to ask what condition the Laplacian soliton equation on $\varphi$ will impose on $g$. We show that for a closed $\GG_2$ structure $\varphi$ and any vector field $X$ on $M$, we have
\begin{equation}\label{symmetry}
 \mathcal{L}_X\varphi=\frac 12 i_{\varphi}(\mathcal{L}_Xg)+\frac 12 \big(d^*(X\lrcorner\varphi)\big)^{\sharp}\lrcorner\psi.
\end{equation}
Thus the symmetries of $\varphi$, namely the vector fields $X$ such $\mathcal{L}_X\varphi=0$,
are precisely given by the Killing vector fields $X$ of $g$ with $d^*(X\lrcorner\varphi)= 0$ on $M$.
 Moreover, using \eqref{symmetry} we can derive an equation for the metric $g$ from the Laplacian soliton equation \eqref{solition-def}, which we expect to be of further use (see Proposition \ref{prop-soliton-metric}). In particular, we deduce that any Laplacian soliton $(\varphi,X,\lambda)$ must satisfy $7\lambda+3\Div(X)=2|T|^2\geq 0$, which leads to a new short proof of the main result in \cite{Lin}.

To conclude the paper in \S\ref{sec:conclusion}, we provide a list of open problems that are inspired by our work and which we intend to study
in the future.

\section{Closed \texorpdfstring{$\GG_2$}{G2} structures} \label{sec:prelim}

We collect some facts on closed $\GG_2$ structures, mainly based on \cite{bryant2005, Kar}.

\subsection{Definitions} Let $\{e_1,e_2,\cdots,e_7\}$ denote the standard basis of $\R^7$ and let $\{e^1,e^2,\cdots,e^7\}$ be its dual basis. Write $e^{ijk}=e^i\wedge e^j\wedge e^k$ for simplicity and define the $3$-form
\begin{equation*}
  \phi=e^{123}+e^{145}+e^{167}+e^{246}-e^{257}-e^{347}-e^{356}.
\end{equation*}
The subgroup of $\GL(7,\R)$ fixing $\phi$ is the exceptional Lie group $\GG_2$, which is a compact, connected, simple Lie subgroup of $\SO(7)$ of dimension $14$.
Note that $\GG_2$ acts irreducibly on $\R^7$ and preserves the metric and orientation for which $\{e_1,e_2,\cdots,e_7\}$ is an oriented orthonormal basis. If $*_{\phi}$ denotes the Hodge star determined by the metric and orientation, then $\GG_2$ also preserves the $4$-form
\begin{equation*}
  *_{\phi}\phi=e^{4567}+e^{2367}+e^{2345}+e^{1357}-e^{1346}-e^{1256}-e^{1247}.
\end{equation*}

Let $M$ be a $7$-manifold. For $x\in M$ we let
\begin{equation*}
  \Lambda^3_+(M)_x=\{\varphi_x\in\Lambda^3T_x^*M\,|\,\exists\text{ invertible }u\in\textrm{Hom}_{\R}(T_xM,\R^7), u^*\phi=\varphi_x\},
\end{equation*}
which is isomorphic to $\GL(7,\R)/{\GG_2}$ since $\phi$ has stabilizer $\GG_2$. The bundle $\Lambda^3_+(M)=\bigsqcup_x \Lambda^3_+(M)_x$ is thus an open subbundle of $\Lambda^3T^*M$.  We call a section $\varphi$ of $\Lambda^3_+(M)$ a positive $3$-form on $M$ and denote the space of positive 3-forms by $\Omega_+^3(M)$. There is a 1-1 correspondence between $\GG_2$ structures (in the sense of subbundles of the frame bundle) and positive $3$-forms, because given $\varphi\in\Omega^3_+(M)$, the subbundle of the frame bundle whose fibre at $x$ consists of invertible $u\in\textrm{Hom}(T_xM,\R^7)$
such that $u^*\phi=\varphi_x$ defines a principal subbundle with fibre $\GG_2$. Thus we usually call a positive $3$-form $\varphi$ on $M$ a $\GG_2$ structure on $M$.  The existence of $\GG_2$ structures is equivalent to the property that $M$ is oriented and spin.

We now see that a positive $3$-form induces a unique metric and orientation. For a 3-form $\varphi$, we define a $\Omega^7(M)$-valued bilinear form $B_{\varphi}$ by
\begin{equation*}
  B_{\varphi}(u,v)=\frac 16(u\lrcorner\varphi)\wedge (v\lrcorner\varphi)\wedge \varphi,
\end{equation*}
where $u,v$ are tangent vectors on $M$. Then $\varphi$ is positive if and only if $B_{\varphi}$ is positive definite, i.e.~if $B_{\varphi}$ is the tensor product of a positive definite bilinear form and a nowhere vanishing $7$-form which defines a unique metric $g$ with volume form $vol_g$ as follows:
\begin{equation}\label{B-varphi}
  g(u,v)vol_g=B_{\varphi}(u,v).
\end{equation}
The metric and orientation determines the Hodge star operator $*_{\varphi}$, and we define $\psi=*_{\varphi}\varphi$,
which is sometimes called a positive $4$-form.  Notice that the relationship between $g$ and $\varphi$, and hence between $\psi$ and $\varphi$, is nonlinear.

The group $\GG_2$ acts irreducibly on $\R^7$ (and hence on $\Lambda^1(\R^7)^*$ and $\Lambda^6(\R^7)^*$), but it acts reducibly on $\Lambda^k(\R^7)^*$ for $2\leq k\leq 5$. Hence a $\GG_2$ structure $\varphi$ induces splittings of the bundles $\Lambda^kT^*M$ ($2\leq k\leq 5$) into direct summands, which we denote by $\Lambda^k_l(T^*M,\varphi)$ so that $l$ indicates the rank of the bundle. We let the space of sections of $\Lambda^k_l(T^*M,\varphi)$ be $\Omega^k_l(M)$.  We have that
\begin{align*}
  \Omega^2(M)= & \Omega^2_7(M)\oplus\Omega^2_{14}(M), \\
   \Omega^3(M)=& \Omega^3_1(M)\oplus\Omega^3_7(M)\oplus \Omega^3_{27}(M),
\end{align*}
where\footnote{Here we use the orientation in \cite{bryant2005} rather than \cite{Kar}.}
\begin{align*}
  \Omega^2_7(M) &= \{\beta\in\Omega^2(M)|\beta\wedge\varphi=2*_{\varphi}\beta\} =\{X\lrcorner\varphi|X\in C^{\infty}(TM)\},\\
  \Omega^2_{14}(M) &= \{\beta\in\Omega^2(M)|\beta\wedge\varphi=-*_{\varphi}\beta\} =\{\beta\in\Omega^2(M)|\beta\wedge\psi=0\},
\end{align*}
and
\begin{align*}
  \Omega^3_1(M)&=\{f\varphi|f\in C^{\infty}(M)\},\\
  \Omega^3_7(M)&=\{X\lrcorner\psi|X\in C^{\infty}(TM)\},\\
  \Omega^3_{27}(M)&=\{\gamma\in\Omega^3(M)|\gamma\wedge\varphi=0=\gamma\wedge\psi\}.
\end{align*}
Hodge duality gives corresponding decompositions of $\Omega^4(M)$ and $\Omega^5(M)$.

To study the Laplacian flow, it is convenient to write key quantities in local coordinates using summation convention. We write a $k$-form $\alpha$ as
\begin{equation*}
  \alpha=\frac 1{k!}\alpha_{i_1i_2\cdots i_k}dx^{i_1}\wedge\cdots\wedge dx^{i_k}
\end{equation*}
in local coordinates $\{x^1,\cdots, x^7\}$ on $M$, where  $\alpha_{i_1i_2\cdots i_k}$ is totally skew-symmetric in its indices.
In particular, we write $\varphi,\psi$ locally as
\begin{equation*}
  \varphi=\frac 16\varphi_{ijk}dx^i\wedge dx^j\wedge dx^k,\quad \psi=\frac 1{24}\psi_{ijkl}dx^i\wedge dx^j\wedge dx^k\wedge dx^l.
\end{equation*}
Note that the metric $g$ on $M$ induces an inner product of two $k$-forms $\alpha,\beta$, given locally by
\begin{equation*}
  \langle\alpha,\beta\rangle=\frac 1{k!}\alpha_{i_1i_2\cdots i_k}\beta_{j_1\cdots j_k}g^{i_1j_1}\cdots g^{i_kj_k}.
\end{equation*}

As in \cite{bryant2005} (up to a constant factor),  we define an operator $i_{\varphi}: S^2T^*M\ra \Lambda^3T^*M$ locally by
\begin{align}\label{ivarphi-def}
  i_{\varphi}(h)&=\frac 12h^l_i\varphi_{ljk}dx^i\wedge dx^j\wedge dx^k \nonumber\\
  &=\frac 16(h^l_i\varphi_{ljk}-h_j^l\varphi_{lik}-h_k^l\varphi_{lji})dx^i\wedge dx^j\wedge dx^k
\end{align}
where $h=h_{ij}dx^idx^j$. Then $\Lambda^3_{27}(T^*M,\varphi)=i_{\varphi}(S^2_0T^*M)$, where $S^2_0T^*M$ denotes the bundle of trace-free symmetric $2$-tensors on $M$. Clearly, $i_{\varphi}(g)=3\varphi$. We also have the inverse map $j_{\varphi}$ of $i_{\varphi}$,
\begin{equation*}
  j_{\varphi}(\gamma)(u,v)=*_{\varphi}((u\lrcorner\varphi)\wedge (v\lrcorner\varphi)\wedge \gamma),\quad u,v\in TM,
\end{equation*}
which is an isomorphism between $\Lambda^3_1(T^*M,\varphi)\oplus \Lambda^3_{27}(T^*M,\varphi)$ and $S^2T^*M$. Then  we have
$j_{\varphi}(i_{\varphi}(h))=4h+2tr_g(h)g $ for any $h\in S^2T^*M$ and $j_{\varphi}(\varphi)=6g$.

We have the following contraction identities of $\varphi$ and $\psi$ in index notation (see \cite{bryant2005,Kar}):
\begin{align}
  \varphi_{ijk}\varphi_{abl}g^{ia}g^{jb} &= 6g_{kl}, \label{contr-iden-1}\\
  \varphi_{ijq}\psi_{abkl}g^{ia}g^{jb}&= 4\varphi_{qkl},\label{contr-iden-2}\\
  \varphi_{ipq}\varphi_{ajk}g^{ia}&=g_{pj}g_{qk}-g_{pk}g_{qj}+\psi_{pqjk},\label{contr-iden-3}\\
  \varphi_{ipq}\psi_{ajkl}g^{ia}&=g_{pj}\varphi_{qkl}-g_{jq}\varphi_{pkl}+g_{pk}\varphi_{jql}-g_{kq}\varphi_{jpl}\nonumber\\
  &\qquad +g_{pl}\varphi_{jkq}-g_{lq}\varphi_{jkp},\label{contr-iden-4}\\
  \psi_{ijkl}\psi_{abcd}g^{jb}g^{kc}g^{ld}&=24g_{ia}. \label{contr-iden-5}
\end{align}

Given any $\GG_2$ structure $\varphi\in\Omega^3_+(M)$, there exist unique differential forms $\tau_0\in\Omega^0(M), \tau_1\in\Omega^1(M), \tau_{2}\in\Omega^2_{14}(M)$ and $\tau_{3}\in\Omega^3_{27}(M)$ such that $d\varphi$ and $d\psi$ can be expressed as follows (see \cite{bryant2005}):
\begin{align}
  d\varphi &= \tau_0\psi+3\tau_1\wedge\varphi+*_{\varphi}\tau_{3},\label{dvarphi}
\\
  d\psi &=  4\tau_1\wedge\psi+\tau_{2}\wedge\varphi.\label{dpsi}
\end{align}
We call $\{\tau_0,\tau_1,\tau_2,\tau_3\}$ the intrinsic torsion forms of the $\GG_2$ structure $\varphi$.  The full torsion tensor is a $2$-tensor $T$ satisfying (see \cite{Kar})
\begin{align}\label{nabla-var}
  \nabla_i\varphi_{jkl} &=T_i^{\,\,m}\psi_{mjkl},\\
\label{T-def}
  T_i^{\,\,j}&=\frac 1{24}\nabla_i\varphi_{lmn}\psi^{jlmn},
\end{align}
and
\begin{equation}
\label{nabla-psi}
  \nabla_m\psi_{ijkl} =-\Big( T_{mi}\varphi_{jkl}-T_{mj}\varphi_{ikl} -T_{mk}\varphi_{jil}-T_{ml}\varphi_{jki}\Big),
\end{equation}
where $T_{ij}=T(\pt_i,\pt_j)$ and $T_i^{\,\,j}=T_{ik}g^{jk}$. 
The full torsion tensor $T_{ij}$ is related to the intrinsic torsion forms by the following:
\begin{equation}\label{torsion-full}
 T_{ij}=\frac{\tau_0}4g_{ij}-(\tau_1^{\#}\lrcorner\varphi)_{ij}-(\bar{\tau}_3)_{ij}-\frac 12(\tau_2)_{ij},
\end{equation}
where $(\tau_1^{\#}\lrcorner\varphi)_{ij}=(\tau_1^{\#})^l\varphi_{lij}$ and $\bar{\tau}_3$ is the trace-free symmetric $2$-tensor such that $\tau_3=i_{\varphi}(\bar{\tau}_3)$.

If $\varphi$ is closed, i.e.~$d\varphi=0$, then \eqref{dvarphi} implies that $\tau_0,\tau_{1}$ and $\tau_3$ are all zero, so the only non-zero torsion form is $\tau_2=\frac 12(\tau_2)_{ij}dx^i\wedge dx^j$. Then from \eqref{torsion-full} we have that the full torsion tensor satisfies $T_{ij}=-T_{ji}=-\frac 12(\tau_2)_{ij}$, so $T$ is a skew-symmetric $2$-tensor. For the rest of the article, we write $\tau=\tau_2$ for simplicity and reiterate that for closed $\GG_2$
structures
\begin{equation}\label{T-tau-eq}
T=-\frac{1}{2}\tau.
\end{equation} Since $d\psi=\tau\wedge\varphi=-\!*_{\varphi}\!\tau$, we have that
\begin{equation}\label{tau.divfree.eq}
  d^*\tau=*_{\varphi}d*_{\varphi}\tau=-*_{\varphi}d^2\psi=0,
\end{equation}
which is given in local coordinates by $g^{mi}\nabla_m\tau_{ij}=0$.

We can write the condition that $\beta=\frac 12\beta_{ij}dx^i\wedge dx^j\in\Omega^2_{14}(M)$ as (see \cite{Kar})
\begin{align}\label{T-ident-1}
  &\beta_{ij}\varphi_{abk}g^{ia}g^{jb}=0\quad\textrm{ and }\quad \beta_{ij}\psi_{abkl}g^{ia}g^{jb}=-2\beta_{kl}
\end{align}
in local coordinates.

\subsection{Hodge Laplacian of \texorpdfstring{$\varphi$}{phi}}\label{sec:hodge-lap}
Since $d\varphi=0$, from \eqref{dvarphi} and \eqref{dpsi} we have that the Hodge Laplacian of $\varphi$ is equal to
\begin{align}\label{lap-varphi}
  \Delta_{\varphi}\varphi= & dd^*\varphi+d^*d\varphi=-d*_{\varphi}d\psi=d\tau,
\end{align}
where in the third equality we used $\tau\wedge\varphi=-*_{\varphi}\tau$ since $\tau\in\Omega^2_{14}(M)$. In local coordinates, we write \eqref{lap-varphi} as
\begin{align*}
  \Delta_{\varphi}\varphi=&\frac 16(\Delta_{\varphi}\varphi)_{ijk}dx^i\wedge dx^j\wedge dx^k,
\end{align*}
with
\begin{equation}\label{lap-varphi-0}
  (\Delta_{\varphi}\varphi)_{ijk}=\nabla_{i}\tau_{jk}-\nabla_j\tau_{ik}-\nabla_k\tau_{ji}.
\end{equation}
We can decompose $\Delta_{\varphi}\varphi$ into three parts:
\begin{align}\label{lap-varphi-decomp}
  \Delta_{\varphi}\varphi=&\pi_1^3(\Delta_{\varphi}\varphi)+\pi^3_7(\Delta_{\varphi}\varphi)+\pi^3_{27}(\Delta_{\varphi}\varphi)
  =a\varphi+X\lrcorner\psi+i_{\varphi}(\bar{h}),
\end{align}
where $\pi^k_l:\Omega^k(M)\ra \Omega^k_l(M)$ denotes the projection onto $\Omega^k_l(M)$, $a$ is a function, $X$ is a vector field and $\bar{h}$ is a trace-free symmetric $2$-tensor. We now calculate the values of $a,X,\bar{h}$.

For $a$, we take the inner product of $\varphi$ and $\Delta_{\varphi}\varphi$, and using the identity \eqref{T-ident-1} (since $\tau\in\Omega^2_{14}(M)$),
\begin{align*}
  a= & \frac 17\langle \Delta_{\varphi}\varphi,\varphi\rangle =\frac 1{42}\left(\nabla_{i}\tau_{jk}-\nabla_j\tau_{ik}-\nabla_k\tau_{ji}\right)\varphi_{lmn}g^{il}g^{jm}g^{kn}\\
  =&\frac 1{14}\nabla_{i}\tau_{jk}\varphi_{lmn}g^{il}g^{jm}g^{kn}\displaybreak[0]\\
  =&\frac 1{14}\nabla_{i}(\tau_{jk}\varphi_{lmn}g^{il}g^{jm}g^{kn})-\frac 1{14}\tau_{jk}\nabla_i\varphi_{lmn}g^{il}g^{jm}g^{kn}\displaybreak[0]\\
  =&\frac 1{28}\tau_{jk}\tau_{i}^s\psi_{slmn}g^{il}g^{jm}g^{kn}=\frac 1{14}\tau_{jk}\tau_{mn}g^{jm}g^{kn}=\frac 17|\tau|^2,
\end{align*}
where in the last equality we used $|\tau|^2=\frac 12\tau_{ij}\tau_{kl}g^{ik}g^{jl}$. For $X$, we use the contraction identities \eqref{contr-iden-2}, \eqref{contr-iden-4}, \eqref{contr-iden-5} and the definition of $i_{\varphi}$:
\begin{align*}
  (\Delta_{\varphi}\varphi\lrcorner\psi)_l= &(\Delta_{\varphi}\varphi)^{ijk}\psi_{ijkl}\\
  =&a\varphi^{ijk}\psi_{ijkl}+X^m\psi_m^{\,\,\,\,ijk}\psi_{ijkl}+(i_{\varphi}(\bar{h}))^{ijk}\psi_{ijkl}\\
  =&-24X_l+(\bar{h}^{im}\varphi_m^{\,\,\,\,jk}-\bar{h}^{jm}\varphi_m^{\,\,\,\,ik}-\bar{h}^{km}\varphi_m^{\,\,\,\,ji})\psi_{ijkl}\\
  =&-24X_l-12\bar{h}^{im}\varphi_{mil}=-24X_l,
\end{align*}
where the index of tensors are raised using the metric $g$. The last equality follows from the fact that $\bar{h}_{im}$ is symmetric in $i,m$, but $\varphi_{mil}$ is skew-symmetric in $i,m$. Using \eqref{lap-varphi-0}, we have
\begin{align*}
  X_l =& -\frac 1{24} (\Delta_{\varphi}\varphi)^{ijk}\psi_{ijkl}=-\frac 18g^{mi}\nabla_m\tau^{jk}\psi_{ijkl}\\
  =& -\frac 18g^{mi}\nabla_m(\tau^{jk}\psi_{ijkl})+\frac 18\tau^{jk}g^{mi}\nabla_m\psi_{ijkl}\\
  =&\frac 14 g^{mi}\nabla_m\tau_{il}+\frac 1{16}\tau^{jk}g^{mi}(\tau_{mi}\varphi_{jkl}-\tau_{mj}\varphi_{ikl}-\tau_{mk}\varphi_{jil}-\tau_{ml}\varphi_{jki})  =0,
\end{align*}
where in the above calculation we used \eqref{nabla-psi}, \eqref{tau.divfree.eq}, \eqref{T-ident-1}  and the totally skew-symmetry in $\varphi_{ijk}$ and $\psi_{ijkl}$.  So $X=0$ and thus the $\Omega^3_7(M)$ part of $\Delta_{\varphi}\varphi$ is zero. To find $h$, using the decomposition \eqref{lap-varphi-decomp}, $X=0$ and the contraction identities \eqref{contr-iden-2} and \eqref{contr-iden-3}, we have (as in \cite{Grig-Yau})
\begin{align*}
  (\Delta_{\varphi}\varphi)_i^{\,\,\,\,mn}\varphi_{jmn}&+(\Delta_{\varphi}\varphi)_j^{\,\,\,\,mn}\varphi_{imn}\\
  &=a\varphi_i^{\,\,\,mn}\varphi_{jmn}+X^l\psi_{li}^{\,\,\,\, mn}\varphi_{jmn}+(i_{\varphi}(\bar{h}))_i^{\,\,\,mn}\varphi_{jmn}\\
  &\qquad+a\varphi_j^{\,\,\,mn}\varphi_{imn}+X^l\psi_{lj}^{\,\,\,\, mn}\varphi_{imn}+(i_{\varphi}(\bar{h}))_j^{\,\,\,mn}\varphi_{imn}\\
  &=\frac{12}7|\tau|^2g_{ij}+8\bar{h}_{ij}.
  \end{align*}
The left-hand side of the above equation can be calculated using \eqref{lap-varphi-0}:
\begin{align*}
 (\nabla_m\tau_{ni}&-\nabla_n\tau_{mi}-\nabla_i\tau_{nm})\varphi_j^{\,\,\,mn}
+(\nabla_m\tau_{nj}-\nabla_n\tau_{mj}-\nabla_j\tau_{nm})\varphi_i^{\,\,\,mn}\\
  &=2(\nabla_m\tau_{ni}\varphi_j^{\,\,\,mn}+\nabla_m\tau_{nj}\varphi_i^{\,\,\,mn})-\nabla_i\tau_{nm}\varphi_j^{\,\,\,mn}-\nabla_j\tau_{nm}\varphi_i^{\,\,\,mn}\\
  &=4\nabla_m\tau_{ni}\varphi_j^{\,\,\,mn}+\tau_{nm}\nabla_i\varphi_j^{\,\,\,mn}+\tau_{nm}\nabla_j\varphi_i^{\,\,\,mn}\\
  &=4\nabla_m\tau_{ni}\varphi_j^{\,\,\,mn}-2\tau_i^{\,\,l}\tau_{lj},
\end{align*}
where we used \eqref{T-ident-1} and that for closed $\GG_2$ structures, $\nabla_m\tau_{ni}\varphi_j^{\,\,mn}$ is symmetric in $i,j$ (see Remark \ref{rem-Ricc}). Then
\begin{align*}
  \bar{h}_{ij}=&-\frac 3{14}|\tau|^2g_{ij}+\frac 12\nabla_m\tau_{ni}\varphi_j^{\,\,\,mn}-\frac 14\tau_i^{\,\,l}\tau_{lj}.
\end{align*}
 We conclude that
\begin{equation}\label{hodge-Lap-varp}
  \Delta_{\varphi}\varphi=d\tau=\frac 17|\tau|^2\varphi+i_{\varphi}(\bar{h})=i_{\varphi}(h)\in\Omega^3_1(M)\oplus\Omega^3_{27}(M),
\end{equation}
for
\begin{equation}\label{hodge-Lap-varp-2}
h_{ij}=\frac 12\nabla_m\tau_{ni}\varphi_j^{\,\,mn}-\frac 16|\tau|^2g_{ij}-\frac 14\tau_i^{\,\,l}\tau_{lj}.
\end{equation}

\subsection{Ricci curvature and torsion}

Since  $\varphi$ determines a unique metric $g$ on $M$, we then have the Riemann curvature tensor $Rm$ of $g$ on $M$. Our convention is the following:
\begin{equation*}
  R(X,Y)Z:=\nabla_X\nabla_YZ-\nabla_Y\nabla_XZ-\nabla_{[X,Y]}Z,
\end{equation*}
and $R(X,Y,Z,W)=g(R(X,Y)W,Z)$ for vector fields $X,Y,Z,W$ on $M$. In local coordinates
denote $R_{ijkl}=R(\pt_i,\pt_j,\pt_k,\pt_l)$. Recall that $Rm$ satisfies the first Bianchi identity:
\begin{equation}\label{Bianchi-Rm}
  R_{ijkl}+R_{iklj}+R_{iljk}=0.
\end{equation}
We also have the following Ricci identities when we commute covariant derivatives of a $(0,k)$-tensor $\alpha$:
\begin{equation}\label{Ricci-identity}
  (\nabla_i\nabla_j-\nabla_j\nabla_i)\alpha_{i_1i_2\cdots i_k}=\sum_{l=1}^kR_{iji_l}^{\quad m}\alpha_{i_1\cdots i_{l-1}mi_{l+1}\cdots i_k}.
\end{equation}

Karigiannis \cite{Kar} derived the following second Bianchi-type identity for the full torsion tensor.

\begin{lem}
\begin{equation}\label{bianchi-Torsion}
  \nabla_iT_{jk}-\nabla_jT_{ik}=\big(\frac 12R_{ijmn}-T_{im}T_{jn}\big)\varphi_k^{\,\,\,mn}.
\end{equation}
\end{lem}
\proof
The proof of \eqref{bianchi-Torsion} in \cite{Kar} is indirect, but
as remarked there, \eqref{bianchi-Torsion} can also be established directly using \eqref{nabla-var}--\eqref{nabla-psi} and the Ricci identity. We provide the detail here for completeness.
\begin{align*}
\nabla_iT_{jk}-\nabla_jT_{ik}&= \frac 1{24}(\nabla_i(\nabla_j\varphi_{abc}\psi_k^{\,\,\,abc})-\nabla_j(\nabla_i\varphi_{abc}\psi_k^{\,\,\,abc})) \\
  &=  \frac 1{24}(\nabla_i\nabla_j-\nabla_j\nabla_i)\varphi_{abc}\psi_k^{\,\,\,abc}\\
  &\quad
+\frac 1{24}(\nabla_j\varphi_{abc}\nabla_i\psi_k^{\,\,\,abc}-\nabla_i\varphi_{abc}\nabla_j\psi_k^{\,\,\,abc})\\
  &=\frac 1{24}(R_{ija}^{\quad m}\varphi_{mbc}+R_{ijb}^{\quad m}\varphi_{amc}+R_{ijc}^{\quad m}\varphi_{abm})\psi_k^{\,\,\,abc}\\
  &\quad-\frac 1{24}T_{j}^{\,\,m}\psi_{mabc}(T_{ik}\varphi^{abc}-T_{i}^{\,\,a}\varphi_k^{\,\,\,bc}+T_{i}^{\,\,b}\varphi_k^{\,\,\,ac}-T_{i}^{\,\,c}\varphi_k^{\,\,\,ab})\\
  &\quad+\frac 1{24}T_{i}^{\,\,m}\psi_{mabc}(T_{jk}\varphi^{abc}-T_{j}^{\,\,a}\varphi_k^{\,\,\,bc}+T_{j}^{\,\,b}\varphi_k^{\,\,\,ac}-T_j^{\,\,c}\varphi_k^{\,\,\,ab})\displaybreak[0]\\
  &=\frac 12R_{ijma}\varphi_k^{\,\,\,ma}+\frac 12T_{jm}T_{ia}\varphi_k^{\,\,\,ma}-\frac 12T_{im}T_{ja}\varphi_k^{\,\,\,ma}\\
  &=\frac 12R_{ijma}\varphi_k^{\,\,\,ma}-T_{ia}T_{jm}\varphi_k^{\,\,\,am},
\end{align*}
where in the third equality we used \eqref{nabla-var}, \eqref{nabla-psi} and \eqref{Ricci-identity}, and in the fourth equality we used the contraction identity \eqref{contr-iden-2}.
\endproof

We now consider the Ricci tensor, given locally as $R_{ik}=R_{ijkl}g^{jl}$, which has been calculated for closed $\GG_2$ structures (and more generally) in \cite{bryant2005, Cleyton-I, Kar}.  We give the general result from \cite{Kar} here.

\begin{prop}\label{Ric-prop}
The Ricci tensor of the associated metric $g$ of the $\GG_2$ structure $\varphi$ is given locally as
\begin{equation}\label{Ricc-prop}
R_{ik}=(\nabla_iT_{jl}-\nabla_jT_{il})\varphi_k^{\,\,\,jl}+Tr(T)T_{ik}-T_{i}^{\,\,j}T_{jk}+T_{im}T_{jn}\psi_k^{\,\,\,jmn}.
\end{equation}
In particular, for a closed $\GG_2$ structure $\varphi$, we have
 \begin{equation}\label{Ricc-prop-closed}
R_{ik}=\nabla_jT_{li}\varphi_k^{\,\,\,jl}-T_{i}^{\,\,j}T_{jk}.
\end{equation}
\end{prop}
\proof
We multiply \eqref{bianchi-Torsion} by $-\varphi_k^{\,\,\,jp}$:
\begin{align*}
-(\nabla_iT_{jp}&-\nabla_jT_{ip})\varphi_k^{\,\,\,jp}\\
&=-(T_{jm}T_{in}+\frac 12R_{ijmn})\varphi_p^{\,\,\,mn}\varphi_k^{\,\,\,jp} \\
  &=  (T_{jm}T_{in}+\frac 12R_{ijmn})(g^{mj}\delta_{nk}-\delta_{mk}g^{nj}-\psi_k^{\,\,\,jmn}) \\
  &= -T_{i}^{\,\,j}T_{jk}+ Tr(T)T_{ik}-T_{jm}T_{in}\psi_k^{\,\,\,jmn}-R_{ik}-\frac 12R_{ijmn}\psi_k^{\,\,\,jmn}\\
  &=  -T_{i}^{\,\,j}T_{jk}+ Tr(T)T_{ik}-T_{jm}T_{in}\psi_k^{\,\,\,jmn}-R_{ik} \\&\qquad
-\frac 16(R_{ijmn}+R_{imnj}+R_{injm})\psi_k^{\,\,\,jmn}\\
&= -T_{i}^{\,\,j}T_{jk}+ Tr(T)T_{ik}-T_{jm}T_{in}\psi_k^{\,\,\,jmn}-R_{ik},
\end{align*}
where the last equality is  due to \eqref{Bianchi-Rm}. The formula \eqref{Ricc-prop} follows.

For a closed $\GG_2$ structure, we have $T_{ij}=-\frac 12\tau_{ij}$, so $T$ is skew-symmetric. Moreover, using \eqref{T-ident-1}, we have
\begin{equation*}
  -T_{jm}T_{in}\psi_k^{\,\,\,jmn}=-\frac 14 \tau_{jm}\tau_{in}\psi_k^{\,\,\,jmn}=-\frac 12\tau_{i}^{\,\,n}\tau_{nk}=-2T_i^{\,\,n}T_{nk},
\end{equation*}
and
\begin{align*}
  \nabla_iT_{jp}\varphi_k^{\,\,\,jp} &= \nabla_i(T_{jp}\varphi_k^{\,\,\,jp})-T_{jp}\nabla_i\varphi_k^{\,\,\,jp}\nonumber\\
  &=  -\frac 12\nabla_i(\tau_{jp}\varphi_k^{\,\,\,jp}) -T_{jp}T_i^{\,\,m}\psi_{mk}^{\,\,\,\,\,\,\,jp}\\
  &=-\frac 14\tau_{jp}\tau_i^{\,\,m}\psi_{mk}^{\,\,\,\,\,\,\,jp}=\frac 12\tau_i^{\,\,m}\tau_{mk}=2T_i^{\,\,m}T_{mk}.
\end{align*}
Then we obtain
\begin{align*}
R_{ik}&=(\nabla_iT_{jp}-\nabla_jT_{ip})\varphi_k^{\,\,\,jp}+Tr(T)T_{ik}-T_{i}^{\,\,j}T_{jk}-T_{jm}T_{in}\psi_k^{\,\,\,jmn}\\
&=2T_i^{\,\,m}T_{mk}-\nabla_jT_{ip}\varphi_k^{\,\,\,jp}-T_{i}^{\,\,j}T_{jk}-2T_i^{\,\,n}T_{nk}\\
&=-\nabla_jT_{ip}\varphi_k^{\,\,\,jp}-T_{i}^{\,\,j}T_{jk},
\end{align*}
which is \eqref{Ricc-prop-closed}.
\endproof

\begin{rem}\label{rem-Ricc}
By \eqref{Ricc-prop-closed}, for a closed $\GG_2$ structure,
$\nabla_jT_{ip}\varphi_k^{\,\,\,jp}$ is symmetric in $i,k$, since $R_{ik}$ and $T_{i}^{\,\,j}T_{jk}$ are symmetric in $i,k$.
\end{rem}

We noted earlier that    $Rm$ and $\nabla T$ are second order in $\varphi$, and $T$ is essentially $\nabla\varphi$, so
we would expect $Rm$ and $\nabla T$ to be related.  We show explicitly using Proposition \ref{Ric-prop} that, for closed $\GG_2$ structures, this is  the case.

\begin{prop}\label{prop-nabla-T}
For a closed $\GG_2$ structure $\varphi$, we have
\begin{align}\label{naT-Rm}
 2\nabla_iT_{jk} =&\frac 12R_{ijmn}\varphi_k^{\,\,\,mn}+\frac 12R_{kjmn} \varphi_i^{\,\,\,mn}-\frac 12R_{ikmn}\varphi_j^{\,\,\,mn}\nonumber\\
 &~~-T_{im}T_{jn}\varphi_k^{\,\,\,mn} -T_{km}T_{jn}\varphi_i^{\,\,\,mn}+T_{im}T_{kn}\varphi_j^{\,\,\,mn}.
 \end{align}
\end{prop}
\proof
By interchanging $i\leftrightarrow k$ and $j\leftrightarrow k$ in \eqref{bianchi-Torsion} respectively, we have
\begin{align}
   \nabla_kT_{ji}-\nabla_jT_{ki}=&\big(\frac 12R_{kjmn}-T_{km}T_{jn}\big)\varphi_i^{\,\,\,mn}\label{bian-2}\\
   \nabla_iT_{kj}-\nabla_kT_{ij}=&\big(\frac 12R_{ikmn}-T_{im}T_{kn}\big)\varphi_j^{\,\,\,mn}.\label{bian-3}
\end{align}
Then \eqref{naT-Rm} follows by combining the equations \eqref{bianchi-Torsion} and \eqref{bian-2}--\eqref{bian-3}.
\endproof

We can also deduce a useful, already known, formula for the scalar curvature of the metric given by a closed $\GG_2$ structure.

\begin{corr}\label{scalar-cor}
The scalar curvature of a metric associated to a closed $\GG_2$ structure satisfies
\begin{equation}\label{scalar-corr}
  R=-|T|^2=-T_{ik}T_{jl}g^{ij}g^{kl}.
\end{equation}
\end{corr}
\proof
By taking trace in \eqref{Ricc-prop-closed}, using $T_{ij}=-\frac 12\tau_{ij}$ and \eqref{T-ident-1}, we obtain the scalar curvature
\begin{align*}
  R=&R_{sk}g^{sk}=-(\nabla_jT_{sp}\varphi_k^{\,\,\,jp}+T_{s}^{\,\,j}T_{jk})g^{sk}\nonumber\\
  =&-\nabla_j(T_{sp}\varphi_k^{\,\,\,jp})g^{sk}+T_{sp}\nabla_j\varphi_k^{\,\,\,jp}g^{sk}+|T|^2\nonumber\\
  =&\frac 12\nabla_j(\tau_{sp}\varphi_k^{\,\,\,jp})g^{sk}+T_{sp}T_{j}^{\,\,m}\psi_{mk}^{\,\,\,\,\,\,\,jp}g^{sk}+|T|^2\nonumber\\
  =&\frac 14\tau_{sp}\tau_j^{\,\,m}\psi_{mk}^{\,\,\,\,\,\,\,jp}g^{sk}+|T|^2=-\frac 12\tau_{sp}\tau^{sp}+|T|^2\nonumber\\
  =&-2T_{sp}T^{sp}+|T|^2=-|T|^2
\end{align*}
as claimed.
\endproof

This result is rather striking since it shows that the scalar curvature, which is a priori second order in the metric and hence in $\varphi$, is given by a first order quantity in $\varphi$
when $d\varphi=0$.

\section{Evolution equations}\label{sec:evlution}

In this section we derive evolution equations for several geometric quantities under the Laplacian flow, including the torsion tensor $T$, Riemann curvature tensor $Rm$, Ricci tensor $Ric$ and scalar curvature $R$.  These are fundamental equations for understanding the flow.

Recall that the Laplacian flow for a closed $\GG_2$ structure is
\begin{equation}\label{flow-1}
  \frac{\pt}{\pt t}\varphi=\Delta_{\varphi}\varphi.
\end{equation}
From \eqref{hodge-Lap-varp} and \eqref{hodge-Lap-varp-2}, the flow \eqref{flow-1} is equivalent to
\begin{equation}\label{flow-2}
  \frac{\pt}{\pt t}\varphi=i_{\varphi}(h),
\end{equation}
where $h$ is the symmetric $2$-tensor given in \eqref{hodge-Lap-varp-2}. We may write $h$ in terms the full torsion tensor $T_{ij}$ as follows:
\begin{equation}\label{hodge-Lap-varp-3}
h_{ij}=-\nabla_mT_{ni}\varphi_{j}^{\,\,mn}-\frac 13|T|^2g_{ij}-T_i^{\,\,l}T_{lj}.
\end{equation}
For closed $\varphi$, the Ricci curvature is equal to
\begin{equation*}
  R_{ij}=\nabla_mT_{ni}\varphi_{j}^{\,\,mn}-T_i^{\,\,k}T_{kj},
\end{equation*}
so we can also write $h$ as
\begin{equation}\label{h-tensor-1}
  h_{ij}=-R_{ij}-\frac 13|T|^2g_{ij}-2T_i^{\,\,k}T_{kj}.
\end{equation}
Notice that $T_i^{\,\,k}=T_{il}g^{kl}$ and $T_{il}=-T_{li}$.

Throughout this section and the remainder of the article we will use the symbol $\Delta$ to denote the ``analyst's Laplacian'' which is a non-positive operator given in local
coordinates as $\nabla^i\nabla_i$.  This is in contrast to $\Delta_{\varphi}$, which is the Hodge Laplacian and is instead a non-negative operator.

\subsection{Evolution of the metric}
Under a general flow for $\GG_2$ structures
\begin{equation}\label{flow-general}
  \frac{\pt}{\pt t}\varphi(t)=i_{\varphi(t)}(h(t))+X(t)\lrcorner \psi(t),
\end{equation}
where $h(t), X(t)$ are a time-dependent symmetric $2$-tensor and vector field on $M$ respectively, it is well known that (see \cite{bryant2005,joyce2000} and explicitly \cite{Kar})
the associated metric tensor $g(t)$ evolves by
\begin{equation*}
  \frac{\pt}{\pt t}g(t)=2h(t).
\end{equation*}
Substituting \eqref{h-tensor-1} into this equation, we have that under the Laplacian flow \eqref{flow-1} (also given by \eqref{flow-2}), the associated metric $g(t)$ of the $\GG_2$ structure $\varphi(t)$ evolves by
\begin{equation}\label{flow-g2}
  \frac{\pt}{\pt t}g_{ij}=-2R_{ij}-\frac 23|T|^2g_{ij}-4T_i^{\,\,k}T_{kj}.
\end{equation}
Thus the leading term of the metric flow \eqref{flow-g2} corresponds to the Ricci flow, as already observed in \cite{bryant2005}.

From \eqref{flow-g2} we have that the inverse of the metric evolves by
\begin{align}\label{flow-g^-1}
  \frac{\pt}{\pt t}g^{ij}=&-g^{ik}g^{jl}\frac{\pt}{\pt t}g_{kl}\nonumber\\
  =&g^{ik}g^{jl}(2R_{kl}+\frac 23|T|^2g_{kl}+4T_k^{\,\,m}T_{ml}),
\end{align}
and the volume form $vol_{g(t)}$ evolves by
\begin{align}\label{evl-volumeform}
  \frac{\pt}{\pt t}vol_{g(t)}=&\frac 12tr_g(\frac{\pt}{\pt t}g(t))vol_{g(t)}=tr_g(h(t))vol_{g(t)}\nonumber\\
  =&(-R-\frac 73|T|^2+2|T|^2)vol_{g(t)}=\frac 23|T|^2vol_{g(t)},
\end{align}
where we used the fact that the scalar curvature $R=-|T|^2$. Hence, along the Laplacian flow, the volume of $M$ with respect to the associated metric $g(t)$ will non-decrease; in fact, the volume form is pointwise non-decreasing (again as already noted in \cite{bryant2005}).

\subsection{Evolution of torsion}

By \cite[Lemma 3.7]{Kar},  the evolution of the full torsion tensor $T$ under the flow \eqref{flow-2} is given by \footnote{Note that compared with \cite[Lemma 3.7]{Kar}, the sign of the second term on the right-hand side of \eqref{evl-torsion} is different due to a different choice of orientation of $\psi$, which also leads to a different sign for the torsion tensor $T$.}
\begin{equation}\label{evl-torsion}
  \frac{\pt}{\pt t}T_{ij}=T_i^{\,\,k}h_{kj}-\nabla_mh_{in}\varphi_j^{\,\,mn}.
\end{equation}
Substituting \eqref{hodge-Lap-varp-3} into \eqref{evl-torsion}, we obtain
\begin{align}\label{evl-T-1}
  \frac{\pt}{\pt t}&T_{ij}=-\nabla_mh_{in}\varphi_j^{\,\,mn}+T_i^{\,\,k}h_{kj}\nonumber\displaybreak[0]\\
  &=-\nabla_m\left(-\nabla_pT_{qi}\varphi_n^{\,\,pq}-\frac 13|T|^2g_{in}-T_i^{\,\,k}T_{kn}\right)\varphi_j^{\,\,mn}\nonumber\\
  & +T_i^{\,\,k}\left(-\nabla_pT_{qk}\varphi_j^{\,\,pq}-\frac 13|T|^2g_{kj}-T_k^{\,\,m}T_{mj}\right)\nonumber\\
  &=\nabla_m\nabla_pT_{qi}\varphi_n^{\,\,pq}\varphi_j^{\,\,mn}+\nabla_pT_{qi}\nabla_m\varphi_n^{\,\,pq}\varphi_j^{\,\,mn}-\frac 13\nabla_m|T|^2\varphi_{ji}^{\,\,\,\,\, m}\nonumber\\
  &+\nabla_m(T_i^{\,\,k}T_{kn})\varphi_j^{\,\,mn}-T_i^{\,\,k}\nabla_pT_{qk}\varphi_j^{\,\,pq}-\frac 13|T|^2T_{ij}-T_i^{\,\,k}T_k^{\,\,m}T_{mj}.
\end{align}
Using the contraction identity \eqref{contr-iden-3} and Ricci identity \eqref{Ricci-identity}, the first term on the right hand side of \eqref{evl-T-1} is equal to
\begin{align}\label{evl-T-2}
  \nabla_m\nabla_pT_{qi}&\varphi_n^{\,\,\,pq}\varphi_j^{\,\,mn}\nonumber\\
&= \nabla_m\nabla_pT_{qi}(\delta^p_jg^{qm}-\delta^q_{j}g^{pm}+\psi_{j}^{\,\,pqm}) \nonumber\\
  &=  \nabla^m\nabla_jT_{mi}-\nabla^m\nabla_m T_{ji}+\nabla_m\nabla_pT_{qi}\psi_j^{\,\,pqm}\nonumber\displaybreak[0]\\
  &=\Delta T_{ij}+\nabla_j\nabla^mT_{mi}-R_j^{\,\,k}T_{ki}+R_{mjik}T^{mk}\nonumber\\
&\quad+\frac 12(\nabla_m\nabla_pT_{qi}-\nabla_p\nabla_mT_{qi})\psi_j^{\,\,pqm}\nonumber\displaybreak[0]\\
  &=\Delta T_{ij}-R_j^{\,\,k}T_{ki}+R_{mjik}T^{mk}+\frac 12(R_{mpi}^{\quad\,\,\, k}T_{qk}+R_{mpq}^{\quad\,\,\,k}T_{ki})\psi_j^{\,\,pqm}\nonumber\\
  &=\Delta T_{ij}-R_j^{\,\,k}T_{ki}+\frac 12(R_{mjik}-R_{kjim})T^{mk}+\frac 12R_{mpi}^{\quad\,\,\, k}T_{qk}\psi_j^{\,\,pqm}\nonumber\\
&\quad +\frac 16(R_{mpq}^{\quad\,\,\,k}+R_{pqm}^{\quad\,\,\,k}+R_{qmp}^{\quad\,\,\,k})T_{ki}\psi_j^{\,\,pqm}\nonumber\\
  &=\Delta T_{ij}-R_j^{\,\,k}T_{ki}+\frac 12R_{ijmk}T^{mk}+\frac 12R_{mpi}^{\quad\,\,\, k}T_{qk}\psi_j^{\,\,pqm}
\end{align}
where we used $\nabla^mT_{mi}=0$ in the fourth equality and the Bianchi identity \eqref{Bianchi-Rm} in the last equality. Using the contraction identity \eqref{contr-iden-4} and \eqref{nabla-var}, we can calculate the second term on the right hand side of \eqref{evl-T-1} as follows:
\begin{align}\label{evl-T-3}
  \nabla_pT_{qi}\nabla_m\varphi_n^{\,\,\,pq}\varphi_j^{\,\,mn}&=\nabla_pT_{qi}T_m^{\,\,\,k}\psi_{kn}^{\quad\!\! pq}\varphi_j^{\,\,mn}  \nonumber\\
   &= \nabla_pT_{qi}T_m^{\,\,k}(\delta^m_{k}\varphi_j^{\,\,\,pq}-g_{jk}\varphi^{mpq}+g^{mp}\varphi_{kj}^{\,\,\,\,\,\,q}\nonumber\\
   &\qquad\qquad\qquad -\delta_{j}^{p}\varphi_k^{\,\,\,mq}-g^{mq}\varphi_{kj}^{\,\,\,\,\,\,p}-\delta_{j}^{q}\varphi_k^{\,\,\,pm})\nonumber\\
   &=-\nabla_pT_{qi}(T_{mj}\varphi^{mpq}-T^{pk}\varphi_{kj}^{\,\,\,\,\,\,q}+T^{qk}\varphi_{kj}^{\,\,\,\,\,\,p}),
\end{align}
where in the last equality we used $T_m^{\,\,k}\delta_{k}^m=0$ and $T_m^{\,\,k}\varphi_k^{\,\,\,mq}=-\frac 12\tau_m^{\,\,k}\varphi_k^{\,\,\,mq}=0$ since $\tau\in \Omega^2_{14}(M)$. Then substituting \eqref{evl-T-2}--\eqref{evl-T-3} into \eqref{evl-T-1}, we obtain
\begin{align*}
 \frac{\pt}{\pt t}T_{ij} =&\Delta T_{ij}-R_j^{\,\,k}T_{ki}+\frac 12R_{ijmk}T^{mk}+\frac 12R_{mpi}^{\quad\,\,\, k}T_{qk}\psi_j^{\,\,pqm}\\
  &-\nabla_pT_{qi}(T_{mj}\varphi^{mpq}-T^{pk}\varphi_{kj}^{\,\,\,\,\,\,q}+T^{qk}\varphi_{kj}^{\,\,\,\,\,\,p})-\frac 13\nabla_m|T|^2\varphi_{ji}^{\,\,\,\,\, m}\\
  &+\nabla_m(T_i^{\,\,k}T_{kn})\varphi_j^{\,\,mn}-T_i^{\,\,k}\nabla_pT_{qk}\varphi_j^{\,\,pq}-\frac 13|T|^2T_{ij}-T_i^{\,\,k}T_k^{\,\,m}T_{mj}.
 \end{align*}
We can further simplify the above equations by noting that
\begin{align*}
   -\nabla_pT_{qi}(&T_{mj}\varphi^{mpq}\!-\!T^{pk}\varphi_{kj}^{\,\,\,\,\,\,q}\!+\!T^{qk}\varphi_{kj}^{\,\,\,\,\,\,p})+\nabla_m(T_i^{\,\,k}T_{kn})\varphi_j^{\,\,mn}-T_i^{\,\,k}\nabla_pT_{qk}\varphi_j^{\,\,pq}\\
  &=-\nabla_pT_{qi}(T_{mj}\varphi^{mpq}-T^{pk}\varphi_{kj}^{\,\,\,\,\,\,q}+2T^{qk}\varphi_{kj}^{\,\,\,\,\,\,p})-2T_i^{\,\,k}\nabla_pT_{qk}\varphi_j^{\,\,pq}\\
  &=\nabla_pT_{qi}(T^{pk}\varphi_{kj}^{\,\,\,\,\,\,q}-2T^{qk}\varphi_{kj}^{\,\,\,\,\,\,p})-R_i^{\,\,k}T_{kj}+2R_j^{\,\,k}T_{ki}-3T_i^{\,\,k}T_k^{\,\,l}T_{lj},
\end{align*}
where we used the expression of Ricci tensor in \eqref{Ricc-prop-closed}. Therefore, we have
\begin{align}
 \frac{\pt}{\pt t}T_{ij} &=\Delta T_{ij}+R_j^{\,\,k}T_{ki}-R_i^{\,\,k}T_{kj}+\frac 12R_{ijmk}T^{mk}+\frac 12R_{mpi}^{\quad\,\,\, k}T_{qk}\psi_j^{\,\,pqm}\nonumber\\
  &\quad+\nabla_pT_{qi}(T^{pk}\varphi_{kj}^{\,\,\,\,\,\,q}-2T^{qk}\varphi_{kj}^{\,\,\,\,\,\,p})-\frac 13\nabla_m|T|^2\varphi_{ji}^{\,\,\,\,\, m}\nonumber\\
  &\quad-\frac 13|T|^2T_{ij}-4T_i^{\,\,k}T_k^{\,\,m}T_{mj}.\nonumber
 \end{align}
The above evolution equation of the torsion tensor can be expressed schematically as
\begin{equation}\label{evl-torsion-1}
  \frac{\pt}{\pt t}T=\Delta T+Rm*T+Rm*T*\psi+\nabla T*T*\varphi+T*T*T,
\end{equation}
where $*$ indicates a contraction using the metric $g(t)$  determined by $\varphi(t)$.  

\subsection{Evolution of curvature}

To calculate the evolution of the Riemann curvature tensor we will use well-known general evolution equations. Recall that for any smooth one-parameter family of metrics $g(t)$ on a manifold evolving by
\begin{equation}\label{flow-g-any}
  \frac{\pt}{\pt t}g(t)=\eta(t),
\end{equation}
for some time-dependent symmetric $2$-tensor $\eta(t)$, the Riemann curvature tensor, Ricci tensor and  scalar curvature evolve by (see e.g.~\cite[Lemma 6.5]{Chow-Knopf})
\begin{align}
  \frac{\pt}{\pt t} R_{ijk}^{\quad l} =& \frac 12g^{lp}\left(\nabla_i\nabla_k\eta_{jp}+\nabla_j\nabla_p\eta_{ik}-\nabla_i\nabla_p\eta_{jk}-\nabla_j\nabla_k\eta_{ip} \right.\nonumber\\&\quad\qquad
-R_{ijk}^{\quad q}\eta_{qp}-R_{ijp}^{\quad q}\eta_{kq}),\label{evl-Rm-any}\\
\label{evl-Ric-any}
  \frac{\pt}{\pt t} R_{ik} =& -\frac 12\left(\Delta_L\eta_{ik}+\nabla_i\nabla_k(\textrm{tr}_g\eta)+\nabla_i(\delta \eta)_k+\nabla_k(\delta\eta)_i\right),\\
\label{evl-scalar-any}
  \frac{\pt}{\pt t} R =& -\Delta \textrm{tr}_g(\eta)+\Div (\Div \eta)-\langle \eta,\textrm{Ric}\rangle,
\end{align}
where $\Delta_L$ denotes the Lichnerowicz Laplacian
\begin{equation*}
  \Delta_L\eta_{ik}:=\Delta \eta_{ik}-R_i^{\,p}\eta_{pk}-R_{k}^{\,p}\eta_{ip}+2R_{pikl}\eta^{lp}
\end{equation*}
and $(\delta\eta)_k=-(\Div \eta)_k=-\nabla^i\eta_{ik}$. Substituting \eqref{flow-g2} into \eqref{evl-Rm-any}, we have
\begin{align*}
  \frac{\pt}{\pt t} R_{ijk}^{\quad l}&= -\nabla_i\nabla_kR_j^l-\nabla_j\nabla^lR_{ik}+\nabla_i\nabla^lR_{jk}+\nabla_j\nabla_kR_{i}^l \\
   & +(R_{ijk}^{\quad q}R_{qp}+R_{ijp}^{\quad q}R_{kq})g^{lp}+2g^{lp}(R_{ijk}^{\quad q}T_{q}^mT_{mp}+R_{ijp}^{\quad q}T_{k}^mT_{mq})\\
   &-\frac 13 g^{lp}(\nabla_i\nabla_k|T|^2g_{jp}+\nabla_j\nabla_p|T|^2g_{ik}-\nabla_i\nabla_p|T|^2g_{jk}-\nabla_j\nabla_k|T|^2g_{ip})\\
   &-2g^{lp}(\nabla_i\nabla_k(T_j^mT_{mp})+\nabla_j\nabla_p(T_i^mT_{mk})\\
&\qquad\qquad\qquad\qquad-\nabla_i\nabla_p(T_j^mT_{mk})-\nabla_j\nabla_k(T_i^mT_{mp})).
\end{align*}
The first six terms in the evolution equation come from the $-2Ric$ term in \eqref{flow-g2}. Then, as in Ricci flow, by applying Bianchi identities and commuting covariant derivatives,  we can obtain
\begin{align*}
  \frac{\pt}{\pt t} R_{ijk}^{\quad l}&=\Delta R_{ijk}^{\quad l}+g^{pq}(R_{ijp}^{\quad r}R_{rqk}^{\quad l}-2R_{pik}^{\quad r}R_{jqr}^{\quad l}+2R_{pir}^{\quad l}R_{jqk}^{\quad r})\\
&-g^{pq}(R_{ip}R_{qjk}^{\quad l}+R_{jp}R_{iqk}^{\quad l})-g^{pq}(R_{kq}R_{ijp}^{\quad l}-R_p^lR_{ijkq})\\
&+2g^{lp}(R_{ijk}^{\quad q}T_{q}^mT_{mp}+R_{ijp}^{\quad q}T_{k}^mT_{mq})\\
   &-\frac 13 g^{lp}(\nabla_i\nabla_k|T|^2g_{jp}+\nabla_j\nabla_p|T|^2g_{ik}-\nabla_i\nabla_p|T|^2g_{jk}-\nabla_j\nabla_k|T|^2g_{ip})\\
   &-2g^{lp}\big(\nabla_i\nabla_k(T_j^mT_{mp})+\nabla_j\nabla_p(T_i^mT_{mk})\\
&\qquad\qquad\qquad\qquad-\nabla_i\nabla_p(T_j^mT_{mk})-\nabla_j\nabla_k(T_i^mT_{mp})\big).
 \end{align*}
We write the above equation schematically as in \eqref{evl-torsion-1}:
\begin{equation}\label{Rm}
  \frac{\pt}{\pt t} Rm=\Delta Rm+Rm*Rm+Rm*T*T+\nabla^2T*T+\nabla T*\nabla T.
\end{equation}
Then from \eqref{flow-g^-1} and \eqref{Rm}, noting that $|T|^2=-R\leq C|Rm|$ for some universal constant $C$, we have
\begin{align}\label{evl-Rm^2}
  \frac{\pt}{\pt t}|Rm|^2&= \frac{\pt}{\pt t}(R_{ijkl}R_{abcd}g^{ia}g^{jb}g^{kc}g^{ld}) \nonumber\\
  &= Rm*Rm*(Ric+T*T)+2\langle Rm, \frac{\pt}{\pt t} Rm\rangle\nonumber\\
  &\leq \Delta |Rm|^2-2|\nabla Rm|^2+C|Rm|^3+C|Rm|^{\frac 32}|\nabla^2T|\nonumber\\
  &\quad +C|Rm||\nabla T|^2
\end{align}
Similarly, substituting \eqref{flow-g2} into \eqref{evl-Ric-any} and \eqref{evl-scalar-any}, we obtain the evolution equation of the Ricci tensor
\begin{align}\label{evl-Ric-1}
  \frac{\pt}{\pt t} R_{ik} &= \Delta_L(R_{ik}+\frac 13|T|^2g_{ik}+2T_i^{\,\,l}T_{lk})-\frac 23\nabla_i\nabla_k|T|^2\nonumber\\
  &\quad -2(\nabla_i\nabla^j(T_j^{\,\,l}T_{lk})+\nabla_k\nabla^j(T_j^{\,\,l}T_{li})),
\end{align}
and the evolution equation of the scalar curvature
\begin{align}\label{evl-scalar-1}
  \frac{\pt}{\pt t} R =& \Delta R-4\nabla^k\nabla^j(T_j^{\,\,l}T_{lk})+2|Ric|^2 -\frac 23R^2+4R^{ik}T_i^{\,\,l}T_{lk}.
\end{align}

\begin{rem} We shall only require the schematic evolution equations \eqref{evl-torsion-1} and \eqref{Rm} for $T$ and $Rm$ to derive our Shi-type estimates.  To obtain these equations we used the fact that $\varphi$ remains closed under the evolution, which is a particular property of the Laplacian flow.  If one is able to obtain the same schematic evolution equations for $T$ and $Rm$ for another flow of $\GG_2$ structures, then the methods of this article will apply more generally to give Shi-type estimates for that flow.
\end{rem}

\section{Derivative estimates of curvature and torsion}\label{sec:shi}

In this section, we use the evolution equations derived in \S \ref{sec:evlution} to obtain global derivative estimates for the curvature tensor $Rm$ and torsion tensor $T$.  Throughout, we use $*$  to denote some contraction between tensors and often use the same symbol $C$ for a finite number of constants for convenience.

First, we show a doubling-time estimate for $\Lambda(t)$ defined in \eqref{Lambda-t-def-2}, which roughly says that $\Lambda(t)$ behaves well and cannot blow up quickly.

\begin{prop}[Doubling-time estimate]\label{prop-Rm-T^2}
Let $\varphi(t)$ be a solution to the Laplacian flow \eqref{Lap-flow-def} on a compact 7-manifold
for $t\in [0,\epsilon]$. There exists a constant $C$ such that $\Lambda(t)\leq 2\Lambda(0)$ for all
$t$ satisfying $0\leq t\leq \min\{\epsilon,\frac 1{C\Lambda(0)}\}$.
\end{prop}
\proof
We will calculate a  differential inequality for $\Lambda(x,t)$ given in \eqref{Lambda-t-def},
$$\Lambda(x,t)=\left(|\nabla T(x,t)|_{g(t)}^2+|Rm(x,t)|_{g(t)}^2\right)^{\frac 12}$$
and thus for $\Lambda(t)=\sup_{x\in M}\Lambda(x,t)$.  Since we already have an evolution equation for $|Rm|^2$ in \eqref{evl-Rm^2}, it suffices to compute
the evolution of $|\nabla T|^2$.

Recall that for any smooth family of metrics $g(t)$ evolving by \eqref{flow-g-any}, the Christoffel symbols of the Levi-Civita connection of $g(t)$ evolve by
\begin{equation*}
  \frac{\pt}{\pt t}\Gamma_{ij}^k=\frac 12g^{kl}(\nabla_i\eta_{jl}+\nabla_j\eta_{il}-\nabla_l\eta_{ij}).
\end{equation*}
Thus, for any time-dependent tensor $A(t)$, we have the commutation formula (see \cite[\S 2.3]{topping2006})
\begin{equation}\label{commut-tesor}
  \frac{\pt}{\pt t}\nabla A-\nabla \frac{\pt}{\pt t}A=A*\nabla\frac{\pt}{\pt t}g.
\end{equation}
The fact that the metric $g$ is parallel gives that for any two tensors $A,B$,
\begin{equation*}
  \nabla(A*B)=\nabla A*B+A*\nabla B.
\end{equation*}
Then using \eqref{flow-g2}, \eqref{evl-torsion-1}  and \eqref{commut-tesor}, we see that
\begin{align}\label{evl-nabla-T1}
  \frac{\pt}{\pt t}\nabla T &= \nabla  \frac{\pt}{\pt t} T+T*\nabla \frac{\pt}{\pt t}g \nonumber\\
  &=  \nabla \Delta T+\nabla Rm*(T+T*\psi)+\nabla T*(Rm+Rm*\psi)\nonumber\\
  &\quad+Rm*T*\nabla\psi+\nabla^2T*T*\varphi+\nabla T*\nabla T*\varphi\nonumber\\
  &\quad+\nabla T*T*\nabla\varphi+\nabla T*T*T\nonumber\\
  &=  \Delta\nabla T+\nabla Rm*(T+T*\psi)+\nabla T*(Rm+Rm*\psi) \nonumber\\
  &\quad+Rm*T*T*\varphi+\nabla^2T*T*\varphi+\nabla T*\nabla T*\varphi \nonumber\\
  &\quad+\nabla T*T*T*\psi+\nabla T*T*T,
\end{align}
where in the last equality we used \eqref{nabla-var} and \eqref{nabla-psi} in the form
\begin{equation*}
  \nabla\varphi=T*\psi,\quad \nabla\psi=T*\varphi,
\end{equation*}
and we commuted covariant derivatives using the Ricci identity, i.e.
\begin{equation*}
\nabla \Delta T=\Delta\nabla T+Rm*\nabla T+\nabla Rm*T.
\end{equation*}
 Then we can calculate the evolution  of the squared norm of $\nabla T$:
\begin{align}\label{evl-nabla_T^2}
  \frac{\pt}{\pt t}|\nabla T|^2 &= 2\langle\nabla T,\frac{\pt}{\pt t}\nabla T\rangle+\nabla T*\nabla T*\frac{\pt}{\pt t} g \nonumber\\
   &\leq \Delta|\nabla T|^2-2|\nabla^2T|^2+C|Rm||\nabla T|^2+C|\nabla Rm||T||\nabla T|\nonumber\\
   &\quad+C|Rm||T|^2|\nabla T|+C|\nabla^2T||\nabla T||T|\nonumber\\
   &\quad+C|\nabla T|^3+C|\nabla T|^2|T|^2 \nonumber\displaybreak[0]\\
   &\leq \Delta|\nabla T|^2-2|\nabla^2T|^2+C|Rm||\nabla T|^2+C|\nabla Rm||Rm|^{\frac 12}|\nabla T|\nonumber\\
   &\quad+C|Rm|^2|\nabla T|+C|Rm|^{\frac 12}|\nabla^2T||\nabla T|+C|\nabla T|^3,
\end{align}
where we used $|T|^2=-R\leq C|Rm|$ for a constant $C$ in the last inequality.

Now, using \eqref{evl-Rm^2} and \eqref{evl-nabla_T^2}, we obtain
\begin{align}
  \frac{\pt}{\pt t}\Lambda(x,t)^2&\leq \Delta (|Rm|^2+|\nabla T|^2)-2|\nabla Rm|^2-2|\nabla^2T|^2+C|Rm|^3\nonumber\\
  &\quad+C|Rm|^{\frac 32}|\nabla^2T|+C|Rm||\nabla T|^2+C|\nabla Rm||Rm|^{\frac 12}|\nabla T|\nonumber\\
  &\quad+C|Rm|^2|\nabla T|+C|Rm|^{\frac 12}|\nabla^2T||\nabla T|+C|\nabla T|^3\label{evl-Rm-T-1}.
\end{align}  
By Young's inequality, namely $ab\leq \frac{1}{2\epsilon}a^2+\frac{\epsilon}{2}b^2$ for any $\epsilon>0$ and $a,b\geq 0$, for all $\epsilon>0$ we have 
\begin{align}
|Rm|^{\frac{3}{2}}|\nabla^2T|&\leq \frac{1}{2\epsilon}|Rm|^3+\frac{\epsilon}{2}|\nabla^2T|^2,\label{evl-Rm-T-2}\\
|\nabla Rm||Rm|^{\frac{1}{2}}|\nabla T|&\leq \frac{1}{2\epsilon}|Rm||\nabla T|^2 + \frac{\epsilon}{2}|\nabla Rm|^2,\label{evl-Rm-T-3}\\
|Rm|^{\frac{1}{2}}||\nabla^2T||\nabla T|&\leq \frac{1}{2\epsilon}|Rm||\nabla T|^2+\frac{\epsilon}{2}|\nabla^2T|^2.\label{evl-Rm-T-4}
\end{align}
The terms $|Rm|^3$, $|Rm||\nabla T|^2$ and $|\nabla T|^3$ can all be bounded above by $\Lambda^3=(|Rm|^2+|\nabla T|^2)^\frac{3}{2}$ up to a multiplicative constant.  Using this bound and substituting \eqref{evl-Rm-T-2}--\eqref{evl-Rm-T-4} into \eqref{evl-Rm-T-1} we obtain 
\begin{align*}
  \frac{\pt}{\pt t}\Lambda(x,t)^2&\leq \Delta \Lambda(x,t)^2+(C\epsilon-2)(|\nabla Rm|^2+|\nabla^2T|^2)+\frac{C}{\epsilon}\Lambda(x,t)^3
\end{align*}
for any $\epsilon>0$.  Choosing $\epsilon$ so  $C\epsilon\leq 1$ then yields
\begin{align} 
\frac{\pt}{\pt t}\Lambda(x,t)^2  &\leq \Delta \Lambda(x,t)^2- (|\nabla Rm|^2+|\nabla^2T|^2)+C\Lambda(x,t)^{3}.\label{evl-Rm-T}
\end{align}
The idea behind the calculations leading to \eqref{evl-Rm-T} is that the  negative gradient terms appearing in the evolution equations of $|\nabla T|^2$ and $|Rm|^2$ allow us to kill the remaining bad terms to leave us with an effective differential inequality. This is precisely the motivation for the definition $\Lambda(x,t)$ in \eqref{Lambda-t-def} as a combination of $|\nabla T|$ and $|Rm|$.

Recall that $\Lambda(t)=\sup_M\Lambda(x,t)$, which is a Lipschitz function of time $t$. Applying the maximum principle to \eqref{evl-Rm-T}, we deduce that
\begin{align*}
  \frac{d}{d t}\Lambda(t)&\leq \frac C2\Lambda(t)^{2},
\end{align*}
in the sense of $\limsup$ of forward difference quotients. We conclude that
\begin{equation}\label{doble-time-pf1}
  \Lambda(t)\leq \frac {\Lambda(0)}{1-\frac 12C\Lambda(0)t}
\end{equation}
as long as $t\leq \min\{\epsilon,\frac 2{C\Lambda(0)}\}$, so $\Lambda(t)\leq 2\Lambda(0)$ if $t\leq \min\{\epsilon,\frac 1{C\Lambda(0)}\}$.
\endproof

We now derive Shi-type derivative estimates for the curvature tensor $Rm$ and torsion tensor $T$ along the Laplacian flow,
 using $\Lambda(x,t)$ given in \eqref{Lambda-t-def}.
\begin{thm}\label{thm-shi}
Suppose that $K>0$ and $\varphi(t)$ is a solution to the Laplacian flow \eqref{Lap-flow-def} for closed $\GG_2$ structures
on a compact  manifold $M^7$ with $t\in [0,\frac 1K]$. For all $k\in\mathbb{N}$, there exists a constant $C_k$ such that if $\Lambda(x,t)\leq K$ on
 $M^7\times [0,\frac 1K]$, then for all $t\in [0,\frac 1K]$ we have
\begin{equation}\label{shi-1}
  |\nabla^kRm|+|\nabla^{k+1}T|\leq {C_k}{t^{-\frac k2}}K.
\end{equation}
\end{thm}
\proof
The proof is by induction on $k$.  The idea is to define a suitable function $f_k(x,t)$ for each $k$, in a similar way to the Ricci flow, which satisfies a parabolic differential inequality amenable to the maximum principle.

For the case
$k=1$,
we define
\begin{equation}\label{shi-f-eq}
  f=t(|\nabla Rm|^2+|\nabla^2T|^2)+\alpha (|\nabla T|^2+|Rm|^2)
\end{equation}
for $\alpha$ to be determined later. To calculate the evolution of $f$, we first need to calculate the evolution of $\nabla Rm$ and $\nabla^2T$. Using \eqref{flow-g2}, \eqref{Rm} and \eqref{commut-tesor},
\begin{align}\label{evl-na-Rm1}
  \frac{\pt}{\pt t}\nabla Rm &= \nabla \frac{\pt}{\pt t}Rm+Rm*\nabla\frac{\pt}{\pt t}g(t) \nonumber\\
  &= \nabla\Delta Rm+Rm*\nabla Rm+\nabla Rm*T*T+Rm*T*\nabla T\nonumber\\
  &\quad+\nabla^3T*T+\nabla^2T*\nabla T +Rm*\nabla(Ric+T*T)\nonumber\\
  &= \Delta\nabla Rm+Rm*\nabla Rm+\nabla Rm*T*T+Rm*T*\nabla T\nonumber\\
  &\quad+\nabla^3T*T+\nabla^2T*\nabla T,
\end{align}
where in the last equality we used the commuting formula
\begin{equation*}
  \nabla\Delta Rm=\Delta\nabla Rm+Rm*\nabla Rm.
\end{equation*}
Then using \eqref{flow-g^-1}, \eqref{evl-na-Rm1} and $|T| \leq C |Rm|^{\frac{1}{2}}$, 
\begin{align}\label{evl-d-Rm^2}
  \frac{\pt}{\pt t}|\nabla Rm|^2 &\leq \Delta|\nabla Rm|^2-2|\nabla^2Rm|^2+C|\nabla Rm|^2|Rm|  \nonumber\\
  &+C|\nabla Rm|\left(|Rm|^{\frac 32}|\nabla T|+|Rm|^{\frac 12}|\nabla^3T|+|\nabla^2T||\nabla T|\right).
\end{align}
Similarly, we can use \eqref{commut-tesor} and \eqref{evl-nabla-T1} to obtain
\begin{align}\label{evl-na^2-T-0}
  \frac{\pt}{\pt t}\nabla^2T &= \Delta\nabla^2T+\nabla^2Rm*(T+T*\psi) \nonumber\\
   &\quad +\nabla Rm*(\nabla T+\nabla T*\psi+T^2*\varphi)\nonumber\\
   &\quad+Rm*(\nabla^2T+\nabla^2 T*\psi+\nabla T*T*\varphi+T^3*\psi)\nonumber\\
   &\quad+\nabla^3T*T*\varphi+\nabla^2T*\nabla T*\varphi+\nabla T*T^3*\varphi \nonumber\\
   &\quad+\nabla^2T*(T^2+T^2*\psi)+\nabla T*\nabla T*(T+T*\psi),
\end{align}
where we use the symbols $T^2$ and $T^3$ here to mean contractions of two or three copies of $T$ respectively, and again use $|T| \leq C |Rm|^{\frac{1}{2}}$ to find
\begin{align}\label{evl-na^2-T}
  &\frac{\pt}{\pt t}|\nabla^2T|^2\leq \Delta|\nabla^2T|^2-2|\nabla^3T|^2+C|\nabla^2Rm||\nabla^2T||Rm|^{\frac 12} \nonumber \\
   & +C|\nabla Rm||\nabla^2T|(|\nabla T|+|Rm|)+C|\nabla^3T||\nabla^2T||Rm|^{\frac 12}
\\
   &+C|\nabla^2T|^2(|Rm|+|\nabla T|)+C|\nabla^2T||Rm|^{\frac 12}(|Rm|^2
   +|Rm||\nabla T|+|\nabla T|^2)\nonumber.
\end{align}
Using Young's inequality, we know that for all $\epsilon>0$ we have
\begin{align*}
2|\nabla Rm||Rm|^{\frac{3}{2}}|\nabla T|&\leq |\nabla Rm||Rm|^{\frac{1}{2}}(|Rm|^2+|\nabla T|^2), \\
2|\nabla Rm||Rm|^{\frac{1}{2}}|\nabla^3 T|&\leq \frac{1}{\epsilon}|\nabla Rm|^2|Rm|+\epsilon |\nabla^3T|^2,  \\
2|\nabla Rm||\nabla^2T|(|\nabla T|+|Rm|)&\leq (|\nabla Rm|^2+|\nabla^2T|^2)(|\nabla T|+|Rm|), \\
2|\nabla^2 Rm||\nabla^2 T||Rm|^{\frac{1}{2}}&\leq \frac{1}{\epsilon}|\nabla^2T|^2|Rm|+\epsilon|\nabla^2Rm|^2, \\
2|\nabla^3T||\nabla^2 T||Rm|^{\frac{1}{2}}&\leq \frac{1}{\epsilon}|\nabla^2T|^2|Rm|+\epsilon|\nabla^3T|^2, \\
2|\nabla^2T||Rm|^{\frac{1}{2}}|Rm||\nabla T|&\leq |\nabla^2 T||Rm|^{\frac{1}{2}}(|Rm|^2+|\nabla T|^2).
\end{align*}
Substituting these bounds into \eqref{evl-d-Rm^2} and \eqref{evl-na^2-T}, for suitably chosen small $\epsilon>0$ as before, then yields
\begin{align}\label{evl-d-Rm-T}
  \frac{\pt}{\pt t}(|\nabla Rm|^2 +|\nabla^2T|^2)&\leq \Delta (|\nabla Rm|^2 +|\nabla^2T|^2)-(|\nabla^2Rm|^2+|\nabla^3T|^2)\nonumber\\
   &+C(|\nabla Rm|^2+|\nabla^2T|^2)(|\nabla T|+|Rm|)
\\
  &+C(|\nabla Rm|+|\nabla^2T|)|Rm|^{\frac 12}(|Rm|^2+|\nabla T|^2).\nonumber
\end{align}

Then, from \eqref{evl-Rm-T} and \eqref{evl-d-Rm-T}, we obtain
\begin{align*}
  \frac{\pt}{\pt t}f &\leq \Delta f
+Ct(|\nabla Rm|^2+|\nabla^2T|^2)(|\nabla T|+|Rm|)\nonumber\\
&\quad+Ct(|\nabla Rm|+|\nabla^2T|)|Rm|^{\frac 12}(|Rm|^2+|\nabla T|^2)\nonumber\\
  &\quad +(1-\alpha)(|\nabla Rm|^2+|\nabla^2T|^2)+C\alpha (|\nabla T|^2+|Rm|^2)^{\frac 32}.\nonumber
\end{align*}
By hypothesis $\Lambda(t)=\sup_{x\in M}\Lambda(x,t)\leq K$ and $tK\leq 1$, so using the above inequality and Young's inequality to combine the middle three terms implies 
\begin{align}
  \frac{\pt}{\pt t}f &\leq \Delta f+(C-\alpha)(|\nabla Rm|^2+|\nabla^2T|^2)+C\alpha K^3.
\end{align}
We can choose $\alpha$ sufficiently large that $  C-\alpha\leq 0$ and thus
\begin{align*}
  \frac{\pt}{\pt t}f &\leq \Delta f+C\alpha K^3.
\end{align*}
Note that $f(x,0)=\alpha (|\nabla T|^2+|Rm|^2)\leq \alpha K^2$, so applying the maximum principle to the above inequality implies that
\begin{equation*}
  \sup_{x\in M}f(x,t)\leq \alpha K^2+Ct\alpha K^3\leq CK^2.
\end{equation*}
 From the definition \eqref{shi-f-eq} of $f$, we obtain \eqref{shi-1} for $k=1$:
 \begin{equation*}
   |\nabla Rm|+|\nabla^2T|\leq CKt^{-\frac 12}.
 \end{equation*}

Given this, we next prove $k\geq 2$ by induction.  It is clear that we need to obtain differential inequalities for $|\nabla^kRm|^2$ and $|\nabla^{k+1}T|^2$, so this is how we proceed. Suppose 
\eqref{shi-1} holds for all $1\leq j<k$. From \eqref{commut-tesor}, for any time-dependent tensor $A(t)$ we have
\begin{equation}\label{commut-tensor-k}
   \frac{\pt}{\pt t}\nabla^k A-\nabla^k \frac{\pt}{\pt t}A=\sum_{i=1}^k\nabla^{k-i}A*\nabla^i\frac{\pt}{\pt t}g.
\end{equation}
By \eqref{flow-g2}, \eqref{Rm} and \eqref{commut-tensor-k}, we have
\begin{align}\label{dt-nabla^kRm}
  &\frac{\pt}{\pt t}\nabla^kRm = \nabla^k\frac{\pt}{\pt t}Rm+\sum_{i=1}^k\nabla^{k-i}Rm*\nabla^i\frac{\pt}{\pt t}g. \nonumber\\
  &= \nabla^k\Delta Rm+\nabla^k(Rm*Rm)+\nabla^k(Rm*T^2)+\nabla^{k+1}(\nabla T*T)\nonumber\\
  &\quad+\sum_{i=1}^k\nabla^{k-i}Rm*\nabla^i(Ric+T*T)\nonumber\\
  &=\Delta \nabla^kRm+\sum_{i=0}^k\nabla^{k-i}Rm*\nabla^i(Rm+T*T)+\sum_{i=0}^{k+1}\nabla^iT*\nabla^{k+2-i}T,
\end{align}
where in the last equality we used the Ricci identity
\begin{equation}\label{commu-tensor-lapl}
  \nabla^k\Delta Rm-\Delta\nabla^k Rm=\sum_{i=0}^k\nabla^{k-i}Rm*\nabla^iRm.
\end{equation}
Using \eqref{dt-nabla^kRm}, the evolution of the squared norm of $\nabla^kRm$ is: 
\begin{align}\label{dt-|na^kRm|}
  \frac{\pt}{\pt t}|\nabla^kRm|^2 &=  \Delta |\nabla^kRm|^2-2|\nabla^{k+1}Rm|^2\nonumber\\
  &+\sum_{i=0}^k\nabla^kRm*\nabla^{k-i}Rm*\nabla^i(Rm+T*T)\nonumber\\
  &+\sum_{i=0}^{k+1}\nabla^kRm*\nabla^iT*\nabla^{k+2-i}T. 
\end{align}
Applying \eqref{shi-1} for $1\leq j<k$ to \eqref{dt-|na^kRm|}, we get
\begin{align}\label{dt-|na^kRm|-2}
  \frac{\pt}{\pt t}|\nabla^kRm|^2 &\leq  \Delta |\nabla^kRm|^2-2|\nabla^{k+1}Rm|^2+CK^{\frac 12}|\nabla^kRm||\nabla^{k+2}T|\nonumber\\
  &\quad+CK(|\nabla^kRm|^2+|\nabla^{k+1}T|^2)+CK^2t^{-\frac k2}|\nabla^{k}Rm|\nonumber\\
  &\leq  \Delta |\nabla^kRm|^2-2|\nabla^{k+1}Rm|^2+CK^{\frac 12}|\nabla^kRm||\nabla^{k+2}T|\nonumber\\
  &\quad+CK^{3}t^{-k}+CK(|\nabla^kRm|^2+|\nabla^{k+1}T|^2),
\end{align}
where the constant $C$ depends on the constants $C_j, 1\leq j<k$ in \eqref{shi-1} and we used Young's inequality to estimate
\[
2K^2t^{-\frac k2}|\nabla^{k}Rm|=2K^{\frac{3}{2}}t^{-\frac k2}K^{\frac{1}{2}}|\nabla^kRm|\leq 
K^3t^{-k}+K|\nabla^kRm|^2.
\]
Similarly, we have
\begin{align*}
  &\frac{\pt}{\pt t}\nabla^{k+1}T= \nabla^{k+1}\frac{\pt}{\pt t}T+\sum_{i=1}^{k+1}\nabla^{k+1-i}T*\nabla^i\frac{\pt}{\pt t}g. \displaybreak[0]\\
  &=  \nabla^{k+1}\Delta T+\nabla^{k+1}(Rm*T)+\nabla^{k+1}(Rm*T*\psi)+\nabla^{k+1}(\nabla T*T*\varphi)\\
  &\quad+\nabla^{k+1}(T*T*T)+\sum_{i=1}^{k+1}\nabla^{k+1-i}T*\nabla^i(Ric+T*T)\displaybreak[0]\\
  &=\Delta\nabla^{k+1} T+\sum_{i=0}^{k+1}\nabla^{k+1-i}T*\nabla^iRm+\sum_{i=0}^{k+1}\nabla^{k+1-i}T*\nabla^i(T*T)\displaybreak[0]\\
  &\quad+\sum_{i=0}^{k+1}\nabla^{k+1-i}(Rm*T)*\nabla^i\psi+\sum_{i=0}^{k+1}\nabla^{k+1-i}(\nabla T*T)*\nabla^i\varphi
\end{align*}
and
\begin{align}\label{dt-na^k-T}
  \frac{\pt}{\pt t}|\nabla^{k+1}T|^2&=  \Delta|\nabla^{k+1} T|^2-2|\nabla^{k+2}T|^2\nonumber\\
  &\quad+\sum_{i=0}^{k+1}\nabla^{k+1} T*\nabla^{k+1-i}T*\nabla^i(Rm+T*T)\nonumber\displaybreak[0]\\
  &\quad+\sum_{i=0}^{k+1}\nabla^{k+1} T*\nabla^{k+1-i}(Rm*T)*\nabla^i\psi\nonumber\\
  &\quad+\sum_{i=0}^{k+1}\nabla^{k+1} T*\nabla^{k+1-i}(\nabla T*T)*\nabla^i\varphi.
\end{align}
The second line of \eqref{dt-na^k-T} can be estimated using the second line of \eqref{dt-|na^kRm|}. To estimate the third line of \eqref{dt-na^k-T}, for $2\leq i\leq k+1$ we have
\begin{align}\label{line3-1}
  |\nabla^{k+1-i}(Rm*T)|&\leq  \sum_{j=0}^{k+1-i}|\nabla^{k+1-i-j}Rm*\nabla^jT|\leq Ct^{-\frac {k-i}2}(K^{\frac 32}t^{-\frac 12}+K^2).
\end{align}
For $i=1$,
\begin{align}\label{thm4.2.eq.i1a}
  \nabla^k(Rm*T) &= \nabla^kRm*T+\sum_{l=1}^k\nabla^{k-l}Rm*\nabla^lT,
\end{align}
where
\begin{align}\label{thm4.2.eq.i1b}
|\sum_{l=1}^k\nabla^{k-l}Rm*\nabla^lT|&\leq CK^2t^{-\frac {k-1}2}.
\end{align}
Similarly for $i=0$, we have
\begin{align}\label{thm4.2.eq.i0a}
  \nabla^{k+1}(Rm*T) =& \nabla^{k+1}Rm*T+\nabla^kRm*\nabla T+\sum_{l=2}^k\nabla^{k+1-l}Rm*\nabla^lT,
\end{align}
where
\begin{align}\label{thm4.2.eq.i0b}
|\sum_{l=2}^k\nabla^{k+1-l}Rm*\nabla^lT|\leq &CK^2t^{-\frac {k}2}.
\end{align}
Using \eqref{nabla-var} and \eqref{nabla-psi}, we can estimate $\nabla^i\psi$.  We see from \eqref{nabla-psi} that \[|\nabla\psi|\leq C|T|\leq CK^{\frac{1}{2}}.\]
Then from \eqref{nabla-var} and \eqref{nabla-psi} we schematically have $$\nabla^2\psi=\nabla T*\varphi+T*T*\psi$$ and hence
\[
|\nabla^2\psi|\leq C(|\nabla T|+|T|^2)\leq CK.
\]
Using the same equations we see that 
\[
\nabla^3\psi=\nabla^2T*\varphi+\nabla T*T*\psi+T*T*T*\varphi
\]
schematically, and thus by hypothesis
\[
|\nabla^3\psi|\leq C(|\nabla^2T|+|\nabla T||T|+|T|^3)\leq C(Kt^{-\frac{1}{2}}+K^{\frac{3}{2}}).
\]
A straightforward induction then shows that for $i\geq 2$ we have
\begin{equation}\label{na^i-psi}
  |\nabla^i\psi|\leq CK\sum_{j=0}^{i-2}K^{\frac j2}t^{\frac {j-i+2}2}.
\end{equation}
Combining \eqref{line3-1}--\eqref{na^i-psi}, using 
\eqref{shi-1} for $0\leq j<k$ and the assumption $tK\leq 1$, the third line of \eqref{dt-na^k-T} can be estimated by
\begin{align*}
  |\sum_{i=0}^{k+1}\nabla^{k+1} T&*\nabla^{k+1-i}(Rm*T)*\nabla^i\psi|\\
  &\leq   |\nabla^{k+1}T*(\nabla^{k+1}Rm*T+\nabla^kRm*\nabla T)*\psi|\\
   &\quad+ |\nabla^{k+1}T*\nabla^kRm*T*\nabla\psi|+CK^2t^{-\frac k2}|\nabla^{k+1}T|,
\end{align*}
where the last term arises from the estimated terms in \eqref{thm4.2.eq.i1b}, \eqref{thm4.2.eq.i0b} and \eqref{na^i-psi}.
We can estimate the last line of \eqref{dt-na^k-T} similarly. We conclude that
\begin{align}\label{dt-na^k-T-2}
  \frac{\pt}{\pt t}|\nabla^{k+1}T|^2&\leq  \Delta|\nabla^{k+1} T|^2-2|\nabla^{k+2}T|^2+CK^2t^{-\frac k2}|\nabla^{k+1}T|\nonumber\\
  &\quad+CK^{\frac 12}|\nabla^{k+1}T|(|\nabla^{k+1}Rm|+|\nabla^{k+2}T|)\nonumber\\
  &\quad+CK(|\nabla^{k+1}T|^2+|\nabla^{k+1}T||\nabla^kRm|)\nonumber\\
  &\leq \Delta|\nabla^{k+1} T|^2-2|\nabla^{k+2}T|^2+CK^3t^{-k}\nonumber\\
  &\quad +CK^{\frac 12}|\nabla^{k+1}T|(|\nabla^{k+1}Rm|+|\nabla^{k+2}T|)\nonumber\\
&\quad+CK(|\nabla^{k+1}T|^2+|\nabla^kRm|^2),
\end{align}
where we again used Young's inequality to estimate
\begin{align*}
2K^2t^{-\frac k2}|\nabla^{k+1}T|&\leq K^3t^{-k}+K|\nabla^{k+1}T|^2,\\
2|\nabla^{k+1}T||\nabla^kRm|&\leq |\nabla^{k+1}T|^2+|\nabla^kRm|^2.
\end{align*}
Combining \eqref{dt-|na^kRm|-2} and \eqref{dt-na^k-T-2}, we have
\begin{align}\label{dt-|na^kLambda|.1}
  \frac{\pt}{\pt t}(|\nabla^kRm|^2+|\nabla^{k+1}T|^2) &
 \leq \Delta (|\nabla^kRm|^2+|\nabla^{k+1}T|^2)+CK^{3}t^{-k}\nonumber\\
 &\quad -2(|\nabla^{k+1}Rm|^2+|\nabla^{k+2}T|^2)\nonumber\\
 &\quad+CK^{\frac{1}{2}}|\nabla^kRm||\nabla^{k+2}T|\nonumber\\
 &\quad +CK^{\frac{1}{2}}|\nabla^{k+1}T|(|\nabla^{k+1}Rm|+|\nabla^{k+2}T|)\nonumber\\
&\quad  +CK(|\nabla^kRm|^2+|\nabla^{k+1}T|^2).
  \end{align}
Using Young's inequality once again, we know that for any $\epsilon>0$ we have
\begin{align*}
2K^{\frac{1}{2}}|\nabla^kRm||\nabla^{k+2}T|&\leq \frac{1}{\epsilon}K|\nabla^kRm|^2+\epsilon|\nabla^{k+2}T|^2,\\
2K^{\frac{1}{2}}|\nabla^{k+1}T|&(|\nabla^{k+1}Rm|+|\nabla^{k+2}T|)\\
&\leq \frac{2}{\epsilon}K|\nabla^{k+1}T|^2+\epsilon(|\nabla^{k+1}Rm|^2+|\nabla^{k+2}T|^2).
\end{align*}
 We deduce from these estimates and \eqref{dt-|na^kLambda|.1} that, by choosing $\epsilon>0$ sufficiently small (depending on $C$), we  have
  \begin{align}\label{dt-|na^kLambda|}
  \frac{\pt}{\pt t}(|\nabla^kRm|^2+|\nabla^{k+1}T|^2)  &\leq \Delta (|\nabla^kRm|^2+|\nabla^{k+1}T|^2)+CK^{3}t^{-k}\nonumber\\
  &\quad -|\nabla^{k+1}Rm|^2-|\nabla^{k+2}T|^2\nonumber\\
  &\quad +CK(|\nabla^kRm|^2+|\nabla^{k+1}T|^2).
\end{align}

Given these calculations, we now define
 \begin{align}\label{fk-def}
   f_k=&t^k(|\nabla^kRm|^2+|\nabla^{k+1}T|^2)\nonumber\\
   &\qquad +\beta_k\sum_{i=1}^k\alpha_i^kt^{k-i}(|\nabla^{k-i}Rm|^2+|\nabla^{k+1-i}T|^2),
 \end{align}
 for some constants $\beta_k$ to be determined later and $\alpha_i^k=\frac{(k-1)!}{(k-i)!}$. Assuming \eqref{shi-1} holds for all $1\leq i<k$, then by a similar calculation to those leading to \eqref{dt-|na^kLambda|}, we have
 \begin{align}\label{dt-na^iLambda|}
  \frac{\pt}{\pt t}(|\nabla^{k-i}Rm|^2+|\nabla^{k+1-i}T|^2) &\leq  \Delta (|\nabla^{k-i}Rm|^2+|\nabla^{k+1-i}T|^2)+CK^{3}t^{i-k}\nonumber\\
  &\quad -|\nabla^{k+1-i}Rm|^2-|\nabla^{k+2-i}T|^2,
\end{align}
where here we do not require the corresponding last term in \eqref{dt-|na^kLambda|}, since by assumption \eqref{shi-1} holds, so we have 
\[
CK(|\nabla^{k-i}Rm|^2+|\nabla^{k-i+1}T|^2)\leq CK^3t^{-(k-i)}.
\]
 From \eqref{dt-|na^kLambda|} and \eqref{dt-na^iLambda|}, we may calculate
 \begin{align*}
   \frac{\pt}{\pt t}f_k&\leq t^k \frac{\pt}{\pt t}(|\nabla^kRm|^2+|\nabla^{k+1}T|^2)+kt^{k-1} (|\nabla^kRm|^2+|\nabla^{k+1}T|^2)\\
   &\quad +\beta_k\sum_{i=1}^k\alpha_i^kt^{k-i}\frac{\pt}{\pt t}(|\nabla^{k-i}Rm|^2+|\nabla^{k-i+1}T|^2)\\
   &\quad+\beta_k\sum_{i=1}^k(k-i)\alpha_i^kt^{k-i-1}(|\nabla^{k-i}Rm|^2+|\nabla^{k-i+1}T|^2)\displaybreak[0]\\
&\leq t^k\Delta (|\nabla^kRm|^2+|\nabla^{k+1}T|^2)+CK^{3}\\
 &\quad -t^k(|\nabla^{k+1}Rm|^2+|\nabla^{k+2}T|^2)\\
 &\quad +(CKt^k+kt^{k-1})(|\nabla^kRm|^2+|\nabla^{k+1}T|^2)\displaybreak[0]\\
  &\quad +\beta_k\sum_{i=1}^k\alpha_i^kt^{k-i}\Delta (|\nabla^{k-i}Rm|^2+|\nabla^{k+1-i}T|^2)+CK^{3}\alpha_i^k\\
  &\quad -\beta_k\sum_{i=1}^k\alpha_i^kt^{k-i}(|\nabla^{k+1-i}Rm|^2+|\nabla^{k+2-i}T|^2)\\
  &\quad +\beta_k\sum_{i=1}^k(k-i)\alpha_i^kt^{k-i-1}(|\nabla^{k-i}Rm|^2+|\nabla^{k-i+1}T|^2).\displaybreak[0]
 \end{align*}
 Collecting terms we see that
 \begin{align}\label{fk-evl}
   \frac{\pt}{\pt t}f_k   &\leq \Delta f_k+(kt^{k-1}+CKt^k-\beta_kt^{k-1})(|\nabla^kRm|^2+|\nabla^{k+1}T|^2)\nonumber\\
   &\quad +\beta_k\sum_{i=1}^{k-1}(\alpha_i^k(k-i)-\alpha_{i+1}^k)t^{k-i-1}(|\nabla^{k-i}Rm|^2+|\nabla^{k+1-i}T|^2)\nonumber\\
   &\quad +(C+C\beta_k\sum_{i=1}^k\alpha_i^k)K^3\nonumber\displaybreak[0]\\
   &\leq \Delta f_k+CK^3,
 \end{align}
where we used the facts $\alpha_i^k(k-i)-\alpha_{i+1}^k=0$, $Kt\leq 1$ and chose $\beta_k$ sufficiently large.  Since $f_k(0)=\beta_k\alpha_k^k(|Rm|^2+|\nabla T|^2)\leq \beta_k\alpha_k^k K^2$, applying the maximum principle to \eqref{fk-evl} gives
 \begin{align*}
   \sup_{x\in M}f_k(x,t) &\leq  \beta_k\alpha_k^k K^2+CtK^3  \leq  CK^2
 \end{align*}
 Then from the definition of $f_k$, we obtain that
 \begin{equation*}
   |\nabla^k Rm|+|\nabla^{k+1}T|\leq CKt^{-\frac k2}.
 \end{equation*}
 This completes the inductive step and finishes the proof of Theorem \ref{thm-shi}.
\endproof

From Proposition \ref{prop-Rm-T^2}, we know the assumption $\Lambda(x,t)\leq K$ in Theorem \ref{thm-shi} is reasonable, since $\Lambda(x,t)$ can not blow up quickly along the flow
. Note that the estimate \eqref{shi-1} blows up as $t$ approaches zero, but the short-time existence result (Theorem \ref{thm-bryant-xu}) already bounds all derivatives of $Rm$ and $T$ for a short time. In fact, when
 $\Lambda(x,t)\leq K$, from \eqref{evl-d-Rm-T} we have
\begin{align*}
  \frac d{dt}\max_{M_t}(|\nabla Rm|^2+|\nabla^2T|^2)\leq   CK\max_{M_t}(|\nabla Rm|^2+|\nabla^2T|^2)+CK^4 ,
\end{align*}
which gives us
\begin{align*}
  \max_{M_t}(|\nabla Rm|^2+|\nabla^2T|^2)\leq  e^{CKt}(\max_{M_0}(|\nabla Rm|^2+|\nabla^2T|^2)+K^3)-K^3
\end{align*}
for  $t\in [0,\epsilon]$ if $\epsilon$ sufficiently small.Using \eqref{dt-|na^kRm|-2}--\eqref{dt-na^k-T} and the maximum principle, we may deduce that such estimates also hold for higher order derivatives, so $\max_{M_t}(|\nabla^k Rm|^2+|\nabla^{k+1}T|^2)$ is also bounded in terms of its initial value and $K$ for a short time.

\begin{rem}
One can ask whether the growth of the constants $C_k$ in Theorem \ref{thm-shi} can be controlled in terms of $k$.  The authors show this is indeed the 
case in \cite{Lotay-Wei-ra} and as a consequence deduce that the Laplacian flow is 
real analytic in space for each fixed positive time.
\end{rem}

We can also prove a local version of Theorem \ref{thm-shi}, stated below.  Since we already established evolution inequalities for the relevant geometric quantities in the proof of Theorem \ref{thm-shi}, the proof just follows by applying a similar argument to Shi \cite{shi} (see also \cite{ha95}) in the Ricci flow case, so we omit it.
\begin{thm}[Local derivative estimates]\label{thm-shi-completecase}
Let $K>0$ and $r>0$.  Let $M$ be a $7$-manifold,
$p\in M$, and $\varphi(t), t\in [0,\frac 1K]$ be a solution to the Laplacian flow \eqref{Lap-flow-def} for closed $\GG_2$ structures on an open neighborhood $U$ of $p$ containing $B_{g(0)}(p,r)$ as a compact subset.

For any $k\in \mathbb{N}$, there exists a constant $C=C(K,r,k)$ such that if $\Lambda(x,t)\leq K$ for all $x\in U$  and $t\in [0,\frac 1K]$, then for all $y\in B_{g(0)}(p,r/2)$ and $t\in [0,\frac 1K]$, we have
\begin{equation}\label{shi-1-complete}
  |\nabla^kRm|+|\nabla^{k+1}T|\leq {C(K,r,k)}
{t^{-\frac k2}}.
\end{equation}
\end{thm}

\begin{rem}
By Proposition \ref{prop-nabla-T} and Corollary \ref{scalar-cor}, we can bound $|\nabla T|$ using bounds on $|Rm|$, and hence we
can, if we wish, replace the bound on $\Lambda$ in \eqref{Lambda-t-def} in Theorems \ref{thm-shi} and \ref{thm-shi-completecase}
by a bound on $|Rm|$.
\end{rem}

\section{Long time existence I}\label{sec:longtime-I}
Given an initial closed $\GG_2$ structure $\varphi_0$, there exists a solution $\varphi(t)$ of Laplacian flow on a maximal time interval $[0,T_0)$, where maximal means that either $T_0=\infty$, or that $T_0<\infty$ but there do not exist $\epsilon>0$ and a smooth Laplacian flow $\tilde{\varphi}(t)$ for $t\in [0,T_0+\epsilon)$ such that $\tilde{\varphi}(t)=\varphi(t)$ for $t\in [0,T_0)$. We call $T_0$ the singular time.

In this section, we use the global derivative estimates \eqref{shi-0} for $Rm$ and $\nabla T$ to prove Theorem \ref{mainthm-blowup}, i.e.~$\Lambda(x,t)$ given in \eqref{Lambda-t-def} will blow up at a finite time singularity along the flow. We restate Theorem \ref{mainthm-blowup} below.
\begin{thm}\label{thm-blowup}
If $\varphi(t)$ is a solution to the Laplacian flow \eqref{Lap-flow-def} for closed $\GG_2$ structures on a compact manifold $M^7$ in a maximal time interval $[0,T_0)$ and the maximal time $T_0<\infty$, then $\Lambda(t)$ given in \eqref{Lambda-t-def-2} satisfies
\begin{equation}\label{thm-bu-eqn1}
  \lim_{t\nearrow T_0}\Lambda(t)= \infty.
\end{equation}
Moreover, we have a lower bound on the blow-up rate,
\begin{equation}\label{thm-bu-eqn2}
  \Lambda(t)\geq \frac{C}{T_0-t}
\end{equation}
for some constant $C>0$.
\end{thm}
\proof
Suppose the solution $\varphi(t)$ exists on a  maximal finite time interval $[0,T_0)$. We first prove, by contradiction, that
\begin{equation}\label{pf-thm5.1-0}
  \limsup_{t\nearrow T_0}\Lambda(t)=\infty.
\end{equation}

Suppose \eqref{pf-thm5.1-0} does not hold, so there exists a constant $K>0$ such that
\begin{equation}\label{pf-thm5.1-1}
  \sup_{M\times [0, T_0)}\Lambda(x,t)=\sup_{M\times [0, T_0)}\left(|\nabla T(x,t)|_{g(t)}^2+|Rm(x,t)|_{g(t)}^2\right)^{\frac 12}\leq K,
\end{equation}
where $g(t)$ is the metric determined by $\varphi(t)$. Then, in particular, we have the uniform curvature bound
\begin{equation*}
 \sup_{M\times [0, T_0)}|Rm(x,t)|_{g(t)}\leq K,
\end{equation*}
which implies that
\begin{equation*}
  \sup_{M\times [0, T_0)}\biggl|\frac{\pt}{\pt t}g_{ij}\biggr|_{g(t)}=\sup_{M\times [0, T_0)}\biggl|-2R_{ij}-\frac 23|T|^2g_{ij}-4T_i^{\,\,k}T_{kj}\biggr|_{g(t)}\leq CK.
\end{equation*}
(Keep in mind that $|T|^2=-R$). Then all the metrics $g(t)$ $(0\leq t<T_0)$ are uniformly equivalent (see e.g.~\cite[Theorem 14.1]{ha82}), as $T_0<\infty$.  We also have from
\eqref{T-tau-eq}, \eqref{lap-varphi} and \eqref{pf-thm5.1-1}:
\begin{align}\label{pf-thm5.1-2}
  \biggl|\frac{\pt}{\pt t}\varphi\biggr|_{g(t)}=&\bigl|\Delta_{\varphi}\varphi\bigr|_{g(t)}\leq CK,
\end{align}
for some uniform positive constant $C$.

We fix a background metric $\bar{g}=g(0)$, the metric determined by $\varphi(0)$. From \eqref{pf-thm5.1-2} and the uniform equivalence of the metrics $\bar{g}$ and $g(t)$, we have
\begin{align}
  \biggl|\frac{\pt}{\pt t}\varphi\biggr|_{\bar{g}}\leq C\biggl|\frac{\pt}{\pt t}\varphi\biggr|_{g(t)}\leq CK.
\end{align}
For any $0<t_1<t_2<T_0$,
\begin{align}\label{pf-thm5.1-3}
  \bigl|\varphi(t_2)-\varphi(t_1)\bigr|_{\bar{g}}\leq \int_{t_1}^{t_2}\biggl|\frac{\pt}{\pt t}\varphi\biggr|_{\bar{g}}dt\leq CK(t_2-t_1),
\end{align}
which implies that $\varphi(t)$ converges to a 3-form $\varphi(T_0)$ continuously as $t\ra T_0$. We may similarly argue using \eqref{flow-g2} and \eqref{pf-thm5.1-1} that the uniformly equivalent Riemannian metrics $g(t)$ converge continuously to a Riemannian metric $g(T_0)$ as $t\ra T_0$, since all the $g(t)$ are uniformly equivalent to $\bar{g}$.

By \eqref{B-varphi}, for each $t\in [0,T_0)$ we have
\begin{equation}\label{pf-5-8}
  g_t(u,v)vol_{g(t)}=\frac 16(u\lrcorner\varphi(t))\wedge(v\lrcorner\varphi(t))\wedge\varphi(t).
\end{equation}
Let $t\ra T_0$ in \eqref{pf-5-8}.  Recall that we have argued above that $g(t)\to g(T_0)$ which is  a Riemannian metric and thus $\vol_{g(t)}\to\vol_{g(T_0)}$ which is a volume hence. Therefore the left hand side  of \eqref{pf-5-8} 
tends to a positive definite $7$-form valued bilinear form.  Thus, the right-hand side of \eqref{pf-5-8} has a positive definite limit, and thus the limit $3$-form $\varphi(T_0)$ is positive, i.e.~$\varphi(T_0)$ is a $\GG_2$ structure on $M$. Moreover, note that $\d\varphi(t)=0$ for all $t$ means that  the limit $\GG_2$ structure $\varphi(T_0)$ is also closed. In summary, the solution $\varphi(t)$ of the Laplacian flow for closed $\GG_2$ structures can be extended continuously to the time interval $[0,T_0]$.

We now show that the extension is actually smooth, thus obtaining our required contradiction.  We beginning by showing that we can uniformly bound the
derivatives of the metric and 3-form with respect to the background Levi-Civita connection along the flow.

\begin{claim}\label{claim-5-1}
There exist constants $C_m$ for $m\in \mathbb{N}$ such that
\begin{equation*}
  \sup_{M\times [0,T_0)}\biggl|\overline{\nabla}^{(m)}g(t)\biggr|_{\bar{g}}\leq C_m,
\end{equation*}
where $\overline{\nabla}$ is the Levi-Civita connection with respect to $\bar{g}$.
\end{claim}
\begin{proof}[Proof of Claim \ref{claim-5-1}]
Since 
$g(t)$ evolves by \eqref{flow-g2}, the proof of the claim 
is similar to the Ricci flow case, see e.g.~\cite[\S 6.7]{Chow-Knopf}, so we omit the detail here.
\end{proof}

\begin{claim}\label{claim-5-2}
There exist constants $C_m$ for $m\in \mathbb{N}$ such that
\begin{equation*}
  \sup_{M\times [0,T_0)}\biggl|\overline{\nabla}^{(m)}\varphi(t)\biggr|_{\bar{g}}\leq C_m.
\end{equation*}
\end{claim}
\begin{proof}[Proof of Claim \ref{claim-5-2}]
We begin with $m=1$. At any $(x,t)\in M\times [0,T_0)$,
\begin{align}\label{pf-claim-1}
  \frac{\pt}{\pt t}\overline{\nabla}\varphi=& \overline{\nabla}\frac{\pt}{\pt t} \varphi=\overline{\nabla}\Delta_{\varphi}\varphi\nonumber\\
  =&\nabla\Delta_{\varphi}\varphi+A*\Delta_{\varphi}\varphi,
\end{align}
where we denote $A=\overline{\nabla}-\nabla$ as the difference of two connections, which is a tensor. Then in a fixed chart around $x$ we have
\begin{align*}
  \frac{\pt}{\pt t} A_{ij}^k=&-\frac{\pt}{\pt t}\Gamma_{ij}^k\\
  =&-\frac 12g^{kl}(\nabla_i(\frac{\pt}{\pt t}g_{jl})+\nabla_j(\frac{\pt}{\pt t}g_{il})-\nabla_l(\frac{\pt}{\pt t}g_{ij})),
\end{align*}
so
\begin{align*}
  \frac{\pt}{\pt t}A=-g^{-1}\nabla(Ric+T*T).
\end{align*}
Integrating in time $t$, we get
\begin{align}\label{A-est}
  |A(t)|_{\bar{g}}&\leq  |A(0)|_{\bar{g}}+\int_0^t\bigl|\frac{\pt}{\pt s} A\bigr|_{\bar{g}}ds\nonumber\\
  &\leq |A(0)|_{\bar{g}}+C\int_0^t\bigl|\frac{\pt}{\pt s} A\bigr|_{g(s)}ds\nonumber\\
  &\leq |A(0)|_{\bar{g}}+C(|\nabla Ric|+|\nabla T||T|)t\leq C,
\end{align}
since $t<T_0$ is finite and $|\nabla Ric|+|\nabla T||T|$ is bounded by \eqref{shi-1} and \eqref{pf-thm5.1-1}. Furthermore,
we can derive from Claim \ref{claim-5-1} that
\begin{equation}\label{A-est2}
  |\overline{\nabla}^kA(t)|_{\bar{g}}\leq C\quad\textrm{ for } 0\leq k\leq m-1.
\end{equation}
From \eqref{shi-1}, \eqref{pf-claim-1} and \eqref{A-est}, we get
\begin{align*}
  \biggl|\frac{\pt}{\pt t}\overline{\nabla}\varphi \biggr|_{\bar{g}}\leq  C,
\end{align*}
and then
\begin{align}\label{pf-claim-2}
  |\overline{\nabla}\varphi(t)|_{\bar{g}}\leq  |\overline{\nabla}\varphi(0)|_{\bar{g}}+\int_0^t\biggl|\frac{\pt}{\pt s}\overline{\nabla}\varphi(s) \biggr|_{\bar{g}}ds
  \leq |\overline{\nabla}\varphi(0)|_{\bar{g}}+CT_0,
\end{align}
which gives the $m=1$ case of Claim \ref{claim-5-2}.

For $m\geq 2$, we can prove by induction that
\begin{align}\label{pf-claim-3}
  \biggl|\frac{\pt}{\pt t}\overline{\nabla}^m\varphi\biggr|_{\bar{g}}&=\bigl|\overline{\nabla}^m\Delta_{\varphi}\varphi\bigr|_{\bar{g}}\nonumber\\
  &\leq  C \sum_{i=0}^m|A|^i|\nabla^{m-i}\Delta_{\varphi}\varphi|+C\sum_{i=1}^{m-1}|\overline{\nabla}^iA||\nabla^{m-1-i}\Delta_{\varphi}\varphi|.
\end{align}
It then follows from \eqref{shi-1}, \eqref{A-est2} and \eqref{pf-claim-3} that
\begin{align}\label{pf-claim-4}
  \biggl|\frac{\pt}{\pt t}\overline{\nabla}^m\varphi\biggr|_{\bar{g}}=&\bigl|\overline{\nabla}^m\Delta_{\varphi}\varphi\bigr|_{\bar{g}}\leq C.
\end{align}
Then Claim \ref{claim-5-2} follows from \eqref{pf-claim-4} by integration.
\end{proof}

Now we continue the proof of Theorem \ref{thm-blowup}. We have that a continuous limit of closed $\GG_2$ structures $\varphi(T_0)$ exists, and in a fixed local coordinate chart $\mathcal{U}$
it satisfies
\begin{equation}\label{pf-thm5.1-4}
  \varphi_{ijk}(T_0)=\varphi_{ijk}(t)+\int_{t}^{T_0}(\Delta_{\varphi(s)}\varphi(s))_{ijk}ds.
\end{equation}
Let $\alpha=(a_1,\cdots,a_r)$ be any multi-index with $|\alpha|=m\in \mathbb{N}$. By Claim \ref{claim-5-2} and \eqref{pf-claim-4}, we have that
\begin{equation}
  \frac{\pt^m}{\pt x^{\alpha}}\varphi_{ijk}\quad\textrm{ and }\quad \frac{\pt^m}{\pt x^{\alpha}}(\Delta_{\varphi}\varphi)_{ijk}
\end{equation}
are uniformly bounded on $\mathcal{U}\times [0,T_0)$. Then from \eqref{pf-thm5.1-4} we have  that $\frac{\pt^m}{\pt x^{\alpha}}\varphi_{ijk}(T_0)$ is bounded
on $\mathcal{U}$ and hence $\varphi(T_0)$ is a smooth closed $\GG_2$ structure. Moreover,
\begin{equation}
  \biggl|\frac{\pt^m}{\pt x^{\alpha}}\varphi_{ijk}(T_0)-\frac{\pt^m}{\pt x^{\alpha}}\varphi_{ijk}(t)\biggr|\leq C(T_0-t),
\end{equation}
and thus $\varphi(t)\ra \varphi(T_0)$ uniformly in any $C^m$ norm as $t\ra T_0$, $m\geq 2$.

Now, Theorem \ref{thm-bryant-xu} gives a solution $\bar{\varphi}(t)$ of the Laplacian flow \eqref{Lap-flow-def} 
with $\bar{\varphi}(0)=\varphi(T_0)$ for a short time $0\leq t<\epsilon$.  Since $\varphi(t)\ra \varphi(T_0)$ smoothly as $t\ra T_0$, this gives that
\begin{equation*}
  \tilde{\varphi}(t)=\left\{\begin{array}{cl}
                              \varphi(t), & 0\leq t<T_0, \\
                               \bar{\varphi}(t-T_0),& T_0\leq t<T_0+\epsilon.
                            \end{array}\right.
\end{equation*}
is a solution of 
\eqref{Lap-flow-def} with initial value $\tilde{\varphi}(0)=\varphi(0)$ for $t\in [0,T_0+\epsilon)$, which is a contradiction to the maximality of $T_0$. So we have
\begin{equation}\label{pf-thm5.1-5}
  \limsup_{t\nearrow T_0}\Lambda(t)=\infty.
\end{equation}

We now prove \eqref{thm-bu-eqn1} by replacing the $\limsup$ in \eqref{pf-thm5.1-5} by $\lim$. Suppose, for a contradiction, that
\eqref{thm-bu-eqn1} does not hold.  Then there exists a sequence $t_i\nearrow T_0$ such that $\Lambda(t_i)\leq K_0$ for some constant $K_0$. By the doubling time estimate in Proposition \ref{prop-Rm-T^2},
\begin{equation}
 \Lambda(t)\leq 2\Lambda(t_i)\leq 2K_0,
\end{equation}
for all $t\in [t_i, \min\{T_0,t_i+\frac 1{CK_0}\})$. Since $t_i\ra T_0$, for sufficiently large $i$ we have $t_i+\frac 1{CK_0}\geq T_0$. Therefore, for all $i$
sufficiently large,
\begin{equation}
 \sup_{M\times [t_i,T_0)}\Lambda(x,t)\leq 2K_0,
\end{equation}
but we already showed above that this leads to a contradiction to the maximality of $T_0$.
This completes the proof of \eqref{thm-bu-eqn1}.

We conclude by proving the lower bound of the blow-up rate \eqref{thm-bu-eqn2}. Applying the maximum principle to \eqref{evl-Rm-T} we have
\begin{equation*}
  \frac d{dt}\Lambda(t)^2\leq C\Lambda(t)^3,
\end{equation*}
which implies that
\begin{equation}\label{pf-thm5.1-6}
  \frac d{dt}\Lambda(t)^{-1}\geq -\frac C2.
\end{equation}
We already proved that $\lim\limits_{t\ra T_0}\Lambda(t)=\infty$, so we have
\begin{equation}\label{pf-thm5.1-7}
  \lim_{t\ra T_0}\Lambda(t)^{-1}=0.
\end{equation}
Integrating \eqref{pf-thm5.1-6} from $t$ to $t'\in (t,T_0)$ and passing to the limit $t'\ra T_0$, we obtain
\begin{equation*}
  \Lambda(t)\geq \frac 2{C(T_0-t)}.
\end{equation*}
This completes the proof of Theorem \ref{thm-blowup}.
\endproof

Combining Theorem \ref{thm-blowup} and Proposition \ref{prop-Rm-T^2} gives us the following corollary on the estimate of the minimal existence time.
\begin{corr}\label{exist-time-cor}
Let $\varphi_0$ be a closed $\GG_2$ structure on a compact manifold $M^7$ with
\begin{equation*}
  \Lambda_{\varphi_0}(x)=\left(|\nabla T(x)|^2+|Rm(x)|^2\right)^{\frac 12}\leq K
\end{equation*}
on $M$, for some constant $K$.  Then the unique solution $\varphi(t)$ of the Laplacian flow \eqref{Lap-flow-def} starting from $\varphi_0$ exists at least for time $t\in [0,\frac 1{CK}]$, where $C$ is a uniform constant as in Proposition \ref{prop-Rm-T^2}. 
\end{corr}

\section{Uniqueness}\label{sec:unique} 

In this section, we will use the ideas in \cite{kot1,kot2} to prove Theorem \ref{mainthm-uniq}: the forwards and backwards uniqueness property of the Laplacian flow.

If $\varphi(t)$, $\tilde{\varphi}(t)$ are two smooth solutions to the flow \eqref{Lap-flow-def}  on a compact manifold $M^7$ for $t\in [0,\epsilon], \epsilon>0$, there exists a constant $K_0$ such that
\begin{equation}\label{Lambda-tLambda-bounds}
  \sup_{M\times [0,\epsilon]}\left(\Lambda(x,t)+\widetilde{\Lambda}(x,t)\right)\leq K_0,
\end{equation}
adopting the obvious notation for quantities determined by $\varphi(t)$ and $\tilde{\varphi}(t)$.
By the Shi-type estimate \eqref{shi-0}, there is a constant $K_1$ depending on $K_0$ such that
\begin{equation}\label{curvature-bounds}
  \sum_{k=0}^2\left(|\nabla^kRm|_{g(t)}+|\tilde{\nabla}^k\widetilde{Rm}|_{\tilde{g}(t)}\right)+\sum_{k=0}^3\left(|\nabla^kT|_{g(t)}+|\tilde{\nabla}^k\widetilde{T}|_{\tilde{g}(t)}\right)\leq K_1
\end{equation}
on $M\times [0,\epsilon]$.  The uniform curvature bounds from \eqref{curvature-bounds} imply that $g(t)$ and $\tilde{g}(t)$ are uniformly equivalent on $M\times [0,\epsilon]$, so the norms $|\cdot|_{g(t)}$ and $|\cdot|_{\tilde{g}(t)}$ only differ by a uniform constant on $M\times [0,\epsilon]$. We deduce the following from \eqref{curvature-bounds}.

\begin{lem}\label{lem-uniq-1}
The inverse $\tilde{g}^{-1}$ of the metric $\tilde{g}$, $\widetilde{\nabla}^k\widetilde{Rm}$ for $0\leq k\leq 2$ and $\widetilde{\nabla}^k\widetilde{T}$ for
$0\leq k\leq 3$ are uniformly bounded with respect to $g(t)$ on $[0,\epsilon]$.
\end{lem}
We will use this fact frequently in the following calculation. We continue to let $A*B$ denote some contraction of two tensors $A,B$ using $g(t)$.
We also recall that if $\varphi(s)=\tilde{\varphi}(s)$ for some $s\in [0,\epsilon]$, then the induced metrics also satisfy $g(s)=\tilde{g}(s)$.

\subsection{Forward uniqueness}\label{sec:for-uniq}
We begin by showing forward uniqueness of the flow as claimed in Theorem \ref{mainthm-uniq}; namely, that if $\varphi(s)=\tilde{\varphi}(s)$ for some $s\in[0,\epsilon]$ then $\varphi(t)=\tilde{\varphi}(t)$
for all $t\in[s,\epsilon]$.
The strategy to show this, inspired by \cite{kot2}, is to define an energy quantity $\mathcal{E}(t)$ by
\begin{align}\label{energy-def}
  \mathcal{E}(t)=&\int_M\biggl(|\phi(t)|^2_{g(t)}+|h(t)|^2_{g(t)}+|A(t)|^2_{g(t)}+|U(t)|^2_{g(t)}\nonumber\\
  &\qquad+|V(t)|^2_{g(t)}+|S(t)|^2_{g(t)}\biggr) 
vol_{g(t)},
\end{align}
and show that $\mathcal{E}(t)$ satisfies a differential inequality which implies that $\mathcal{E}(t)$ vanishes identically if $\mathcal{E}(0)=0$ initially. Here in the definition \eqref{energy-def} of $\mathcal{E}(t)$,
\begin{gather*}
\phi=\varphi-\tilde{\varphi}, \quad h=g-\tilde{g},\quad A=\nabla-\widetilde{\nabla},\\
U=T-\widetilde{T},\quad V=\nabla T-\widetilde{\nabla}\widetilde{T}, \quad S=Rm-\widetilde{Rm}.
\end{gather*}  In local coordinates, we have $A_{ij}^k=\Gamma_{ij}^k-\widetilde{\Gamma}_{ij}^k$, $U_{ij}=T_{ij}-\widetilde{T}_{ij}$, $V_{ijk}=\nabla_iT_{jk}-\widetilde{\nabla}_i\widetilde{T}_{jk}$ and $S_{ijk}^{\quad l}=R_{ijk}^{\quad l}-\widetilde{R}_{ijk}^{\quad l}$.

We begin by deriving inequalities for the derivatives of the quantities in the integrand defining $\mathcal{E}(t)$.

\begin{lem}\label{lem-forward-uniq-1}
We have the following inequalities:
\begin{align}
   \biggl|\frac{\pt}{\pt t}\phi(t)\biggr|_{g(t)} &\leq  C(|V(t)|_{g(t)}+|A(t)|_{g(t)}); \label{evl-phit-0}\\
   \biggl|\frac{\pt}{\pt t}h(t)\biggr|_{g(t)} &\leq  C(|S(t)|_{g(t)}+|h(t)|_{g(t)}+|U(t)|_{g(t)});\label{evl-ht-0}\\
  \left|\frac{\pt}{\pt t}A(t)\right|_{g(t)} &\leq  C\big(|A(t)|_{g(t)}+|h(t)|_{g(t)}\nonumber\\
&\qquad\;+|U(t)|_{g(t)}+|V(t)|_{g(t)}+|\nabla S(t)|_{g(t)}\big);\label{evl-At-0}\displaybreak[0]\\
   \left|\frac{\pt}{\pt t}U(t)\right|_{g(t)}&\leq C\big(|\phi(t)|_{g(t)}+|A(t)|_{g(t)}+|U(t)|_{g(t)}+|S(t)|_{g(t)} \nonumber \\
   &\qquad\; +|\nabla V(t)|_{g(t)}+|V(t)|_{g(t)}\big);\label{evl-Ut-0}
\end{align}
\begin{align}\label{evl-Vt-0}
   \bigg|\frac{\pt}{\pt t}V(t)&-\Delta V(t)-\Div \mathcal{V}(t)\bigg|_{g(t)}\nonumber\\
&\leq C\big(|V(t)|_{g(t)}+|A(t)|_{g(t)}+|U(t)|_{g(t)}+|S(t)|_{g(t)} \nonumber \\   & \qquad \;
+|h(t)|_{g(t)}+|\phi(t)|_{g(t)}+|\nabla S(t)|_{g(t)}+|\nabla V(t)|_{g(t)}\big),
\end{align}
where $\mathcal{V}$ given by $\mathcal{V}^a_{\;\,ijk}=(g^{ab}\nabla_b-\tilde{g}^{ab}\widetilde{\nabla}_b)\widetilde{\nabla}_i\widetilde{T}_{jk}$ satisfies
\begin{equation*}
  |\mathcal{V}(t)|_{g(t)}\leq C(|h(t)|_{g(t)}+|A(t)|_{g(t)});
\end{equation*}
and
\begin{align}\label{evl-St-0}
   \bigg|\frac{\pt}{\pt t} &S(t)-\Delta S(t)-\Div \mathcal{S}(t)\bigg|_{g(t)}\nonumber\\
&\leq C\big(|V(t)|_{g(t)}+|A(t)|_{g(t)}+|U(t)|_{g(t)}+|S(t)|_{g(t)}+|\nabla V(t)|_{g(t)}\big),
\end{align}
where $\mathcal{S}_{\;\,ijk}^{a\quad\! l}=(g^{ab}\nabla_b-\tilde{g}^{ab}\widetilde{\nabla}_b)\widetilde{Rm}_{ijk}^{\quad l}$ satisfies
\begin{equation*}
  |\mathcal{S}(t)|_{g(t)}\leq C(|h(t)|_{g(t)}+|A(t)|_{g(t)}).
\end{equation*}
In the above inequalities, $\nabla$, $\Delta$ and  $\Div$ are the Levi-Civita connection, Laplacian  and divergence on $M$ with respect to $g(t)$ and
 $C$ denotes uniform constants depending on $K_1$ given in \eqref{curvature-bounds}. 
\end{lem}
\proof
We have the following basic facts:
\begin{gather*}
  g^{ij}-\tilde{g}^{ij}=-g^{ik}\tilde{g}^{jl}h_{kl},\;\, \nabla_ih_{jk}=A_{ij}^l\tilde{g}_{lk}+A_{ik}^l\tilde{g}_{jl},\;\,
 \nabla_k\tilde{g}^{ij}=A_{kl}^i\tilde{g}^{lj}+A_{kl}^j\tilde{g}^{il}.
\end{gather*}
The above equations can be expressed schematically as
\begin{equation}\label{g-invers-diff-3}
 g^{-1}-\tilde{g}^{-1}=\tilde{g}^{-1}*h, \quad \nabla h=A*\tilde{g}, \quad \nabla \tilde{g}^{-1}=\tilde{g}^{-1}*A.
\end{equation}
We now calculate the evolution equations of $\phi,h,A,U,S$ on $M\times [0,\epsilon]$.

From the Laplacian flow equation \eqref{Lap-flow-def}  and \eqref{lap-varphi}, we have
\begin{align*}
   \frac{\pt}{\pt t}\phi=  \Delta_{\varphi}\varphi-\Delta_{\tilde{\varphi}}\tilde{\varphi} =d\tau-d\widetilde{\tau}.
\end{align*}
This satisfies the estimate
\begin{align*}
  \biggl|\frac{\pt}{\pt t}\phi\biggr|_{g(t)} &\leq  C|\nabla U(t)|_{g(t)}=C|\nabla T-\widetilde{\nabla}\widetilde{T}+(\widetilde{\nabla}-\nabla)\widetilde{T}|_{g(t)}\nonumber \\
  &\leq C|V(t)|_{g(t)}+C|A(t)|_{g(t)}|\widetilde{T}|_{g(t)}\leq C(|V(t)|_{g(t)}+|A(t)|_{g(t)}),
\end{align*}
where we used the fact that $|\widetilde{T}|_{g(t)}$ is bounded due to Lemma \ref{lem-uniq-1}. We thus obtain the inequality \eqref{evl-phit-0}.

From the evolution equation \eqref{flow-g2} for the metric, we have in coordinates
\begin{align}\label{evl-ht-1}
   \frac{\pt}{\pt t}h_{ik} &=-2(R_{ik}-\widetilde{R}_{ik})-\frac 23(|T|^2_{g(t)}g_{ik}-|\widetilde{T}|^2_{\tilde{g}(t)}\tilde{g}_{ik})-4(T_i^{\,\,j}T_{jk}-\widetilde{T}_i^{\,\,j}\widetilde{T}_{jk})\nonumber\\
  &=  -2S_{ijk}^{\quad j}-\frac 23|\widetilde{T}|^2_{\tilde{g}(t)}h_{ij}-\frac 23(|T|^2_{g(t)}-|\widetilde{T}|^2_{\tilde{g}(t)}){g}_{ij}\nonumber\\  &\quad
-4(g^{jl}T_{il}T_{jk}-\tilde{g}^{jl}\widetilde{T}_{il}\widetilde{T}_{jk}).
\end{align}
Since
\begin{align*}
|T|^2_{g(t)}-|\widetilde{T}|^2_{\tilde{g}(t)}&= T_{ij}T_{kl}g^{ik}g^{jl}- \widetilde{T}_{ij}\widetilde{T}_{kl}\tilde{g}^{ik}\tilde{g}^{jl}\\
  &= (T_{ij}+\widetilde{T}_{ij})U_{kl}g^{ik}g^{jl}+\widetilde{T}_{ij}\widetilde{T}_{kl}(g^{ik}+\tilde{g}^{ik})(g^{jl}-\tilde{g}^{jl})\\
  &=(T+\widetilde{T})*U+\widetilde{T}*\widetilde{T}*(g^{-1}+\tilde{g}^{-1})*h
\end{align*}
and
\begin{align*}
  g^{jl}T_{il}T_{jk}-\tilde{g}^{jl}\widetilde{T}_{il}\widetilde{T}_{jk}&=  U_{il}T_{jk}g^{jl}+\widetilde{T}_{il}U_{jk}g^{jl} +(g^{jl}-\tilde{g}^{jl})\widetilde{T}_{il}\widetilde{T}_{jk}\\
   &= (T+\widetilde{T})*U+\widetilde{T}*\widetilde{T}*\tilde{g}^{-1}*h,
\end{align*}
we obtain from \eqref{evl-ht-1} that
\begin{equation}\label{evl-ht}
   \frac{\pt}{\pt t}h=-2\check{S}-\frac 23|\widetilde{T}|_{\tilde{g}}^2h+(T+\widetilde{T})*U*h+(T+\widetilde{T})*U*\tilde{g}+\widetilde{T}*\widetilde{T}*(g^{-1}+\tilde{g}^{-1})*h,
\end{equation}
where $\check{S}_{ik}=S_{ijk}^{\quad j}$.  
Then \eqref{evl-ht-0} follows from \eqref{evl-ht} and Lemma \ref{lem-uniq-1}.

Recall that under the evolution \eqref{flow-g2} of $g(t)$, the connection evolves by
\begin{equation*}
  \frac{\pt}{\pt t}\Gamma_{ij}^k=\frac 12g^{kl}(\nabla_i\eta_{jl}+\nabla_j\eta_{il}-\nabla_l\eta_{ij}),
\end{equation*}
where schematically
\[
\eta=-2Ric-\frac{2}{3}|T|^2_gg-2T*T.
\]
Thus, the tensor $A_{ij}^k=\Gamma_{ij}^k-\widetilde{\Gamma}_{ij}^k$ satisfies
\begin{align}\label{evl-At}
   \frac{\pt}{\pt t}A&=  \tilde{g}^{-1} *\widetilde{\nabla}(\widetilde{Ric}+\frac 13|\widetilde{T}|^2_{\tilde{g}}\tilde{g}+\widetilde{T}*\widetilde{T})-g^{-1}*\nabla(Ric+\frac 13|T|^2_gg+T*T)\nonumber\\
   &=(\tilde{g}^{-1}-g^{-1})*\widetilde{\nabla}\widetilde{Rm}+(\widetilde{\nabla}-\nabla)*\widetilde{Rm}+g^{-1}*\nabla(\widetilde{Rm}-Rm)\nonumber\\
   &\quad+(\tilde{g}^{-1}-{g}^{-1})*\widetilde{T}*\widetilde{\nabla}\widetilde{T}*\tilde{g}^{-1}+(\widetilde{\nabla}\widetilde{T}-\nabla T)*\widetilde{T}*\tilde{g}^{-1}\nonumber\\
   &\quad+\nabla T*(\widetilde{T}-T)*\tilde{g}^{-1}+\nabla T*T*(\tilde{g}^{-1}-{g}^{-1})\nonumber\\
   &=\tilde{g}^{-1}*h*\widetilde{\nabla}\widetilde{Rm}+A*\widetilde{Rm}+g^{-1}*\nabla S+\tilde{g}^{-1}*h*\widetilde{T}*\widetilde{\nabla}\widetilde{T}*\tilde{g}^{-1}\nonumber\\
   &\quad +V*\widetilde{T}*\tilde{g}^{-1}+\nabla T* U*\tilde{g}^{-1}+\nabla T*T*\tilde{g}^{-1}*h,
\end{align}
which gives \eqref{evl-At-0}.

From the evolution equation \eqref{evl-torsion-1} of $T$, we have
\begin{align*}
  \frac{\pt}{\pt t}U &=\frac{\pt}{\pt t} T-\frac{\pt}{\pt t}\widetilde{T}\\
  &=A*\widetilde{\nabla}\widetilde{T}+\nabla V+S*(\widetilde{T}+\widetilde{T}*\tilde{\psi})+U*(Rm+Rm*\tilde{\psi})\\
  &\quad+{Rm}*{T}*(\tilde{\psi}-\psi)+V*\widetilde{T}*\tilde{\varphi}+\nabla T*U*\tilde{\varphi}\\
  &\quad+\nabla T*T*\phi+U*(T*T+\widetilde{T}*T+\widetilde{T}*\widetilde{T}).
\end{align*}
Noting that
\begin{equation*}
 |\tilde{\psi}-\psi|\leq C|\tilde{\varphi}-\varphi|=C|\phi|,
\end{equation*}
we see that \eqref{evl-Ut-0} follows from the evolution equation for $U$.

We next compute the evolution of $V$ using \eqref{evl-nabla-T1}.  We start by seeing that
\begin{align*}
\Delta(\nabla T)-\widetilde{\Delta}(\widetilde{\nabla} \widetilde{T})&=\nabla_{a}g^{ab}\nabla_b(\nabla T)-\widetilde{\nabla}_a\tilde{g}^{ab}\widetilde{\nabla}_b(\widetilde{\nabla} \widetilde{T})\\
&=\nabla_ag^{ab}\nabla_b(\nabla T-\widetilde{\nabla}\widetilde{T})
+\nabla_a(g^{ab}\nabla_b-\tilde{g}^{ab}\widetilde{\nabla}_b)(\widetilde{\nabla}\widetilde{T})\\
&\quad+(\nabla_a-\widetilde{\nabla}_a)(\tilde{g}^{ab}\widetilde{\nabla}_b(\widetilde{\nabla} \widetilde{T}))\\
&=\Delta V+\nabla_a(g^{ab}\nabla_b\widetilde{\nabla}\widetilde{T}-\tilde{g}^{ab}\widetilde{\nabla}_b\widetilde{\nabla}\widetilde{T})+A*\widetilde{\nabla}^2\widetilde{T}.
\end{align*}
The second terms from \eqref{evl-nabla-T1} give schematically that
\begin{align*}
&\nabla Rm*(T+T*\psi)-\widetilde{\nabla}\widetilde{Rm}*(\tilde{T}+\tilde{T}*\tilde{\psi})\\
&=\nabla Rm*(U+U*\tilde{\psi}+T*(\psi-\tilde{\psi}))
+(\nabla Rm-\widetilde{\nabla}\widetilde{Rm})*(\widetilde{T}+\widetilde{T}*\tilde{\psi})\\
&=\nabla Rm*(U+U*\tilde{\psi}+T*(\psi-\tilde{\psi}))
+(A*\widetilde{Rm}+\nabla S)*(\widetilde{T}+\widetilde{T}*\tilde{\psi}).
\end{align*}
Similarly, the third and fourth terms from \eqref{evl-nabla-T1} yield
\begin{align*}
&\nabla T*(Rm+Rm*\psi)-\widetilde{\nabla}\widetilde{T}*(\widetilde{Rm}+\widetilde{Rm}*\tilde{\psi})\\
&=\nabla T*(S+S*\tilde{\psi}+Rm*(\psi-\tilde{\psi}))+
V*(\widetilde{Rm}+\widetilde{Rm}*\tilde{\psi}).
\end{align*}
and
\begin{align*}
& Rm*T*T*\varphi-\widetilde{Rm}*\widetilde{T}*\widetilde{T}*\tilde{\varphi}\\
&=Rm*T*T*\phi+Rm*U*(T+\tilde{T})*\tilde{\varphi}+S*\widetilde{T}*\widetilde{T}*\tilde{\varphi}.
\end{align*}
We now observe that
\begin{align*}
g^{ab}\nabla_b\nabla_i  T_{jk}-\tilde{g}^{ab}\widetilde{\nabla}_a\widetilde{\nabla}_i\widetilde{T}_{jk}=g^{ab}\nabla_bV_{ijk}+g^{ab}\nabla_b\widetilde{\nabla}_i\widetilde{T}_{jk}-\tilde{g}^{ab}\widetilde{\nabla}_b\widetilde{\nabla}_i\widetilde{T}_{jk}
\end{align*}
and, by virtue of \eqref{g-invers-diff-3}, the last term is given schematically as
\begin{equation}\label{est-V-uniq}
  \mathcal{V}^a_{\,\,\,ijk}=g^{ab}\nabla_b\widetilde{\nabla}_i\widetilde{T}_{jk}-\tilde{g}^{ab}\widetilde{\nabla}_b\widetilde{\nabla}_i\widetilde{T}_{jk}=\big(\tilde{g}^{-1}*h*\widetilde{\nabla}^2\widetilde{T}+A*\widetilde{\nabla}\widetilde{T}\big)^a_{\;\,ijk}.
\end{equation}
Hence, the fifth terms in \eqref{evl-nabla-T1} give
\begin{align*}
&\nabla^2T*T*\varphi-\widetilde{\nabla}^2\widetilde{T}*\widetilde{T}*\tilde{\varphi}\\
&=\nabla^2T*T*\phi+\nabla^2T*U*\tilde{\varphi}
+(\nabla^2T-\widetilde{\nabla}^2\widetilde{T})*\widetilde{T}*\tilde{\varphi}\\
&=\nabla^2T*T*\phi+\nabla^2T*U*\tilde{\varphi}
+(\nabla V+h*\widetilde{\nabla}^2\widetilde{T}+A*\widetilde{\nabla}\widetilde{T})*\widetilde{T}*\tilde{\varphi}
\end{align*}
The sixth terms in \eqref{evl-nabla-T1} yield
\begin{align*}
\nabla T*\nabla T*\varphi-\widetilde{\nabla}\widetilde{T}*\widetilde{\nabla}\widetilde{T}*\tilde{\varphi}=
V*(\nabla T+\widetilde{\nabla}\widetilde{T})
*\tilde{\varphi}+\nabla T*\nabla T*\phi.
\end{align*}
For the remaining terms in \eqref{evl-nabla-T1} we observe that
\begin{align*}
\nabla T*T*T-\widetilde{\nabla}\widetilde{T}*\widetilde{T}*\widetilde{T}
&=V*\widetilde{T}*\widetilde{T}+\nabla T*(T+\widetilde{T})*U.
\end{align*}
Altogether, we find the evolution equation for $V$:
\begin{align*}
  \frac{\pt}{\pt t}V &=\frac{\pt}{\pt t}\nabla T-\frac{\pt}{\pt t}\widetilde{\nabla}\widetilde{T} \displaybreak[0]\\
  &=\Delta V+\nabla_a(g^{ab}\nabla_b\widetilde{\nabla}\widetilde{T}-\tilde{g}^{ab}\widetilde{\nabla}_b\widetilde{\nabla}\widetilde{T})+A*\widetilde{\nabla}^2\widetilde{T}\nonumber\\ &+(A*\widetilde{Rm}+\nabla S)*(\widetilde{T}+\widetilde{T}*\tilde{\psi})+\nabla Rm*(U+U*\tilde{\psi}+T*(\psi-\tilde{\psi}))\nonumber\\
  &+V*(\widetilde{Rm}+\widetilde{Rm}*\tilde{\psi})+\nabla T*(S+S*\tilde{\psi}+Rm*(\psi-\tilde{\psi}))\nonumber\\
  &+S*\widetilde{T}^2*\tilde{\varphi}+Rm*U*(T+\widetilde{T})*\tilde{\varphi}+Rm*T^2*\phi\nonumber\\
  &+(A*\widetilde{\nabla}\widetilde{T}+h*\widetilde{\nabla}^2\widetilde{T}+\nabla V)*\widetilde{T}*\tilde{\varphi}+\nabla^2T*U*\tilde{\varphi}+\nabla^2T*T*\phi\nonumber\\
  &+V*(\widetilde{\nabla}\widetilde{T}+\nabla T)*\tilde{\varphi}+\nabla T*\nabla T*\phi+V*(\widetilde{T}^2+\widetilde{T}^2*\tilde{\psi})\\
  &+\nabla T*T^2*(\tilde{\psi}-\psi)+\nabla T*(T+\widetilde{T})*(U+U*\tilde{\psi}).\nonumber
\end{align*}
We thus obtain \eqref{evl-Vt-0} as claimed.

Finally, we compute the evolution of $S$ using the evolution \eqref{Rm} for $Rm$:
\begin{align*}
    \frac{\pt}{\pt t}S &=  \frac{\pt}{\pt t} Rm-\frac{\pt}{\pt t}\widetilde{Rm}\nonumber\\
     &=\Delta S+\nabla_a(g^{ab}\nabla_b\widetilde{Rm}-\tilde{g}^{ab}\widetilde{\nabla}_b\widetilde{Rm})+A*\widetilde{\nabla}\widetilde{Rm}+S*(Rm+\widetilde{Rm})\nonumber\\
   &\quad+S*T^2+\widetilde{Rm}*U*(T+\widetilde{T})+(A*\widetilde{\nabla}\widetilde{T}+\nabla V)*\widetilde{T} \\&\quad
+\nabla^2T*U+V*(\nabla T+\widetilde{\nabla}\widetilde{T}),\nonumber
\end{align*}
where we used
\begin{equation}\label{est-S-uniq}
  \mathcal{S}_{\;\,ijk}^{a\quad\! l}=g^{ab}\nabla_b\widetilde{Rm}_{ijk}^{\quad l}-\tilde{g}^{ab}\widetilde{\nabla}_b\widetilde{Rm}_{ijk}^{\quad l}=\big(\tilde{g}^{-1}*h*\widetilde{\nabla}\widetilde{Rm}+A*\widetilde{Rm}\big)_{\;\,ijk}^{a\quad l}.
\end{equation}
We thus obtain \eqref{evl-St-0} as required. 
\endproof

We now use Lemma \ref{lem-forward-uniq-1} to obtain a differential inequality for $\mathcal{E}(t)$.

\begin{lem}\label{lem-forward-uniq}
The quantity $\mathcal{E}(t)$ defined by \eqref{energy-def} satisfies
\begin{align*}
  \frac d{dt}\mathcal{E}(t) \leq C\mathcal{E}(t),
\end{align*}
where $C$ is a uniform constant depending only on $K_0$ given in \eqref{Lambda-tLambda-bounds}.
\end{lem}
\proof
Under the curvature and torsion bounds \eqref{curvature-bounds}, the evolution equations  of the metric \eqref{flow-g2} and  volume form \eqref{evl-volumeform}
imply
\begin{equation}
  \biggl|\frac{\pt}{\pt t}g(t)\biggr|_{g(t)}\leq C,\quad \biggl|\frac{\pt}{\pt t}
vol_{g(t)}\biggr|_{g(t)}\leq C.
\end{equation}
For any tensor $P(t)$ we therefore have
\begin{align*}
\frac{d}{dt}\int_M|P(t)|_{g(t)}^2 vol_{g(t)}&=\int_M\frac{\partial }{\partial t}g(t)\big(P(t),P(t)\big)+2\langle P(t),\frac{\partial}{\partial t}P(t)\rangle vol_{g(t)}\\
&\quad+\int_M|P(t)|_{g(t)}^2\frac{\partial}{\partial t}vol_{g(t)}\\
&\leq C\int_M |P(t)|_{g(t)}^2 vol_{g(t)}+2\int_M\langle P(t),\frac{\partial}{\partial t}P(t)\rangle vol_{g(t)}.
\end{align*}
Hence,
\begin{align*}
 \frac d{dt}\mathcal{E}(t) \leq C \mathcal{E}(t)&+2\int_M\biggl(\langle \phi(t),\frac{\pt}{\pt t}\phi(t)\rangle +\langle h(t),\frac{\pt}{\pt t}h(t)\rangle+\langle A(t),\frac{\pt}{\pt t}A(t)\rangle \\
  &+\langle U(t),\frac{\pt}{\pt t}U(t)\rangle+\langle V(t),\frac{\pt}{\pt t}V(t)\rangle+\langle S(t),\frac{\pt}{\pt t}S(t)\rangle\biggr)
vol_{g(t)}.
\end{align*}
We also observe that, by integration by parts, we have
\begin{align*}
\int_M\langle P(t),\Delta P(t)\rangle vol_{g(t)}=-\int_M |\nabla P(t)|_{g(t)}^2 vol_{g(t)}
\end{align*}
and, if $\mathcal{P}(t)$ is another tensor,
\begin{align*}
\int_M\langle P(t),\Div \mathcal{P}(t)\rangle vol_{g(t)}=-\int_M\langle \nabla P(t),\mathcal{P}(t)\rangle vol_{g(t)}.
\end{align*}
Using Lemma \ref{lem-forward-uniq-1}, including the estimates for $\mathcal{V}$ and $\mathcal{S}$, we may calculate: 
\begin{align*}
  \frac d{dt}\mathcal{E}(t) &\leq C\mathcal{E}(t)+C \int_M\biggl(|\phi(t)|^2_{g(t)}+|h(t)|^2_{g(t)}+|A(t)|^2_{g(t)}\\
  &\qquad\qquad\qquad\qquad +|U(t)|^2_{g(t)}+|V(t)|^2_{g(t)}+|S(t)|^2_{g(t)}\biggr)vol_{g(t)}\\
   &\quad-2\int_M(|\nabla S(t)|^2+|\nabla V(t)|^2)
vol_{g(t)} \displaybreak[0]\\
   &\quad+C\int_M|\nabla V(t)|(|h(t)|+|A(t)|+|U(t)|+|V(t)|+|S(t)|)
vol_{g(t)}\\
   &\quad+C\int_M|\nabla S(t)|(|h(t)|+|A(t)|+|V(t)|)
vol_{g(t)}.
\end{align*}
The second term is clearly bounded by $C\mathcal{E}(t)$.  
Now we use the negative third integral in the inequality to crucially cancel
the terms involving $\nabla V$ and $\nabla S$ arising from the fourth and fifth integrals via Young's inequality.  Concretely, for any $\epsilon>0$, we have
\begin{align*}
2&|\nabla V(t)|(|h(t)|+|A(t)|+|U(t)|+|V(t)|+|S(t)|)\\
&\leq \frac{5}{\epsilon}(|h(t)|^2+|A(t)|^2+|U(t)|^2+|V(t)|^2+|S(t)|^2)+\epsilon|\nabla V(t)|^2\\
2&|\nabla S(t)|(|h(t)|+|A(t)|+|V(t)|)\\
&\leq \frac{3}{\epsilon}(|h(t)|^2+|A(t)|^2+|V(t)|^2)+\epsilon|\nabla S(t)|^2,
\end{align*} 
so by choosing $\epsilon$ sufficiently we obtain
\begin{align*}
 \frac{d}{dt}\mathcal{E}(t) &\leq C\mathcal{E}(t)-\int_M(|\nabla S(t)|^2+|\nabla V(t)|^2)vol_{g(t)}\leq C\mathcal{E}(t)
\end{align*}
as claimed.
\endproof

The forward uniqueness property in Theorem \ref{mainthm-uniq} now follows immediately from Lemma \ref{lem-forward-uniq}. If $\varphi(s)=\tilde{\varphi}(s)$ for some $s\in [0,\epsilon]$, then $\mathcal{E}(s)=0$. Thus for $t\in [s,\epsilon]$, we can integrate the differential inequality in Lemma \ref{lem-forward-uniq} to obtain
\begin{equation*}
  \mathcal{E}(t)\leq e^{C(t-s)}\mathcal{E}(s)=0,
\end{equation*}
which implies that $\varphi(t)=\tilde{\varphi}(t)$ for all $t\in [s,\epsilon]$ as required.

\subsection{Backward uniqueness}
To complete the proof of Theorem \ref{mainthm-uniq}, we need to show backward uniqueness of the flow; i.e.~if $\varphi(s)=\tilde{\varphi}(s)$ for some $s\in [0,\epsilon]$, then $\varphi(t)=\tilde{\varphi}(t)$ for all $t\in [0,s]$. To this end, we apply a general backward uniqueness theorem \cite[Theorem 3.1]{kot1} for time-dependent sections of vector bundles satisfying certain differential inequalities. Since we only consider compact manifolds, we state \cite[Theorem 3.1]{kot1} here for this setting.
\begin{thm}\label{thm-kot-backuniq}
Let $M$ be a compact manifold and $g(t), t\in [0,\epsilon]$ be a family of smooth Riemannian metrics on $M$ with Levi-Civita connection $\nabla=\nabla_{g(t)}$. Assume that there exists a positive constant $C$ such that
 \begin{equation*}
  \biggl|\frac{\pt}{\pt t}g(t)\biggr|^2_{g(t)}+\biggl|\nabla\frac{\pt}{\pt t}g(t)\biggr|^2_{g(t)}\leq C,\quad \biggl|\frac{\pt}{\pt t}g^{-1}(t)\biggr|^2_{g(t)}+\biggl|\nabla\frac{\pt}{\pt t}g^{-1}(t)\biggr|^2_{g(t)}\leq C,
\end{equation*}
and that the Ricci curvature of the metric $g(t)$ is bounded below by a uniform constant, i.e.~$Ric(g(t))\geq -Kg(t)$ for some $K\geq 0$. Let $\mathcal{X}$ and $\mathcal{Y}$ be finite direct sums of the bundles $T_l^k(M)$, and $\textbf{X}(t)\in C^{\infty}(\mathcal{X})$, $\textbf{Y}(t)\in C^{\infty}(\mathcal{Y})$, for $t\in[0,\epsilon]$, be smooth families of sections satisfying
\begin{align}
   \left|\left(\frac{\pt}{\pt t}-\Delta_{g(t)}\right) \textbf{X}(t)\right|^2_{g(t)}&\leq C\biggl(|\textbf{X}(t)|^2_{g(t)}+|\nabla \textbf{X}(t)|^2_{g(t)}+|\textbf{Y}(t)|^2_{g(t)}\biggr),\label{evl-Xt}\\
   \left|\frac{\pt}{\pt t}\textbf{Y}(t)\right|^2_{g(t)}&\leq C\biggl(|\textbf{X}(t)|^2_{g(t)}+|\nabla \textbf{X}(t)|^2_{g(t)}+|\textbf{Y}(t)|^2_{g(t)}\biggr)\label{evl-Yt}
\end{align}
for some constant $C\geq 0$, where $\Delta_{g(t)}\textbf{X}(t)=g^{ij}(t)\nabla_i\nabla_j\textbf{X}(t)$ is the Laplacian with respect to $g(t)$ acting on tensors. Then $\textbf{X}(\epsilon)\equiv 0$, $\textbf{Y}(\epsilon)\equiv 0$ implies $\textbf{X}(t)\equiv 0$, $\textbf{Y}(t)\equiv 0$ on $M$ for all $t\in[0,\epsilon]$.
\end{thm}

Suppose $\varphi(s)=\tilde{\varphi}(s)$ for some $s\in[0,\epsilon]$. For our purpose,  we let
\begin{align}
  \textbf{X}(t)&=  U(t)\oplus V(t)\oplus W(t)\oplus S(t)\oplus Q(t), \label{X-def}\\
  \textbf{Y}(t)&=  \phi(t)\oplus h(t)\oplus A(t) \oplus B(t),\label{Y-def}
\end{align}
where $\phi,h,A,U,V,S$ are defined as in \S \ref{sec:for-uniq} and
$$B=\nabla A,\quad W=\nabla^2T-\widetilde{\nabla}^2\widetilde{T},\quad Q=\nabla Rm-\widetilde{\nabla}\widetilde{Rm}.$$
Then
\begin{align*}
  \textbf{X}(t) & \in T_2(M)\oplus T_3(M)\oplus T_4(M)\oplus T_3^1(M)\oplus T_4^1(M) \\
   \textbf{Y}(t) & \in T_3(M)\oplus T_2(M)\oplus T_2^1(M)\oplus T_3^1(M).
\end{align*}
To be able to apply Theorem \ref{thm-kot-backuniq}, we need to show that $\textbf{X}(t)$, $\textbf{Y}(t)$ defined in \eqref{X-def}--\eqref{Y-def} satisfy the system of differential inequalities \eqref{evl-Xt}--\eqref{evl-Yt}.

We begin with the following.
\begin{lem}\label{lem-back-uniq}
The quantities $\phi,h,A,U,V,S,B,W,Q$ defined above are uniformly bounded with respect to $g(t)$ on $M\times [0,\epsilon]$.
\end{lem}
\proof
At the beginning of this section, we argued that the metrics $g(t)$ and $\tilde{g}(t)$ are uniformly equivalent on  $M\times [0,\epsilon]$. We immediately deduce that $|h(t)|_{g(t)}=|g(t)-\tilde{g}(t)|_{g(t)}$ is bounded.
From \eqref{curvature-bounds} and the uniform equivalence of $g(t)$ and $\tilde{g}(t)$, we further have
\begin{gather*}
|V|_{g(t)}=|\nabla T-\widetilde{\nabla}\widetilde{T}|_{g(t)},\quad |S|_{g(t)}=|Rm-\widetilde{Rm}|_{g(t)},\\
 |W|_{g(t)}=|\nabla^2T-\widetilde{\nabla}^2\widetilde{T}|_{g(t)},\quad|Q|_{g(t)}=|\nabla Rm-\widetilde{\nabla}\widetilde{Rm}|_{g(t)}
\end{gather*} are bounded on $M\times [0,\epsilon]$. Recall  $|T|^2_g=-R$, where $R$ is the scalar curvature of $g$. Thus we also have that $|U|_{g(t)}=|T-\tilde{T}|_{g(t)}$ is bounded on $M\times [0,\epsilon]$.

Since $\varphi(s)=\tilde{\varphi}(s)$ for some $s\in [0,\epsilon]$,  we have
\begin{align*}
  |\phi(t)|_{g(t)} &= |\varphi(t)-\tilde{\varphi}(t)|_{g(t)} \\
  &\leq   |\varphi(t)-\varphi(s)|_{g(t)}+|\tilde{\varphi}(s)-\tilde{\varphi}(t)|_{g(t)}\\
  &\leq \biggl|\int_t^s\frac{\pt}{\pt u}\varphi(u)d u\biggr|_{g(t)}+\biggl|\int_t^s\frac{\pt}{\pt u}\tilde{\varphi}(u)d u\biggr|_{g(t)}\\
  &\leq C\biggl|\int_t^{s} |\Delta_{\varphi(u)}\varphi(u)|_{g(u)}+|\Delta_{\tilde{\varphi}(u)}\tilde{\varphi}(u)|_{\tilde{g}(u)} d u\biggr|.
\end{align*}
Since $g(t)$ and $\tilde{g}(t)$ are uniformly equivalent on $M\times [0,\epsilon]$, we know that 
\[
|\Delta_{\tilde{\varphi}(u)}\tilde{\varphi}(u)|_{\tilde{g}(u)}\leq C|\Delta_{\tilde{\varphi}(u)}\tilde{\varphi}(u)|_{g(u)}.
\]
Hence, by virtue of \eqref{T-tau-eq} and \eqref{lap-varphi-0}, we have
\begin{align*}
|\Delta_{\varphi(u)}\varphi(u)|_{g(u)}&+|\Delta_{\tilde{\varphi}(u)}\tilde{\varphi}(u)|_{\tilde{g}(u)}\\
&\leq C(|T(u)|_{g(u)}+|\nabla T(u)|_{g(u)}+|\widetilde{T}(u)|_{g(u)}+|\widetilde{\nabla}\widetilde{T}(u)|_{g(u)}).
\end{align*}
Therefore, by \eqref{curvature-bounds} and the fact that $s,t\in[0,\epsilon]$, there is a uniform constant $C$ depending on $K_1$ such that
\[
 |\phi(t)|_{g(t)} \leq C\epsilon.
\]

Finally, we show $A,B$ are bounded on $M\times [0,\epsilon]$. Since $A(s)=0$, we have
\begin{align}
  |A(t)|_{g(t)} &=|A(t)-A(s)|_{g(t)}\nonumber\\
  &  \leq C\biggl|\int_t^s\bigg|\frac{\pt}{\pt u}A(u)\bigg|_{g(u)}du\biggr|\nonumber\\
  &\leq C\biggl|\int_t^s\biggl|\tilde{g}^{-1} \widetilde{\nabla}(\widetilde{Ric}+\frac 13|\widetilde{T}|^2_{\tilde{g}}\tilde{g}+\widetilde{T}*\widetilde{T})\nonumber\\
  &\qquad\qquad\qquad-g^{-1}\nabla(Ric+\frac 13|T|^2_gg+T*T)\biggr|_{g(u)}\!\!\!\!d u\biggr|.\nonumber
\end{align}
In \eqref{evl-At} we showed that 
\begin{align*}
\tilde{g}^{-1} *\widetilde{\nabla}(\widetilde{Ric}&+\frac 13|\widetilde{T}|^2_{\tilde{g}}\tilde{g}+\widetilde{T}*\widetilde{T})-g^{-1}*\nabla(Ric+\frac 13|T|^2_gg+T*T)
\\
   &=(\tilde{g}^{-1}-g^{-1})*\widetilde{\nabla}\widetilde{Rm}+(\widetilde{\nabla}-\nabla)*\widetilde{Rm}+g^{-1}*\nabla(\widetilde{Rm}-Rm)\\
   &\quad+(\tilde{g}^{-1}-{g}^{-1})*\widetilde{T}*\widetilde{\nabla}\widetilde{T}*\tilde{g}^{-1}+(\widetilde{\nabla}\widetilde{T}-\nabla T)*\widetilde{T}*\tilde{g}^{-1}\\
   &\quad+\nabla T*(\widetilde{T}-T)*\tilde{g}^{-1}+\nabla T*T*(\tilde{g}^{-1}-{g}^{-1}).
\end{align*}
Thus, by the uniform equivalence of $g(t)$ and $\tilde{g}(t)$ and \eqref{curvature-bounds}, we have a uniform constant $C$ depending on $K_1$ such that
\[
|A(t)|_{g(t)}\leq C\epsilon.
\]
 Similarly, we can bound $B=\nabla A$ 
on $M\times [0,\epsilon]$.
\endproof

We derived the evolution equations of $\phi,h,A,U,V,S$ in \S \ref{sec:for-uniq}, so now we compute the evolutions of $B,W,Q$.
\begin{lem}\label{lem-back-uniq-2}
We have the following estimates on the evolution of $B,W,Q$:
\begin{align}\label{evl-Bt}
  \left|\frac{\pt}{\pt t}B(t)\right|^2_{g(t)}  &\leq  C\left(|h(t)|^2_{g(t)}+| A(t)|^2_{g(t)}+|B(t)|^2_{g(t)}+|\nabla Q(t)|^2_{g(t)} \right.\nonumber \\
  & \quad\left.+|U(t)|^2_{g(t)}+|\nabla U(t)|^2_{g(t)}+|\nabla V(t)|^2_{g(t)}+|V(t)|^2_{g(t)}\right);
\end{align}

\vspace{-3ex}

\begin{align}
   \left|\frac{\pt}{\pt t}W(t)-\Delta W(t)\right|^2_{g(t)}&\leq C\Bigl(|A(t)|^2_{g(t)}+|B(t)|^2_{g(t)}+|Q(t)|^2_{g(t)}+|\nabla Q(t)|^2_{g(t)}\nonumber \\
   & \qquad+|\phi(t)|^2_{g(t)}+|U(t)|^2_{g(t)}+|V(t)|^2_{g(t)}\nonumber\\
   &\qquad+|S(t)|^2_{g(t)}+|W(t)|^2_{g(t)}+|\nabla W(t)|^2_{g(t)}\Bigr);\label{evl-Wt}\\
   \left|\frac{\pt}{\pt t}Q(t)-\Delta Q(t)\right|^2_{g(t)}&\leq C\Bigl(|A(t)|^2_{g(t)}+|B(t)|^2_{g(t)}+|Q(t)|^2_{g(t)}+|S(t)|^2_{g(t)} \nonumber \\
   & +|U(t)|^2_{g(t)}+|V(t)|^2_{g(t)}+|W(t)|^2_{g(t)}+|\nabla W(t)|^2_{g(t)}\Bigr).\label{evl-Qt}
\end{align}
\end{lem}
\proof
Since $A$, as a difference of connections,
is a tensor, \eqref{commut-tesor} gives
\begin{align*}
   \frac{\pt}{\pt t}B=\frac{\pt}{\pt t}\nabla A=&\nabla \frac{\pt}{\pt t}A+A*\nabla\frac{\pt}{\pt t}g.
\end{align*}
Since $g$ is uniformly bounded and $\nabla Rm$, $T$ and $\nabla T$ are uniformly bounded in light of \eqref{curvature-bounds}, we immediately deduce 
from the evolution equation \eqref{flow-g2} for $g$ that 
\[
|A(t)*\nabla\frac{\pt}{\pt t}g(t)|_{g(t)}\leq C|A(t)|_{g(t)}.
\]  For the first term we observe from \eqref{g-invers-diff-3} that $\nabla h$ and $\nabla \tilde{g}^{-1}$ are bounded by $C|A|$ as well since $\tilde{g}$ and $\tilde{g}^{-1}$ are uniformly bounded by the uniform equivalence of $\tilde{g}$ and $g$ and Lemma \ref{lem-uniq-1}.  Using these observations together with the uniform boundedness of derivatives of $Rm$, $\widetilde{Rm}$, $T$, $\widetilde{T}$ by \eqref{curvature-bounds}, Lemma \ref{lem-uniq-1} and the boundedness of $A$ by Lemma \ref{lem-back-uniq}, we may apply $\nabla$ to the evolution equation \eqref{evl-At} for $A$ to deduce that 
\begin{align*}
|\nabla \frac{\pt}{\pt t}A(t)|_{g(t)}&\leq C(|h(t)|_{g(t)}+|A(t)|_{g(t)}+|B(t)|_{g(t)}+|\nabla^2 S(t)|_{g(t)}\\
&\qquad+|V(t)|_{g(t)}+|\nabla V(t)|_{g(t)}+|U(t)|_{g(t)}+|\nabla U(t)|_{g(t)}).
\end{align*}
(Note that there is no $\nabla S$ term since $\nabla g=0$.)  
We then observe that 
\begin{align*}
 |\nabla^2 S(t)|^2_{g(t)} &=  |\nabla^2 (Rm(t)-\widetilde{Rm}(t))|^2_{g(t)} \\
  &= |\nabla(\nabla Rm(t)-\widetilde{\nabla}\widetilde{Rm}(t))+(\nabla(\widetilde{\nabla}-\nabla) )\widetilde{Rm}(t)|^2_{g(t)}\\
  &\leq  C\left(|\nabla Q(t)|^2_{g(t)}+| A(t)|^2_{g(t)}+|B(t)|^2_{g(t)}\right).
\end{align*}
Hence,
\begin{align*}
  \left|\frac{\pt}{\pt t}B(t)\right|^2_{g(t)} &\leq  C\left(|h(t)|^2_{g(t)}+| A(t)|^2_{g(t)}+|B(t)|^2_{g(t)}+|\nabla^2 S(t)|^2_{g(t)}\right.\nonumber \\
  & \quad\qquad\left.+|U(t)|^2_{g(t)}+|\nabla U(t)|^2_{g(t)} +|V(t)|^2_{g(t)}+|\nabla V(t)|^2_{g(t)}\right)\nonumber \\
  &\leq  C\left(|h(t)|^2_{g(t)}+| A(t)|^2_{g(t)}+|B(t)|^2_{g(t)}+|\nabla Q(t)|^2_{g(t)}\right.\nonumber \\
  & \quad\qquad\left.+|U(t)|^2_{g(t)}+|\nabla U(t)|^2_{g(t)} +|V(t)|^2_{g(t)}+|\nabla V(t)|^2_{g(t)}\right).
\end{align*}
This gives the inequality \eqref{evl-Bt}.

The inequalities \eqref{evl-Wt} and \eqref{evl-Qt} follow from similar calculations using \eqref{evl-na-Rm1} and \eqref{evl-na^2-T-0}.
\endproof

Recall the elementary inequality
\[
\left|\frac{\pt}{\pt t}V(t)-\Delta V(t)\right|^2_{g(t)}\leq 2\left|\frac{\pt}{\pt t}V(t)-\Delta V(t)-\Div\mathcal{V}(t)\right|^2_{g(t)}+2\left|\Div\mathcal{V}(t)\right|^2_{g(t)}.
\]
By taking the divergence of \eqref{est-V-uniq} and using the uniform boundedness of $\tilde{g}^{-1}$, derivatives of $\widetilde{T}$ and $A$ by Lemmas \ref{lem-uniq-1} and \ref{lem-back-uniq}, we have 
\[
|\Div\mathcal{V}(t)|_{g(t)}\leq C(|h(t)|_{g(t)}+|A(t)|_{g(t)}+|B(t)|_{g(t)}).
\]
We deduce from these observations and the evolution equation \eqref{evl-Vt-0} for $V$ that
\begin{align}
   \Big|&\frac{\pt}{\pt t}V(t)-\Delta V(t)\Big|^2_{g(t)}\leq C\Bigl(|A(t)|^2_{g(t)}+|B(t)|^2_{g(t)}+|S(t)|^2_{g(t)}+|\nabla S(t)|^2_{g(t)} \nonumber \\
   &\qquad+|h(t)|_{g(t)}^2+|\phi(t)|^2_{g(t)}+|U(t)|^2_{g(t)}+|V(t)|^2_{g(t)}+|\nabla V(t)|^2_{g(t)}\Bigr);\label{evl-Vt-2}
   \end{align}
We now observe that by taking the divergence of \eqref{est-S-uniq} we have an  estimate for $\Div\mathcal{S}$:
\[
|\Div\mathcal{S}(t)|_{g(t)}\leq C(|h(t)|_{g(t)}+|A(t)|_{g(t)}+|B(t)|_{g(t)}).
\]
Hence, using the evolution equation \eqref{evl-St-0} for $S$ together with the above estimate, we have:
\begin{align}
   \Big|\frac{\pt}{\pt t}S(t)-\Delta S(t)\Big|^2_{g(t)}&\leq C\Bigl(|A(t)|^2_{g(t)}+|B(t)|^2_{g(t)}+|S(t)|^2_{g(t)} +|h(t)|_{g(t)}^2\nonumber \\
   &\qquad+|U(t)|^2_{g(t)}+|V(t)|^2_{g(t)}+|\nabla V(t)|^2_{g(t)}\Bigr).\label{evl-St-2}
\end{align}

Recall the definition of $\textbf{X}(t)$ and \textbf{Y}(t) in \eqref{X-def}.  We see from Lemma \ref{lem-back-uniq-2}, \eqref{evl-Vt-2} and \eqref{evl-St-2} we have estimates of the form
\[
\left|\frac{\pt}{\pt}P(t)-\Delta P(t)\right|_{g(t)}^2\leq C(|\mathbf{X}(t)|^2_{g(t)}+|\nabla\mathbf{X}(t)|^2_{g(t)}+|\mathbf{Y}(t)|_{g(t)}^2)
\]
for $P=V,W,S,Q$.  Moreover, we have from Lemma \ref{lem-forward-uniq-1} that
\[
\left|\frac{\pt}{\pt t} U(t)\right|_{g(t)}^2\leq C(|\mathbf{X}(t)|^2_{g(t)}+|\nabla\mathbf{X}(t)|^2_{g(t)}+|\mathbf{Y}(t)|_{g(t)}^2),
\]
and we also observe that
\begin{align*}
|\Delta U(t)|^2_{g(t)}&=|\Delta(T(t)-\widetilde{T}(t))|^2_{g(t)}\\
&\leq |\nabla^2(T(t)-\widetilde{T}(t))|^2_{g(t)}\\
&=|\nabla(\nabla T(t)-\widetilde{\nabla}\widetilde{T}(t))+(\nabla(\widetilde{\nabla}-\nabla))\widetilde{T}(t)|_{g(t)}^2\\
&\leq C(|\nabla V|^2_{g(t)}+|A(t)|^2_{g(t)}+|B(t)|^2_{g(t)}).
\end{align*}
Hence, $\textbf{X}(t)$ satisfies \eqref{evl-Xt} in Theorem \ref{thm-kot-backuniq}.
Similarly, from Lemma \ref{lem-forward-uniq-1} and Lemma \ref{lem-back-uniq-2}, we 
have estimates of the form
\[
\left|\frac{\pt}{\pt}P(t)\right|_{g(t)}^2\leq C(|\mathbf{X}(t)|^2_{g(t)}+|\nabla\mathbf{X}(t)|^2_{g(t)}+|\mathbf{Y}(t)|_{g(t)}^2)
\]
for $P=\phi,h,A,B$.  Thus, $\textbf{Y}(t)$ satisfies \eqref{evl-Yt} in Theorem \ref{thm-kot-backuniq}.

Overall, since $M$ is compact and we have the estimates \eqref{curvature-bounds}, we have demonstrated that all of the conditions in Theorem \ref{thm-kot-backuniq} are satisfied.

Hence, if $\varphi(s)=\tilde{\varphi}(s)$ at some time $s\in [0,\epsilon]$, then $\textbf{X}(s)=\textbf{Y}(s)=0$ and thus, by
Theorem \ref{thm-kot-backuniq},  $\textbf{X}(t)=\textbf{Y}(t)=0$ for all $t\in[0,s]$.  This in turn implies $\varphi(t)=\tilde{\varphi}(t)$ for all
$t\in[0,s]$, which is the claimed backward uniqueness property in Theorem \ref{mainthm-uniq}.

\subsection{Applications} We finish this section with two applications of Theorem \ref{mainthm-uniq}; specifically, to the isotropy subgroup of
the $\GG_2$ structure under the flow, and to solitons.

 Let $M$ be a 7-manifold and let $\mathcal{D}$ be the group of diffeomorphisms of $M$ isotopic to the identity.  
 For a $\GG_2$ structure $\varphi$ on $M$, we let $I_{\varphi}$ denote the subgroup of $\mathcal{D}$ fixing $\varphi$.
We now study the behaviour of $I_{\varphi}$ under the Laplacian flow.
\begin{corr}\label{Ivarphi-cor}
Let $\varphi(t)$ be a solution to the Laplacian flow \eqref{Lap-flow-def} on a compact manifold $M$ for $t\in [0,\epsilon]$. Then  $I_{\varphi(t)}= I_{\varphi(0)}$ for all $t\in [0,\epsilon]$.
\end{corr}
\proof
Let $\Psi\in I_{\varphi(0)}$ and $\tilde{\varphi}(t)=\Psi^*\varphi(t)$. Then $\tilde{\varphi}(t)$ is closed for all $t$ and
\begin{align*}
  \frac {\pt}{\pt t} \tilde{\varphi}(t)&=\Psi^*\left(\frac {\pt}{\pt t}\varphi(t)\right)  =\Psi^*\big(\Delta_{\varphi(t)}\varphi(t)\big)
  =\Delta_{\Psi^*\varphi(t)}\Psi^*\varphi(t)=\Delta_{\tilde{\varphi}(t)}\tilde{\varphi}(t),
\end{align*}
so $\tilde{\varphi}(t)$ is also a solution to the flow \eqref{Lap-flow-def}. 
Since $\tilde{\varphi}(0)=\Psi^*\varphi(0)=\varphi(0)$ as $\Psi\in I_{\varphi(0)}$, the forward uniqueness in Theorem \ref{mainthm-uniq} implies that $\tilde{\varphi}(t)=\varphi(t)$ for all $t\in [0,\epsilon]$. Thus, $\Psi\in I_{\varphi(t)}$ for all $t\in [0,\epsilon]$.

Similarly, using the backward uniqueness in Theorem \ref{mainthm-uniq}, we can show if $s\in [0,\epsilon]$ and $\Psi\in I_{\varphi(s)}$, then $\Psi\in I_{\varphi(t)}$ for all $t\in [0,s]$. Therefore, for all $t\in[0,\epsilon]$,
  $I_{\varphi(0)}\subset I_{\varphi(t)}\subset I_{\varphi(0)}$,
 which means $I_{\varphi(t)}= I_{\varphi(0)}$. 
\endproof

\noindent Irreducible compact $\GG_2$ manifolds $(M,\varphi)$ cannot have continuous symmetries and so $I_{\varphi}$ is trivial.  Since the symmetry group $I_{\varphi}$ is not expected to become smaller at an infinite time limit, 
Corollary \ref{Ivarphi-cor}
suggests an immediate test on a closed $\GG_2$ structure $\varphi_0$ 
to determine when the Laplacian flow starting at $\varphi_0$ 
can converge to an irreducible torsion-free $\GG_2$ structure.

We can also use Theorem \ref{mainthm-uniq} in a straightforward way to deduce the following result, which says that any Laplacian flow satisfying the
Laplacian soliton equation at some time must in fact be a Laplacian soliton.
\begin{corr}
Suppose $\varphi(t)$ is a solution to the Laplacian flow \eqref{Lap-flow-def} on a compact manifold $M$ for $t\in [0,\epsilon]$. If for some time $s\in [0,\epsilon]$, $\varphi(s)$ satisfies the Laplacian soliton equation \eqref{solition-def} for some $\lambda\in\R$ and vector field $X$ on $M$, then there exists a family of diffeomorphisms $\phi_t$ and a scaling factor $\rho(t)$ with $\phi_s=id$ and $\rho(s)=1$ such that $\varphi(t)=\rho(t)\phi_t^*\varphi(s)$ on $M\times [0,\epsilon]$.
\end{corr}

\section{Compactness}\label{sec:compact} 

In this section, we  prove a Cheeger--Gromov-type compactness theorem  for solutions to the Laplacian flow for closed $\GG_2$ structures.

\subsection[Compactness for {\boldmath $\GG_2$} structures]{Compactness for \texorpdfstring{$\GG_2$}{G2} structures}

We begin by proving a compactness theorem for the space of $\GG_2$ structures.

Let $M_i$ be a sequence of $7$-manifolds and let $p_i\in M_i$ for each $i$. Suppose that $\varphi_i$ is a $\GG_2$ structure on $M_i$ for each $i$
such that the associated metrics $g_i$ on $M_i$ are complete. Let $M$ be a $7$-manifold with $p\in M$ and let $\varphi$ be a $\GG_2$ structure on $M$. We say that
\begin{equation*}
  (M_i,\varphi_i,p_i)\ra (M,\varphi,p)\quad\textrm{ as }i\ra\infty
\end{equation*}
if there exists a sequence of compact subsets $\Omega_i\subset M$ exhausting $M$ with $p\in int(\Omega_i)$ for each $i$, a sequence of diffeomorphisms $F_i: \Omega_i\ra F_i(\Omega_i)\subset M_i$ with $F_i(p)=p_i$ such that
\begin{equation*}
  F_i^*\varphi_i\ra \varphi\quad\textrm{ as }i\ra \infty,
\end{equation*}
in the sense that $F_i^*\varphi_i-\varphi$ and its covariant derivatives of all orders (with respect to any fixed metric) converge uniformly to zero on every compact subset of $M$.

We may thus give our compactness theorem for $\GG_2$ structures.
\begin{thm}\label{compat-thm-G2}
Let $M_i$ be a sequence of smooth $7$-manifolds and for each $i$ we let $p_i\in M_i$ and $\varphi_i$ be a $\GG_2$ structure on $M_i$
such that the metric $g_i$ on $M_i$ induced by $\varphi_i$  is complete on $M_i$. Suppose that
\begin{equation}\label{compact-thm1-cond1}
  \sup_i\sup_{x\in M_i}\left(|\nabla_{g_i}^{k+1} T_i(x)|_{g_i}^2+|\nabla_{g_i}^kRm_{g_i}(x)|_{g_i}^2\right)^{\frac 12}<\infty
\end{equation}
for all $k\geq 0$ and
\begin{equation*}
  \inf_i \textrm{inj}(M_i,g_i,p_i)>0,
\end{equation*}
where $T_i$, $Rm_{g_i}$ are the torsion and curvature tensor of $\varphi_i$ and $g_i$ respectively, and
$\textrm{inj}(M_i,g_i,p_i)$ denotes the injectivity radius of $(M_i,g_i)$ at $p_i$.

Then there exists a $7$-manifold $M$, a $\GG_2$ structure $\varphi$ on $M$ and a point $p\in M$ such that, after passing to a subsequence, we have
\begin{equation*}
  (M_i,\varphi_i,p_i)\ra (M,\varphi,p)\quad\textrm{ as }i\ra\infty.
\end{equation*}
\end{thm}
\proof In the proof we always use the convention that, after taking a subsequence, we will continue to use the index $i$.

By the Cheeger-Gromov compactness theorem \cite[Theorem 2.3]{ha95-compact} for complete pointed Riemannian manifolds, there exists a complete Riemannian $7$-manifold $(M,g)$ and $p\in M$ such that,  after passing to a subsequence,
\begin{equation}\label{compact-g}
  (M_i,g_i,p_i)\ra (M,g,p)\quad\textrm{ as }i\ra\infty.
\end{equation}
The convergence in \eqref{compact-g} means that, as above, there exist nested compact sets $\Omega_i\subset M$ exhausting $M$ with $p\in int(\Omega_i)$ for all $i$ and diffeomorphisms $F_i:\Omega_i\ra F_i(\Omega_i)\subset M_i$ with $F_i(p)=p_i$ such that $F_i^*g_i\ra g$ smoothly as $i\ra \infty$ on any compact subset of $M$.   

Fix $i$ sufficiently large.
 For $j\geq 0$ we have $\Omega_i\subset \Omega_{i+j}$ and a diffeomorphism $F_{i+j}: \Omega_{i+j} \ra F_{i+j}(\Omega_{i+j})\subset M_{i+j}$. We can then define a
 restricted diffeomorphism
 \begin{equation*}
   F_{i,j}=F_{i+j}|_{\Omega_i}: \Omega_i\ra F_{i+j}(\Omega_i)\subset M_{i+j}\quad\textrm{ for all }j\geq 0.
 \end{equation*}
The convergence \eqref{compact-g} implies that the sequence $\{g_{i,j}=F_{i,j}^*g_{i+j}\}_{j=0}^{\infty}$ of Riemannian metrics on $\Omega_i$ converges to $g_{i,\infty}=g$ on $\Omega_i$ as $j\ra \infty$.

Let $\nabla$, $\nabla_{g_{i,j}}$ be the Levi-Civita connections of $g$, $g_{i,j}$ on $\Omega_i$ respectively. As before, let $h=g-g_{i,j}$ and $A=\nabla-\nabla_{g_{i,j}}$ be the difference of the metrics and their connections, respectively.  It is straightforward to see locally that 
 \begin{equation*}
   A_{ab}^c=\frac 12(g_{i,j})^{cd}\left(\nabla_ah_{bd}+\nabla_bh_{ad}-\nabla_dh_{ab}\right).
 \end{equation*}
 Since $g_{i,j}\ra g$ smoothly on $\Omega_i$ as $j\ra \infty$, $g_{i,j}$ and $g$ are equivalent for sufficiently large $j$, and $|\nabla^kh|_g$ tends to zero as $j\ra\infty$ for all $k\geq 0$. Hence, $A$ is uniformly bounded with respect to $g$ for all large $j$. Moreover,
 \begin{align*}
   &\nabla^{(k)}A_{ab}^c = \frac 12\sum_{l=1}^k\nabla^{(k+1-l)}(g_{i,j})^{cd}\left(\nabla^{(l)}\nabla_ah_{bd}+\nabla^{(l)}\nabla_bh_{ad}-\nabla^{(l)}\nabla_dh_{ab}\right)\\
    &=-\frac 12\sum_{l=1}^k\nabla^{(k+1-l)}(g^{cd}-(g_{i,j})^{cd})\left(\nabla^{(l)}\nabla_ah_{bd}+\nabla^{(l)}\nabla_bh_{ad}-\nabla^{(l)}\nabla_dh_{ab}\right).
 \end{align*}
 Thus there exist constants  $c_k$ for $k\geq 0$ such that $|\nabla^kA|_g\leq c_k$ for all $j\geq 0$.

Using each diffeomorphism $F_{i,j}$, we can define a $\GG_2$ structure $\varphi_{i,j}=F_{i,j}^*\varphi_{i+j}$ on $\Omega_i$
by pulling back the $\GG_2$ structure $\varphi_{i+j}$ on $M_{i+j}$. We next estimate $|\nabla^k\varphi_{i,j}|_g$.  
First, since $g$ and $g_{(i,j)}$ are all equivalent for large $j$, $|\varphi_{i,j}|_g\leq c_0|\varphi_{i,j}|_{g_{i,j}}\leq 7c_0=\tilde{c}_0$ for some
constants $c_0,\tilde{c}_0$.
We next observe trivially that
 \begin{equation*}
   \nabla \varphi_{i,j}=\nabla_{g_{i,j}}\varphi_{i,j}+(\nabla-\nabla_{g_{i,j}})\varphi_{i,j},
 \end{equation*}
 so, since $A$ is uniformly bounded, there is a constant $\tilde{c}_1$ such that
 \begin{align*}
   |\nabla \varphi_{i,j}|_g \leq  c_0|\nabla_{g_{i,j}}\varphi_{i,j}|_{g_{i,j}}+C|A|_g|\varphi_{i,j}|_g \leq \tilde{c}_1.
 \end{align*}
 Similarly, we have
 \begin{align*}
 \nabla^2 \varphi_{i,j}&=\nabla^2_{g_{i,j}}\varphi_{i,j}+(\nabla-\nabla_{g_{i,j}})\nabla_{g_{i,j}}\varphi_{i,j} \\
    & \quad +(\nabla(\nabla-\nabla_{g_{i,j}}))\varphi_{i,j}+(\nabla-\nabla_{g_{i,j}})\nabla \varphi_{i,j},
 \end{align*}
 and so, since $A$, $\nabla A$ are uniformly bounded, there is a constant $\tilde{c}_2$ such that
  \begin{align*}
   |\nabla^2 \varphi_{i,j}|_g &\leq  C|\nabla^2_{g_{i,j}}\varphi_{i,j}|_{g_{i,j}}+C|A|_g|\nabla_{g_{i,j}}\varphi_{i,j}|_g \\
   &\quad +C|\nabla A|_g|\varphi_{i,j}|_g+C|A|_g|\nabla \varphi_{i,j}|_g\leq \tilde{c}_2.
 \end{align*}
For $k\geq 2$, we have the estimate
 \begin{align*}
    |\nabla^k \varphi_{i,j}|_g &\leq  C\sum _{l=0}^k|A|^l_g|\nabla^{(k-l)}_{g_{i,j}}\varphi_{i,j}|_{g_{i,j}}+C\sum_{l=1}^{k-1}|\nabla^lA|_g|\nabla^{(k-1-l)}\varphi_{i,j}|_g.
 \end{align*}
By an inductive argument, using the estimate $|\nabla^kA|_g\leq c_k$ and the assumption \eqref{compact-thm1-cond1}, we can show the existence of constants $\tilde{c}_k$ for $k\geq 0$ such $|\nabla^k\varphi_{i,j}|_g\leq \tilde{c}_k$ on $\Omega_i$ for all $j,k\geq 0$.

The Arzel\`{a}--Ascoli theorem (see, e.g.~\cite[Corollary 9.14]{Ben-hopper}) now implies that there exists a $3$-form $\varphi_{i,\infty}$ and a subsequence of $\varphi_{i,j}$ in $j$, which we still denote by $\varphi_{i,j}$, that converges to $\varphi_{i,\infty}$ smoothly on $\Omega_i$, i.e.
 \begin{equation}\label{comp-pf-5}
   |\nabla^k(\varphi_{i,j}-\varphi_{i,\infty})|_g\ra 0\quad \textrm{ as }j\ra \infty
 \end{equation}
 uniformly on $\Omega_i$ for all $k\geq 0$.

 Since each $\varphi_{i,j}$ is a $\GG_2$ structure on $\Omega_i$ with associated metric $g_{i,j}$, the $7$-form valued bilinear form
 \begin{equation*}
   B_{\varphi_{i,j}}(u,v)=\frac 16 (u\lrcorner \varphi_{i,j})\wedge (v\lrcorner \varphi_{i,j})\wedge \varphi_{i,j}
 \end{equation*}
 is positive definite for each $j$ and satisfies
 \begin{equation}\label{comp-pf-2}
   g_{i,j}(u,v)vol_{g_{i,j}}=B_{\varphi_{i,j}}(u,v),
 \end{equation}
 where $u,v$ are any 
vector fields on $\Omega_i\subset M$. Letting $j\ra \infty$ in \eqref{comp-pf-2} gives
 \begin{equation}\label{comp-pf-3}
     g_{i,\infty}(u,v)vol_{g_{i,\infty}}=B_{\varphi_{i,\infty}}(u,v).
 \end{equation}
 Since  the Cheeger--Gromov compactness theorem guarantees the limit metric $g_{i,\infty}=g$ is a Riemannian metric on $\Omega_i$, \eqref{comp-pf-3} implies that $\varphi_{i,\infty}$ is a positive  $3$-form and hence defines a $\GG_2$ structure on $\Omega_i$ with associated metric $g_{i,\infty}=g$.

We now denote the inclusion map of $\Omega_i$ into $\Omega_k$ for $k\geq i$ by
 \begin{equation*}
   I_{ik}: \Omega_i\ra \Omega_k,\quad\textrm{ for  }k\geq i.
 \end{equation*} 
For each $\Omega_k$, we can argue as before to define $g_{k,j}$, $\varphi_{k,j}$ which converge to $g_{k,\infty}$, $\varphi_{k,\infty}$ respectively as $j\ra\infty$,  after taking a subsequence. By definition,
 \begin{equation*}
   I_{ik}^*g_{k,j}=g_{i,j}\quad\text{and}\quad I_{ik}^*\varphi_{k,j}=\varphi_{i,j}.
 \end{equation*}
 Since $I_{i,k}^*$ is independent of $j$, by taking $j\ra \infty$ here we find that
  \begin{equation}\label{comp-pf-4}
   I_{ik}^*g_{k,\infty}=g_{i,\infty}\quad\text{and}\quad I_{ik}^*\varphi_{k,\infty}=\varphi_{i,\infty}.
 \end{equation}
From \eqref{comp-pf-4}, we see that there exists a $3$-form $\varphi$ on $M$, which is a $\GG_2$ structure with associated metric $g$, such that
 \begin{equation}
   I_{i}^*g=g_{i,\infty},\quad I_i^*\varphi=\varphi_{i,\infty},
 \end{equation}
 where $I_i:\Omega_i\ra M$ is the inclusion map.

 Finally, we show that $(M_i,\varphi_i,p_i)$ converges to $(M,\varphi,p)$.  For any compact subset $\Omega\subset M$, there exists $i_0$ such that $\Omega$ is contained in $\Omega_i$ for all $i\geq i_0$. Fixing $i$ such that $\Omega\subset \Omega_i$,  on $\Omega$ we have by \eqref{comp-pf-5} that
 \begin{align*}
   |\nabla^k(F_l^*\varphi_l-\varphi)|_g &= |\nabla^k(F_{i+j}^*\varphi_{i+j}-\varphi)|_g,\quad \textrm{ where }l=i+j, \\
    &= |\nabla^k(\varphi_{i,j}-\varphi_{i,\infty})|_g \ra 0 \quad\textrm{ as }l\ra \infty
 \end{align*}
for all $k\geq 0$, as required. 
\endproof

\subsection{Compactness for the Laplacian flow}\label{sec:8-3}

Now we can prove Theorem \ref{mainthm-compact}, the compactness theorem for the Laplacian flow for closed $\GG_2$ structures, which we restate here for convenience.

\begin{thm}\label{mainthm-compact-a}
Let $M_i$ be a sequence of compact $7$-manifolds and let $p_i\in M_i$ for each $i$. Suppose that $\varphi_i(t)$ is a sequence of solutions to the Laplacian flow \eqref{Lap-flow-def} for closed $\GG_2$ structures on $M_i$ with the associated metric  $g_i(t)$ on $M_i$ for $t\in (a,b)$, where $-\infty\leq a<0<b\leq \infty$. Suppose further that
\begin{equation}\label{mainthm-compc-cond1-a}
  \sup_i\sup_{x\in M_i,t\in (a,b)}\left(|\nabla_{g_i(t)} T_i(x,t)|_{g_i(t)}^2+|Rm_{g_i(t)}(x,t)|_{g_i(t)}^2\right)^{\frac 12}<\infty,
\end{equation}
where $T_i$ and $Rm_{g_i(t)}$ denote the torsion and curvature tensors determined by $\varphi_i(t)$ respectively, and the injectivity radius of $(M_i,g_i(0))$ at $p_i$ satisfies
\begin{equation}\label{mainthm-compc-cond2-a}
  \inf_i \textrm{inj}(M_i,g_i(0),p_i)>0.
\end{equation}

There exists a $7$-manifold $M$, $p\in M$ and a solution $\varphi(t)$ of the flow \eqref{Lap-flow-def} on $M$ for $t\in (a,b)$ such that, after passing to a subsequence, we have
\begin{equation*}
  (M_i,\varphi_i(t),p_i)\ra (M,\varphi(t),p)\quad\textrm{ as }i\ra\infty.
\end{equation*}
\end{thm}

\noindent The proof is an adaptation of Hamilton's argument  in 
the Ricci flow case \cite{ha95-compact}.

\begin{proof}
By a usual diagonalization argument, without loss of generality, we can assume $a,b$ are finite. From the Shi-type estimates  in $\S$\ref{sec:shi} and \eqref{mainthm-compc-cond1-a}, we have
\begin{equation}\label{comp-pf-6}
  |\nabla_{g_i(t)}^kRm_i(x,t)|_{g_i(t)}+|\nabla_{g_i(t)}^{k+1}T_i(x,t)|_{g_i(t)}\leq {C_k}.
\end{equation} 
Assumption \eqref{mainthm-compc-cond2-a} allows us to apply Theorem \ref{compat-thm-G2} to extract a subsequence of $(M_i,\varphi_i(0),p_i)$ which converges  to a complete limit $(M, \tilde{\varphi}_{\infty}(0),p)$ in the sense described above. Using the notation of Theorem \ref{compat-thm-G2}, we have 
\begin{equation*}
  F_i^*\varphi_i(0)\ra \tilde{\varphi}_{\infty}(0)
\end{equation*}
smoothly on any compact subset $\Omega\subset M$ as $i\ra \infty$. Since each $\varphi_i(0)$ is closed, we  see that $d\tilde{\varphi}_{\infty}(0)=0$.

Let $\tilde{\varphi}_i(t)=F_i^*\varphi_i(t)$.  Fix a compact subset $\Omega\times [c,d]\subset M\times (a,b)$, and let $i$ be sufficiently large that $\Omega\subset \Omega_i$, in the notation of Theorem \ref{compat-thm-G2}. Then $\tilde{\varphi}_i(t)$ is a sequence of solutions of the Laplacian flow on $\Omega\subset M$ defined for $t\in [c,d]$, with associated metrics  $\tilde{g}_i(t)=F_i^*g_i(t)$.  By Claims \ref{claim-5-1} and \ref{claim-5-2}, we may deduce from \eqref{comp-pf-6} that there exist constants $C_{k}$,  independent of $i$, such that
\begin{equation}\label{comp-pf-7}
  \sup_{\Omega\times [c,d]}\left(|\nabla_{\tilde{g}_i(0)}^k\tilde{g}_i(t)|_{\tilde{g}_i(0)}+|\nabla_{\tilde{g}_i(0)}^k\tilde{\varphi}_i(t)|_{\tilde{g}_i(0)}\right)\leq C_{k}.
\end{equation} 
Recall that  $\tilde{\varphi}_i(0)$ and $\tilde{g}_i(0)$ converge to $\tilde{\varphi}_{\infty}(0)$ and $\tilde{g}_{\infty}(0)$ uniformly, with all their covariant derivatives, on $\Omega$. By a similar argument to the proof of Theorem \ref{compat-thm-G2}, we can show from \eqref{comp-pf-7} that there
are constants $\tilde{C}_k$ such that
\begin{equation}\label{comp-pf-8}
  \sup_{\Omega\times [c,d]}\left(|\nabla_{\tilde{g}_{\infty}(0)}^k\tilde{g}_i(t)|_{\tilde{g}_{\infty}(0)}+|\nabla_{\tilde{g}_{\infty}(0)}^k\tilde{\varphi}_i(t)|_{\tilde{g}_{\infty}(0)}\right)\leq \widetilde{C}_{k},
\end{equation}
for sufficiently large $i$, which in turn gives us  constants $\tilde{C}_{k,l}$ such that
\begin{equation}\label{comp-pf-9}
  \sup_{\Omega\times [c,d]}\left(\biggl|\frac{\pt^l}{\pt t^l}\nabla_{\tilde{g}_{\infty}(0)}^k\tilde{g}_i(t)\biggr|_{\tilde{g}_{\infty}(0)}+\biggl|\frac{\pt^l}{\pt t^l}\nabla_{\tilde{g}_{\infty}(0)}^k\tilde{\varphi}_i(t)\biggr|_{\tilde{g}_{\infty}(0)}\right)\leq {C}_{k,l},
\end{equation}
since the time derivatives can be written in terms of spatial derivatives via the Laplacian flow evolution equations. It follows from the Arzel\'{a}--Ascoli theorem that there exists a subsequence of $\tilde{\varphi}_i(t)$ which converges smoothly on $\Omega\times [c,d]$. A diagonalization argument then produces a subsequence that converges smoothly on any compact subset of $M\times (a,b)$ to a solution $\tilde{\varphi}_{\infty}(t)$ of the Laplacian flow.
\end{proof} 

As in Ricci flow, we would want to use our compactness theorem for the Laplacian flow to analyse singularities of the flow as follows.

Let $M$ be a compact $7$-manifold and let $\varphi(t)$ be a solution of the Laplacian flow \eqref{Lap-flow-def} on a maximal time interval $[0,T_0)$ with  $T_0<\infty$. Theorem \ref{mainthm-blowup} implies that $\Lambda(t)$ given in \eqref{Lambda-t-def-2} satisfies $\lim \Lambda(t)=\infty$ as $t\nearrow T_0$. Choose a sequence of points $(x_i,t_i)$ such that $t_i\nearrow T_0$ and
\begin{equation*}
  \Lambda(x_i,t_i)=\sup_{x\in M,\, t\in [0,t_i]}\left(|\nabla T(x,t)|_{g(t)}^2+|Rm(x,t)|_{g(t)}^2\right)^{\frac 12},
\end{equation*}
where $T$ and $Rm$ are the torsion and curvature tensor as usual.

We consider a sequence of parabolic dilations of the Laplacian flow
\begin{equation}\label{phi-i-def}
  \varphi_i(t)=\Lambda(x_i,t_i)^{\frac 32}\varphi(t_i+\Lambda(x_i,t_i)^{-1}t)
\end{equation}
and define
\begin{equation}\label{Lambda-i-def}
  \Lambda_{\varphi_i}(x,t)=\left(|\nabla_{g_i(t)} T_i(x,t)|_{g_i(t)}^2+|Rm_i(x,t)|_{g_i(t)}^2\right)^{\frac 12}.
\end{equation}
From the basic conformal property for $3$-forms we have
\begin{equation*}
  \tilde{\varphi}=\lambda \varphi\quad \Rightarrow \quad \Delta_{\tilde{\varphi}}\tilde{\varphi}=\lambda^{\frac 13}\Delta_{\varphi}\varphi.
\end{equation*}
Thus, for each $i$, $(M,\varphi_i(t))$ is a solution of the Laplacian flow \eqref{Lap-flow-def} on the time interval
\begin{equation*}
  t\in [-t_i\Lambda(x_i,t_i), (T_0-t_i)\Lambda(x_i,t_i))
\end{equation*}
satisfying $\Lambda_{\varphi_i}(x_i,0)=1$ and
\begin{equation*}
  \sup_M|\Lambda_{\varphi_i}(x,t)|\leq 1\quad\textrm{ for }t\leq 0.
\end{equation*}

Since $\sup_M|\Lambda_{\varphi_i}(x,0)|=1$, by the doubling-time estimate (Proposition \ref{prop-Rm-T^2}) and Corollary \ref{exist-time-cor}, there exists a uniform $b>0$ such that
\begin{equation*}
  \sup_M|\Lambda_{\varphi_i}(x,t)|\leq 2\quad\textrm{ for }t\leq b.
\end{equation*}
Therefore, we obtain a sequence of solutions $(M,\varphi_i(t))$ to the Laplacian flow defined on $(a,b)$ for some $a<0$, with
\begin{equation*}
  \sup_i\sup_M|\Lambda_{\varphi_i}(x,t)|< \infty\quad\textrm{ for }t\in(a,b).
\end{equation*}

If we can establish the injectivity radius estimate
\begin{equation*}
  \inf_i inj(M, g_i(0), x_i)>0,
\end{equation*}
which is equivalent to
\begin{equation*}
  \inf_i inj(M, g(t_i), x_i)\geq c\Lambda(x_i,t_i)^{-1},
\end{equation*}
we can apply our compactness theorem (Theorem \ref{mainthm-compact}) and extract a subsequence of $(M,\varphi_i(t), x_i)$ which converges to a limit flow
$({M}_{\infty},{\varphi}_{\infty}(t), {x}_{\infty})$.
Such a blow-up of the flow at the singularity will provide an invaluable tool for
further study of the Laplacian flow.

\section{Long time existence II}\label{sec:longtime-II}

Theorem \ref{mainthm-blowup} states that  the Riemann curvature or the derivative of the torsion tensor must blow-up at a finite singular time of the Laplacian
flow.  However, based on Joyce's existence result for torsion-free $\GG_2$ structures \cite{joyce96-1}, we would
hope to be able to characterise the finite-time singularities of the flow via the blow-up of the torsion tensor itself.

In this section we will show that, under an additional continuity assumption on the metrics along the flow, that the Laplacian flow will exist as long as the torsion tensor remains bounded. From this result, stated below, our improvement Theorem \ref{mainthm-longtime-II} of Theorem \ref{mainthm-blowup} follows as a corollary.

\begin{thm}\label{thm-longtime-II}
Let $M^7$ be a compact manifold and $\varphi(t)$ for $t\in [0,T_0)$, where $T_0<\infty$, be a solution to the Laplacian flow \eqref{Lap-flow-def} for closed $\GG_2$ structures with associated metric $g(t)$ for each $t$. If $g(t)$ is uniformly continuous and the torsion tensor $T(x,t)$ of $\varphi(t)$ satisfies
\begin{equation}\label{thm-9-1-cond}
  \sup_{M\times [0,T_0)}|T(x,t)|_{g(t)}<\infty,
\end{equation}
then the solution $\varphi(t)$ can be extended past time $T_0$.
\end{thm}

Here we say  
$g(t)$ is uniformly continuous if for any $\epsilon>0$ there exists $\delta>0$ such that for any $0\leq t_0<t<T_0$ with $t-t_0\leq \delta$ we have
\begin{equation*}
  |g(t)-g(t_0)|_{g(t_0)}\leq \epsilon,
\end{equation*}
which implies that, as symmetric $2$-tensors, we have
\begin{equation}\label{metric-unif-cont}
  (1-\epsilon)g(t_0)\leq g(t)\leq (1+\epsilon)g(t_0).
\end{equation}

Before we prove Theorem \ref{thm-longtime-II}, we deduce Theorem \ref{mainthm-longtime-II} from Theorem \ref{thm-longtime-II}.
\begin{proof}[Proof of Theorem \ref{mainthm-longtime-II} (assuming Theorem \ref{thm-longtime-II})]
We recall that, for closed $\GG_2$ structures $\varphi$,
\begin{equation*}
  \Delta_{\varphi}\varphi=i_{\varphi}(h),
\end{equation*}
where $h$ is a symmetric $2$-tensor satisfying, in local coordinates,
\begin{equation*} 
h_{ij}=-\nabla_mT_{ni}\varphi_{jmn}-\frac 13|T|^2g_{ij}-T_{il}T_{lj}
\end{equation*}
by \eqref{hodge-Lap-varp-3}.  Moreover, \eqref{evl-volumeform} shows that  the trace of $h$ is equal to
\begin{equation*}
  tr_g(h)=g^{ij}h_{ij}=\frac 23|T|^2_{g}.
\end{equation*}
By  \cite[Proposition 2.9]{Kar},
\begin{equation*}
  |\Delta_{\varphi}\varphi|^2_g=|i_{\varphi}(h)|^2_g=(tr_g(h))^2+2h_i^kh_k^i.
\end{equation*}
Thus, under the assumed bound \eqref{mainthm-9-1-cond} on $\Delta_{\varphi(t)}\varphi(t)$ from Theorem \ref{mainthm-longtime-II},
\begin{equation} 
 \sup_{M\times [0,T_0)}|T(x,t)|_{g(t)}<\infty
\quad\text{and}\quad
\label{pf-mainthm9-1-1}
  \sup_{M\times [0,T_0)}|h(x,t)|_{g(t)}<\infty.
\end{equation}
Along the Laplacian flow \eqref{Lap-flow-def}, the metric $g(t)$ evolves by
\begin{equation*}
  \frac{\pt}{\pt t}g(x,t)=2h(x,t),
\end{equation*}
so it follows from \eqref{pf-mainthm9-1-1} that $g(t)$ is uniformly continuous.  Theorem \ref{thm-longtime-II} then implies that the flow  extends past time $T_0$
as required. 
\end{proof}

Now we give the proof of Theorem \ref{thm-longtime-II}.
\begin{proof}[Proof of Theorem \ref{thm-longtime-II}]
We adapt the argument for an analogous result for Ricci flow in \cite[\S 6.4]{Chow-Lu-Ni-2006}. (Note that Sesum's original proof \cite{sesum2005} of the Ricci flow result used Perelman's noncollapsing theorem, but Lei Ni pointed out that the result can be proved without the noncollapsing theorem.)

Assume, for a contradiction, that the conditions of Theorem \ref{thm-longtime-II} hold but the solution $\varphi(t)$ of the flow cannot be extended pass the time $T_0$.
By the long time existence theorem (Theorem \ref{mainthm-blowup}), there exists a sequence of points and times $(x_i,t_i)$ with $t_i\nearrow T_0$ such that
\begin{equation*}
  \Lambda(x_i,t_i)=\sup_{x\in M,\, t\in [0,t_i]}\left(|\nabla T(x,t)|_{g(t)}^2+|Rm(x,t)|_{g(t)}^2\right)^{\frac 12}\ra \infty.
\end{equation*}
Then arguing as in \S \ref{sec:8-3}, we can define $\varphi_i(t)$  by \eqref{phi-i-def}  and  obtain a sequence of flows $(M, \varphi_i(t), x_i)$ defined on $[-t_i\Lambda(x_i,t_i), 0]$.  Moreover, $\Lambda_{\varphi_i(t)}(x,t)$ given by \eqref{Lambda-i-def} satisfies
\begin{equation*}
  \sup_{M\times [-t_i\Lambda(x_i,t_i), 0]}|\Lambda_{\varphi_i}(x,t)|\leq 1\quad\text{and}\quad |\Lambda_{\varphi_i}(x_i,0)|= 1,
\end{equation*}
 and the associated metric $g_i(t)$ of $\varphi_i(t)$ is
\begin{equation*}
 g_i(t)=\Lambda(x_i,t_i)g(t_i+\Lambda(x_i,t_i)^{-1}t).
\end{equation*}

By assumption, $g(t)$ is uniformly continuous. Let $\epsilon\in(0,\frac{1}{2}]$ and let $\delta>0$ be given 
by the definition of uniform continuity of $g(t)$ so that if $t_0=T_0-\delta$ then \eqref{metric-unif-cont} holds for all $t_0<t<T_0$. Suppose $i$ is sufficiently large that $t_i\geq t_0$. From \eqref{metric-unif-cont}, for any $x,y\in M$ and $t\in [t_0,T_0)$, we have
\begin{equation*}
  (1-\epsilon)^{\frac 12}d_{g(t_0)}(x,y)\leq d_{g(t)}(x,y)\leq (1+\epsilon)^{\frac 12}d_{g(t_0)}(x,y).
\end{equation*}
Therefore, if $B_{g(t)}(x,r)$ denotes the geodesic ball of radius $r$ centred at $x$ with respect to the metric $g(t)$, we have
\begin{equation*}
  B_{g(t)}(x,r)\supset B_{g(t_0)}(x,(1+\epsilon)^{-\frac 12}r).
\end{equation*}
Along the Laplacian flow, the volume form increases, so
\begin{equation*}
  Vol_{g(t)}(B_{g(t)}(x,r))\geq Vol_{g(t_0)}(B_{g(t_0)}(x,(1+\epsilon)^{-\frac 12}r)
\end{equation*}
for any $x\in M$, $r>0$ and $t\in [t_0,T_0)$.    Then, for $x\in M$ and $r\leq \Lambda(x_i,t_i)^{\frac 12}$ we have 
\begin{align*}
  Vol_{g_i(0)}(B_{g_i(0)}(x,r))&=\Lambda(x_i,t_i)^{\frac 72}Vol_{g(t_i)}\big(B_{g(t_i)}(x,\Lambda(x_i,t_i)^{-\frac 12}r)\big)\\
  &\geq \Lambda(x_i,t_i)^{\frac 72}Vol_{g(t_0)}\big(B_{g(t_0)}(x,(1+\epsilon)^{-\frac 12}\Lambda(x_i,t_i)^{-\frac 12}r)\big)\\
  &\geq c(1+\epsilon)^{-\frac 72}r^7,
\end{align*}
for some uniform positive constant $c$. 
Hence we have
\begin{equation}\label{volume-ratio}
  Vol_{g_i(0)}(B_{g_i(0)}(x,r))\geq cr^7
\end{equation}
for all $x\in M$ and $r\in [0, \Lambda(x_i,t_i)^{\frac 12}]$. 

Note that by definition of $\Lambda_{\varphi_i}$ in \eqref{Lambda-i-def} that
\begin{equation*}
  |Rm_{g_i}(x,0)|\leq  \sup_{M\times [-t_i\Lambda(x_i,t_i), 0]}|\Lambda_{\varphi_i}(x,t)|\leq 1
\end{equation*}
on $M$. By the volume ratio bound \eqref{volume-ratio} and \cite[Theorem 5.42]{Chow-Lu-Ni-2006}, we have a uniform injectivity radius estimate $inj(M,g_i(0),x_i)\geq c$ for some constant $c>0$.
 We can thus apply our compactness theorem (Theorem \ref{mainthm-compact}) to obtain a subsequence of $(M, \varphi_i(t), x_i)$ converging to a limit $({M}_{\infty},{\varphi}_{\infty}(t), {x}_{\infty})$, $t\in (-\infty,0]$ with $|\Lambda_{\varphi_{\infty}}(x_{\infty},0)|= 1$.

By the assumption \eqref{thm-9-1-cond} that $T$ remains bounded and $\Lambda(x_i,t_i)\ra \infty$ as $i\ra \infty$, we have
\begin{equation}
  |T_i(x,t)|_{g_i(t)}^2=\Lambda(x_i,t_i)^{-1}|T(x,t_i+\Lambda(x_i,t_i)^{-1}t)|^2_{g(t_i+\Lambda(x_i,t_i)^{-1}t)}\ra 0
\end{equation}
 as $i\ra \infty$. Therefore, $(M_{\infty},\varphi_{\infty}(t))$ has zero torsion for all $t\in (-\infty,0]$. 
Thus $Ric_{g_{\infty}(t)}\equiv 0$ for all $t\in (-\infty,0]$, where $g_{\infty}(t)$ denotes the metric defined by $\varphi_{\infty}(t)$, since
torsion-free $\GG_2$ structures define Ricci-flat metrics.

We can then argue as in \cite{sesum2005} (see also \cite[\S 6.4]{Chow-Lu-Ni-2006}) that $g_{\infty}(0)$ has precisely Euclidean volume growth; i.e.~for all $r>0$, $$Vol_{g_{\infty}(0)}\big(B_{g_{\infty}(0)}(x_{\infty},r)\big)=Vol_{g_{\R^7}}\big(B_{g_{\R^7}}(0,1)\big)r^7.$$ Since a Ricci-flat complete manifold with this property must be isometric to Euclidean space by the Bishop--Gromov relative volume comparison theorem, 
$Rm(g_{\infty}(0))\equiv 0$ on $M_{\infty}$.  This contradicts the fact that
 \begin{equation*}
   |Rm_{g_{\infty}}(x_{\infty},0)|=|\Lambda_{\varphi_{\infty}}(x_{\infty},0)|=1,
 \end{equation*}
 where in the first equality we used the fact that the torsion of $(M_{\infty},\varphi_{\infty}(0))$ vanishes.  We have our required contradiction, so the result follows. 
\end{proof}

\section{Laplacian solitons}\label{sec:solit}

In this section we study what are called soliton solutions of the Laplacian flow.

Given a 7-manifold $M$, a Laplacian soliton of the Laplacian flow \eqref{Lap-flow-def} for closed $\GG_2$ structures on $M$
is a triple $(\varphi,X, \lambda)$ satisfying
\begin{equation}\label{solition-eq1}
  \Delta_{\varphi}\varphi=\lambda\varphi+\mathcal{L}_X\varphi,
\end{equation}
where $d\varphi=0, \lambda\in\R$ and $X$ is a vector field on $M$.
We are interested in $\GG_2$ structures $\varphi$ satisfying \eqref{solition-eq1} as they
naturally give self-similar solutions to the Laplacian flow \eqref{Lap-flow-def}.

Concretely, suppose the initial condition $\varphi_0$ satisfies \eqref{solition-eq1} for some $X$ and $\lambda$.  Define, for all $t$ such that $1+\frac{2}{3}\lambda t>0$,
\begin{equation}\label{solit-scale}
  \rho(t)=(1+\frac {2}3\lambda t)^{\frac 32} \quad\text{and}\quad X(t)=\rho(t)^{-\frac 23}X.
\end{equation}
Let $\phi_t$ be the family of diffeomorphism generated by the vector fields $X(t)$ such that $\phi_0$ is the identity. If we define
\begin{equation}\label{self-simil}
  \varphi(t)=\rho(t)\phi_t^*\varphi_0,
\end{equation}
which only changes by a scaling factor $\rho(t)$ and pullback by a diffeomorphism $\phi_t$ at each time $t$, then
\begin{align*}
  \frac{\pt}{\pt t}\varphi(t) &= \rho'(t)\phi_t^*\varphi_0+ \rho(t)\phi_t^*(\mathcal{L}_{X(t)}\varphi_0)\\
  &= \rho(t)^{\frac 13}\phi_t^*\left(\lambda\varphi_0+\mathcal{L}_X\varphi_0\right)\displaybreak[0]\\
  &= \rho(t)^{\frac 13}\phi_t^*(\Delta_{\varphi_0}\varphi_0)\\
   &= \rho(t)^{\frac 13}(\Delta_{\phi_t^*\varphi_0}\phi_t^*\varphi_0)=\Delta_{\varphi(t)}\varphi(t).
\end{align*}
Hence, $\varphi(t)$ defined in \eqref{self-simil} satisfies the Laplacian flow \eqref{Lap-flow-def} with $\varphi(0)=\varphi_0$.

Based on the formula \eqref{solit-scale} for the scaling factor $\rho(t)$, we say a Laplacian soliton $(\varphi,X, \lambda)$ is expanding if $\lambda>0$; steady if $\lambda=0$; and shrinking if $\lambda<0$. For a closed $\GG_2$ structure $\varphi$ on $M$, we already showed in \eqref{hodge-Lap-varp} that
 \begin{equation}
    \Delta_{\varphi}\varphi=\frac 17|\tau|^2\varphi+\gamma,
 \end{equation}
 where $\gamma\in\Omega^3_{27}(M)$. Therefore, \eqref{solition-eq1} is equivalent to
\begin{equation}\label{solition-eq2}
  (\frac 17|\tau|^2-\lambda)\varphi=-\gamma+\mathcal{L}_X\varphi.
\end{equation}
From this equation we observe that if $X=0$ then since $\gamma\in\Omega^3_{27}$ and $\varphi\in\Omega^3_1$ we must have $\gamma=0$ and
 $\lambda=\frac 1 7|\tau|^2$.  We deduce the following, which is Proposition \ref{prop-soliton}(a).

\begin{prop}\label{prop-soliton-a} A Laplacian soliton of the type $\Delta_{\varphi}\varphi=\lambda\varphi$ must have $\lambda\geq 0$, and $\lambda=0$ if and only if $\varphi$ is torsion-free.
\end{prop}

We now give the proof of Proposition \ref{prop-soliton}(b), which we restate here.
\begin{prop}
The only compact Laplacian solitons of the type $\Delta_{\varphi}\varphi=\lambda\varphi$ are when $\varphi$ is torsion-free.
\end{prop}
\proof
Let $X=0$ in \eqref{solition-eq2}, so
\begin{equation}\label{solition-eq3}
  (\frac 17|\tau|^2-\lambda)\varphi=-\gamma.
\end{equation}
Since the left-hand side of \eqref{solition-eq3} belongs to $\Omega^3_1(M)$ while the right hand side of \eqref{solition-eq3} belongs to $\Omega^3_{27}(M)$, we have
\begin{equation*}
  (\frac 17|\tau|^2-\lambda)\varphi=-\gamma=0.
\end{equation*}
Thus $\lambda=\frac 17|\tau|^2$, which means that
\begin{equation*}
  \qquad d\tau=\Delta_{\varphi}\varphi=\frac 17|\tau|^2\varphi.
\end{equation*}
We can deduce that
\begin{align*}
  \frac 13 d(\tau\wedge \tau\wedge \tau) &= \tau\wedge \tau\wedge d\tau =  \frac 17|\tau|^2\tau\wedge \tau\wedge\varphi\\
  &=-\frac 17|\tau|^2\tau\wedge*_{\varphi}\tau=-\frac 17|\tau|^4*_{\varphi}1,
\end{align*}
where in the third equality we used $\tau\wedge\varphi=-*_{\varphi}\tau$ as $\tau\in\Omega^2_{14}(M)$. Since $M$ is compact, integrating the above equality over $M$ gives that
\begin{equation*}
  0=\frac 13\int_M d(\tau\wedge \tau\wedge \tau) =-\frac 17\int_M|\tau|^4*_{\varphi}1.
\end{equation*}
Thus $\tau=0$ and $\lambda=0$, which means that $\varphi$ is torsion-free.
\endproof

We may call a vector field $X$ such that $\mathcal{L}_X\varphi=0$ 
a \emph{symmetry} of the $\GG_2$ structure $\varphi$. The following lemma shows that the symmetries of a closed $\GG_2$ structure 
correspond to certain Killing vector fields of the associated metric. 
\begin{lem}
Let $\varphi$ be a closed $\GG_2$ structure on a compact manifold $M$ with associated metric $g$ and let $X$ be a vector field on $M$. Then
\begin{equation}\label{L-varphi0}
\mathcal{L}_X\varphi=\frac 12i_{\varphi}\left(\mathcal{L}_Xg\right)+\frac 12 \big(d^*(X\lrcorner\varphi)\big)^{\sharp}\lrcorner\psi,
\end{equation}
where $i_{\varphi}:S^2T^*M\to \Lambda^3T^*M$ is the injective map given in \eqref{ivarphi-def}. In particular, any symmetry $X$ of the closed $\GG_2$ structure $\varphi$ must be a Killing vector field of the associated metric $g$ and satisfy $d^*(X\lrcorner\varphi)= 0$ on $M$.
\end{lem}
\proof
Since $\varphi$ is closed, we have
\begin{equation*}
   \mathcal{L}_X\varphi=d(X\lrcorner\varphi)+X\lrcorner d\varphi=d(X\lrcorner\varphi).
\end{equation*}
Denote $\beta=X\lrcorner\varphi$. Then $\beta_{ij}=X^l\varphi_{lij}$ and
\begin{equation*}
    \mathcal{L}_X\varphi=d\beta=\frac 16(\nabla_i\beta_{jk}-\nabla_j\beta_{ik}-\nabla_k\beta_{ji})dx^i\wedge dx^j\wedge dx^k,
\end{equation*}
i.e., in index notation, we have
\begin{equation}\label{L-varphi-index}
    (\mathcal{L}_X\varphi)_{ijk}=\nabla_i\beta_{jk}-\nabla_j\beta_{ik}-\nabla_k\beta_{ji}.
\end{equation}
We decompose $\mathcal{L}_X\varphi$ into three parts
\begin{align*}
   \mathcal{L}_X\varphi=&\pi_1^3(\mathcal{L}_X\varphi)+\pi_{7}^3(\mathcal{L}_X\varphi)+\pi_{27}^3(\mathcal{L}_X\varphi) =a\varphi+W\lrcorner\psi+i_{\varphi}(\eta),
\end{align*}
where $\pi^k_l:\Omega^k(M)\ra \Omega^k_l(M)$ denotes the projection onto $\Omega^k_l(M)$, $a$ is a function, $W$ is a vector field and $\eta$ is a trace-free symmetric $2$-tensor on $M$.
We now calculate $a$, $W$ and $\eta$, using a similar method to \S \ref{sec:hodge-lap}.

To calculate $a$:
\begin{align*}
  a&=  \frac 17\langle \mathcal{L}_X\varphi,\varphi\rangle =\frac 1{42}(\nabla_i\beta_{jk}-\nabla_j\beta_{ik}-\nabla_k\beta_{ji})\varphi^{ijk} \\
  &=\frac 1{14}\nabla_i\beta_{jk}\varphi^{ijk}=\frac 1{14}\nabla_i(\beta_{jk}\varphi^{ijk})-\frac 1{14}\beta_{jk}\nabla_i\varphi^{ijk}\\
  &=\frac 1{14}\nabla_i(X^l\varphi_{ljk}\varphi^{ijk})-\frac 1{14}X^l\varphi_{ljk}T_{i}^m\psi_m^{\,\,\,\,ijk}\\
  &=\frac 37\nabla_iX_i+\frac 1{28}X^l\varphi_{ljk}\tau_{i}^m\psi_m^{\,\,\,\,ijk}\\
  &=\frac 37\nabla_iX_i+\frac 1{14}X^l\varphi_{ljk}\tau^{jk}=\frac 37 \Div(X),
\end{align*}
where we used \eqref{contr-iden-1}, $ \varphi_{ljk}\tau^{jk}=0$ and $\tau_{i}^m\psi_m^{\,\,\,\,ijk}=2\tau^{jk} $ in \eqref{T-ident-1} since $\tau\in \Omega^2_{14}(M)$ for closed $\GG_2$ structures $\varphi$.

To calculate $W$, using the contraction identities \eqref{contr-iden-1}--\eqref{contr-iden-2},
\begin{align*}
  \big((\mathcal{L}_X\varphi)\lrcorner\psi\big)_l&= (\mathcal{L}_X\varphi)^{ijk}\psi_{ijkl}\\
  &=a\varphi^{ijk}\psi_{ijkl}+W^m\psi_m^{\,\,\,\,ijk}\psi_{ijkl}+(i_{\varphi}(\eta))^{ijk}\psi_{ijkl}\\
  &=-24W_l+(\eta^{im}\varphi_m^{\,\,\,\,jk}-\eta^{jm}\varphi_m^{\,\,\,\,ik}-\eta^{km}\varphi_m^{\,\,\,\,ji})\psi_{ijkl}\\
  &=-24W_l+12\eta^{im}\varphi_{mil}=-24W_l,
\end{align*}
where the last equality follows since $\eta_{im}$ is symmetric in $i,m$ and $\varphi_{mil}$ is skew-symmetric in $i,m$. Using \eqref{L-varphi-index}, we have
\begin{align*}
  W_l =& -\frac 1{24} (\mathcal{L}_X\varphi)^{ijk}\psi_{ijkl}=-\frac 18g^{mi}\nabla_m\beta^{jk}\psi_{ijkl}\\
  = &-\frac 18g^{mi}\nabla_m(\beta^{jk}\psi_{ijkl})+\frac 18\beta^{jk}g^{mi}\nabla_m\psi_{ijkl}\\
  =&-\frac 18 g^{mi}\nabla_m(X^n\varphi_n^{\,\,\,jk}\psi_{ijkl})\\
  &\qquad +\frac 1{16}\beta^{jk}g^{mi}(\tau_{mi}\varphi_{jkl}-\tau_{mj}\varphi_{ikl}-\tau_{mk}\varphi_{jil}-\tau_{ml}\varphi_{jki})\\
  =&-\frac 12 g^{mi}\nabla_m(X^n\varphi_{nil})-\frac 18X^n\varphi_n^{\,\,\,jk}g^{mi}\tau_{mj}\varphi_{ikl}-\frac 1{16}X^n\varphi_n^{\,\,\,jk}g^{mi}\tau_{ml}\varphi_{jki}\\
  =&-\frac 12 g^{mi}\nabla_m(X^n\varphi_{nil}),
\end{align*}
where in the above calculation we used \eqref{contr-iden-2}, \eqref{contr-iden-3}, \eqref{nabla-psi}, \eqref{T-ident-1} and skew-symmetry in the index of $\psi_{ijkl}$. So $$W=\frac 12 \big(d^*(X\lrcorner\varphi)\big)^{\sharp}.$$
If we define the $\GG_2$ curl operator on vector fields by
\begin{equation}\label{curl-eq}
\curl(X)=\big(\!*(d X^{\flat}\wedge\psi)\big)^{\sharp}\quad\text{so}\quad \curl(X)_i=\varphi_{ijk}\nabla^jX^k,
\end{equation}
then in local coordinates
\begin{align*}
  W_l&=  -\frac 12 g^{mi}\nabla_m(X^n\varphi_{nil})=-\frac 12\nabla^iX^n\varphi_{nil}-\frac 12 X^n\nabla^i\varphi_{nil} \\
  &=  \frac 12 \curl(X)_l-\frac 12X^nT_i^{\,\,m}\psi_{mnil}=\frac 12 \curl(X)_l+X^nT_{nl},
\end{align*}
i.e.~the vector field $W$ is
 \begin{equation}\label{W-eq}
   W=\frac 12 \big(d^*(X\lrcorner\varphi)\big)^{\sharp}=\frac 12 \curl(X)+X\lrcorner T.
 \end{equation}

Finally, to calculate $\eta$:
\begin{align}\label{eta-1}
    (\mathcal{L}_X\varphi)_{mni}\varphi_j^{\,\,\,mn}&+(\mathcal{L}_X\varphi)_{mnj}\varphi_i^{\,\,\,mn} \nonumber\\
   &=a\varphi_{mni}\varphi_j^{\,\,\,mn}+W^l\psi_{li}^{\,\,\,\, mn}\varphi_{jmn}+ i_{\varphi}(\eta)_{mni}\varphi_j^{\,\,\,mn}\nonumber\\
   &\qquad +a\varphi_{mnj}\varphi_i^{\,\,\,mn}+W^l\psi_{lj}^{\,\,\,\, mn}\varphi_{imn}+i_{\varphi}(\eta)_{mnj}\varphi_i^{\,\,\,mn}\nonumber\\
   &=12ag_{ij}+8\eta_{ij},
\end{align}
where in the last equation we used the contraction identity \eqref{contr-iden-2} to obtain
\begin{align*}
  W^l\psi_{li}^{\,\,\,\, mn}\varphi_{jmn}+W^l\psi_{lj}^{\,\,\,\, mn}\varphi_{imn}=& 4W^l(\varphi_{jli}+\varphi_{ilj})=0
\end{align*}
and \eqref{contr-iden-3} on the terms involving $\eta$. 
We can calculate the left hand side of \eqref{eta-1} as follows
\begin{align*}
    (\mathcal{L}_X\varphi)_{mni}&\varphi_j^{\,\,\,mn}+(\mathcal{L}_X\varphi)_{mnj}\varphi_i^{\,\,\,mn} \\
   &=(\nabla_m\beta_{ni}-\nabla_n\beta_{mi}-\nabla_i\beta_{nm})\varphi_j^{\,\,\,mn}\\
   &\qquad +(\nabla_m\beta_{nj}-\nabla_n\beta_{mj}-\nabla_j\beta_{nm})\varphi_i^{\,\,\,mn}\\
   &=2(\nabla_m\beta_{ni}\varphi_j^{\,\,\,mn}+\nabla_m\beta_{nj}\varphi_i^{\,\,\,mn})-\nabla_i\beta_{nm}\varphi_j^{\,\,\,mn}-\nabla_j\beta_{nm}\varphi_i^{\,\,\,mn}\\
   &=2\nabla_m(X^l\varphi_{lni}\varphi_j^{\,\,\,mn})-2X^l\varphi_{lni}T_{m}^{\,\,\,\,k}\psi_{kj}^{\,\,\,\,\,\,mn}+2\nabla_m(X^l\varphi_{lnj}\varphi_i^{\,\,\,mn})\\
   &\quad -2X^l\varphi_{lnj}T_{m}^{\,\,\,\,k}\psi_{ki}^{\,\,\,\,\,\,mn}-\nabla_i(X^l\varphi_{lnm}\varphi_j^{\,\,\,mn})+X^l\varphi_{lnm}T_{i}^{\,\,k}\psi_{kj}^{\,\,\,\,\,\,mn}\\
   &\quad -\nabla_j(X^l\varphi_{lnm}\varphi_i^{\,\,\,mn})+X^l\varphi_{lnm}T_{j}^{\,\,k}\psi_{ki}^{\,\,\,\,\,\,mn}\\
   &=2\Div(X)g_{ij}-2\nabla_iX_j+2\nabla_m(X^l\psi_{ilj}^{\quad m})+4X^l\varphi_{lni}T_{j}^{\,\,n}\\
   &\quad +2\Div(X)g_{ij}-2\nabla_jX_i+2\nabla_m(X^l\psi_{jli}^{\quad m})+4X^l\varphi_{lnj}T_{i}^{\,\,n}\\
   &\quad +6\nabla_iX_j-4X^l\varphi_{lkj}T_i^{\,\,k}+6\nabla_jX_i-4X^l\varphi_{lki}T_j^{\,\,k}\\
   &=4\Div(X)g_{ij}+4(\nabla_iX_j+\nabla_jX_i),
\end{align*}
where in the above calculation we again used the equations \eqref{contr-iden-1}--\eqref{contr-iden-3} and \eqref{T-ident-1}. We deduce that
\begin{align*}
  \eta_{ij}=&-\frac 32ag_{ij}+\frac 12 \Div(X)g_{ij}+\frac 12(\nabla_iX_j+\nabla_jX_i)\nonumber\\
  =&-\frac 17 \Div(X)g_{ij}+\frac 12(\mathcal{L}_Xg)_{ij}.
\end{align*}
Then
\begin{align*}
 \mathcal{L}_X\varphi&=a\varphi+W\lrcorner\psi+i_{\varphi}(\eta)=i_{\varphi}(\frac 13ag+\eta)+W\lrcorner\psi\nonumber\\
 &=\frac 12i_{\varphi}(\mathcal{L}_Xg)+\frac 12 \big(d^*(X\lrcorner\varphi)\big)^{\sharp}\lrcorner\psi.
\end{align*}
This proves the formula \eqref{L-varphi0}.

If $X$ is a symmetry of the closed $\GG_2$ structure $\varphi$, i.e.~$\mathcal{L}_X\varphi=0$, then
\begin{align*}
  i_{\varphi}(\frac 12\mathcal{L}_Xg)=&\pi_1^3(\mathcal{L}_X\varphi)+\pi^3_{27}(\mathcal{L}_X\varphi) = 0
\end{align*}
and $\frac 12 (d^*(X\lrcorner\varphi))^{\sharp}\lrcorner\psi=\pi_7^3(\mathcal{L}_X\varphi)=0$. This implies that $\mathcal{L}_Xg=0$ and $d^*(X\lrcorner\varphi)=0$, since $i_{\varphi}$ is an injective operator and $\Omega^3_7(M)\cong \Omega^1(M)$.  
\endproof

We can now derive the condition satisfied by the metric $g$ induced by $\varphi$ when $(\varphi,X,\lambda)$ is a Laplacian soliton, which we expect to have further use.
\begin{prop}\label{prop-soliton-metric}
Let $(\varphi,X,\lambda)$ be a Laplacian soliton
as defined by \eqref{solition-eq1}. Then the associated metric $g$ of $\varphi$ satisfies, in local coordinates,
\begin{equation}\label{soliton-metric}
-R_{ij}-\frac 13|T|^2g_{ij}-2T_i^kT_{kj}=\frac 13\lambda g_{ij}+\frac 12(\mathcal{L}_Xg)_{ij}
\end{equation}
and the vector field $X$ satisfies $d^*(X\lrcorner\varphi)=0$.
\end{prop}
\proof
We know from \S \ref{sec:hodge-lap} that for closed $\GG_2$ structures $\varphi$,
\begin{equation*}
  \Delta_{\varphi}\varphi=i_{\varphi}(h)\in \Omega^3_1(M)\oplus \Omega^3_{27}(M),
\end{equation*}
where $h$ is a symmetric $2$-tensor satisfying
\begin{equation*}
  h_{ij}=-Ric_{ij}-\frac 13|T|^2g_{ij}-2T_i^{\,\,k}T_{kj}.
\end{equation*}
Since $\lambda\varphi\in\Omega^3_1(M)$, from the Laplacian soliton equation \eqref{solition-eq1} we know that
\begin{equation*}
  \mathcal{L}_X\varphi=d(X\lrcorner\varphi)\in\Omega^3_1(M)\oplus \Omega^3_{27}(M).
\end{equation*}
Thus, from \eqref{L-varphi0},  we have
\begin{equation}\label{L-varphi-4}
\mathcal{L}_X\varphi=i_{\varphi}(\frac 12\mathcal{L}_Xg)\quad\text{and}\quad d^*(X\lrcorner\varphi)=0.
\end{equation}
Substituting the first equation of \eqref{L-varphi-4} into the Laplacian soliton equation \eqref{solition-eq1}, and noting that
\begin{equation*}
  \Delta_{\varphi}\varphi=i_{\varphi}(h),\quad \lambda\varphi=i_{\varphi}(\frac 13\lambda g),
\end{equation*}
we get
\begin{equation*}
  i_{\varphi}(h-\frac 13\lambda g-\frac 12\mathcal{L}_Xg)=0.
\end{equation*}
Since $i_{\varphi}$ is injective, the above equation implies that
\begin{equation*}
h-\frac 13\lambda g-\frac 12\mathcal{L}_Xg=0,
\end{equation*}
which is equivalent to \eqref{soliton-metric}.
\endproof

Recall that Ricci solitons $(g,X,\lambda)$ are given by $Ric=\lambda g+\mathcal{L}_Xg$,
so we see that \eqref{soliton-metric} can be viewed as a perturbation of the
Ricci soliton equation using the torsion tensor $T$.  We also re-iterate that the non-existence of compact Laplacian solitons of the form $(\varphi,0,\lambda)$
is somewhat surprising given that we have many compact Ricci solitons of the form $(g,0,\lambda)$ since these correspond to Einstein metrics.

As an application of Proposition \ref{prop-soliton-metric}, we can give a short proof of the main result in \cite{Lin}.
\begin{prop}
\emph{(a)} There are no compact shrinking Laplacian solitons.

\noindent \emph{(b)} The only compact steady Laplacian solitons are given by torsion-free $\GG_2$ structures.
\end{prop}
\proof
Taking the trace of \eqref{soliton-metric}, we have
\begin{equation}\label{T.lambda.X.eq}
  \frac 23|T|^2=\frac 73\lambda+\Div(X).
\end{equation}
When the soliton is defined on a compact manifold $M$, integrating the above equation gives
\begin{equation*}
  \lambda Vol_g(M)=\frac 27\int_M|T|^2 vol_g\geq 0.
\end{equation*}
So $\lambda\geq 0$, and $\lambda=0$ if and only if $T\equiv 0$.
\endproof

\begin{rem} Observe that  \eqref{T.lambda.X.eq} immediately leads to the non-existence of nontrivial steady or shrinking Laplacian solitons with $\Div(X)=0$, thus strengthening Proposition \ref{prop-soliton-a}.
\end{rem}

In Ricci flow, every compact Ricci soliton is a gradient Ricci soliton, meaning that the vector field $X$ in that case satisfies $X=\nabla f$ for some function $f$.
This was proved by Perelman using the $\mathcal{W}$-functional and a logarithmic Sobolev inequality. In the Laplacian
flow the situation is quite different and there is currently no reason to suspect that an analogous result
to the Ricci flow will hold. In fact, we see from \eqref{curl-eq}--\eqref{W-eq} and Proposition \ref{prop-soliton-metric} that if $(\varphi,\nabla f,\lambda)$ is a Laplacian soliton then $\nabla f\lrcorner T=0$.  
It is thus currently an interesting open question whether \emph{any} non-trivial compact Laplacian soliton is a gradient Laplacian soliton.

\section{Concluding remarks}\label{sec:conclusion}

The research in this paper motivates several natural questions that form objectives for future study.  We list some of these problems here.

\begin{enumerate}
\item Show that torsion-free $\GG_2$ structures are dynamically stable under the Laplacian flow.  This has been proved by the authors in
\cite{Lotay-Wei} using the theory developed in this article.
    \item Prove a noncollapsing result along the Laplacian flow for closed $\GG_2$ structures as in Perelman's work \cite{perel} on Ricci flow.  This would mean, in particular,
    that our compactness theory would give rise to well-defined blow-ups at finite-time singularities, which would further allow us to relate singularities of the
flow to Laplacian solitons.
  \item Study the behavior of the torsion tensor near the finite singular time $T_0$ of the Laplacian flow.
Since for closed $\GG_2$ structures $\varphi$, we have $\Delta_{\varphi}\varphi=d\tau$, Theorem \ref{mainthm-longtime-II} says that $d\tau$ will blow up when $t\nearrow T_0$ along the Laplacian flow. The question is whether the torsion tensor $T$, or equivalently $\tau$, will blow up when $t\nearrow T_0$.  Since $|T|^2=-R$, this is entirely analogous to the question in Ricci flow as to whether the scalar curvature will blow up at a finite-time singularity.
This is true for Type-I Ricci flow  on compact manifolds by Enders--M\"uller--Topping \cite{Ender-M-T2010} and K\"{a}hler--Ricci flow by Zhang \cite{zhang}, but it is still open in general and currently forms an active topic of research.
  \item Find some conditions on the torsion tensor under which the Laplacian flow for closed $\GG_2$ structures will exist for all time and converge to a torsion-free $\GG_2$ structure. Based on the work of Joyce \cite{joyce2000}, it is expected that a reasonable condition to impose is that the initial $\GG_2$ structure $\varphi_0$ is closed and has sufficiently small torsion, in a suitable sense.  The Laplacian flow would then provide a parabolic method for
  proving the fundamental existence theory for torsion-free $\GG_2$ structures (c.f.~\cite{joyce2000}).  We can already show that such a result holds in \cite{Lotay-Wei} assuming the work of Joyce,
but it would also be desirable to find a proof only using the flow.
  \item Study the space of gradient Laplacian solitons on a compact manifold.  As mentioned earlier, this would show the similarities or differences
with the analogous theory for Ricci solitons, which it would be instructive to study (see \cite{Cao-surv} for a recent survey on Ricci solitons).
  \item Construct nontrivial examples of Laplacian solitons.  Recent progress on this problem has been made by Bryant \cite{bryant-private}, and also forms
  a topic of current investigation by the authors.
\end{enumerate}

\bibliographystyle{Plain}

\end{document}